\documentclass{article}
\usepackage[utf8]{inputenc}

\usepackage{amsmath}
\usepackage{amsfonts}
\usepackage{amsthm}
\usepackage{amssymb}
\usepackage{mathtools}
\usepackage{microtype}
\DisableLigatures[f]{encoding = *, family = * }
\usepackage[dvipsnames]{xcolor}
\usepackage[colorlinks]{hyperref}
\hypersetup{
	linkcolor = MidnightBlue, citecolor = ForestGreen
}
\usepackage{enumitem}
\usepackage{tikz-cd}

\usepackage{standalone}
\usepackage{tikz}
\usetikzlibrary{hobby,shapes.misc}
\usetikzlibrary{patterns}
\usetikzlibrary{shapes,shadings,positioning}
\usetikzlibrary{decorations.markings}

\theoremstyle{plain}\newtheorem{theorem}{Theorem}[section]
\theoremstyle{plain}
\theoremstyle{plain}\newtheorem{lemma}[theorem]{Lemma}
\theoremstyle{plain}
\theoremstyle{definition}
\theoremstyle{definition}\newtheorem{remark}[theorem]{Remark}

\newtheorem{lemmaA}{Lemma}[section]

\theoremstyle{plain}\newtheorem{corollaryA}[lemmaA]{Corollary}

\theoremstyle{definition}

\def\rest{\hskip 1pt{\hbox to 10.8pt{\hfill
\vrule height 7pt width 0.4pt depth 0pt\hbox{\vrule height 0.4pt
width 7.6pt depth 0pt}\hfill}}}

\definecolor{airforceblue}{rgb}{0.36, 0.54, 0.66}
\definecolor{calpolypomonagreen}{rgb}{0.12, 0.3, 0.17}
\definecolor{brickred}{rgb}{0.8, 0.25, 0.33}

\usepackage[capitalise]{cleveref}
\usepackage{authblk}

\title{ $\Gamma$-Convergence of the Ginzburg-Landau Functional with tangential boundary conditions}
\date{\today}

\begin{document}

\author[1]{Stan Alama\footnote{alama@mcmaster.ca}}
\author[1]{Lia Bronsard\footnote{bronsard@mcmaster.ca }}
\author[1]{Andrew Colinet\footnote{colineta@mcmaster.ca}}
\affil[1]{Department of Mathematics and Statistics, McMaster University, Hamilton, ON L8S 4L8 Canada}

\maketitle

\begin{abstract}
A classical result in the study of Ginzburg-Landau equations is that, for
Dirichlet or Neumann boundary conditions,
if a sequence of functions has energy uniformly bounded on a logarithmic
scale then we can find a subsequence whose Jacobians are convergent in suitable
dual spaces and whose renormalized energy is at least the sum of absolute
degrees of vortices.
However, the corresponding question for the case of tangential or normal
boundary conditions has not been considered.
In addition, the question of convergence of up to the boundary is not very
well understood.
Here, we consider these questions for a bounded, connected, open set of
$\mathbb{R}^{2}$ with $C^{2,1}$ boundary.
\end{abstract}

In this paper we study the Ginzburg-Landau functional in a problem with tangential $u\cdot \nu=0$ (or normal $u\cdot \tau=0$) boundary conditions. These are motivated by experimental results of Volovik and Lavrentovich \cite{VL} in which nematic drops are placed in an isotropic medium, allowing for the control of nematic boundary behaviour. For 3D samples,  Volovik and Lavrentovich found a single interior hedgehog defect when molecules are asked to be normal to the boundary and a bipolar boojum pair when requiring tangential conditions. Inspired by the physical phenomena observed in \cite{VL},   in \cite{ABvB} Alama, Bronsard, and van Brussel studied minimizers for the Ginzburg-Landau energy in a 2D setting with tangential and normal boundary conditions. In particular, they show that minimizers can exhibit half-degree vortices on the boundary. Here we study the full $\Gamma$-convergence in their setting, by extending the work of Jerrard and Soner from
\cite{JS}. In their work, Jerrard and Soner relate, through the framework of $\Gamma$-convergence,
convergence of the Jacobian in the interior
of the domain $\Omega\subseteq\mathbb{R}^{2}$ to lower bounds on energy, and hence to formation of interior defects.
We show that, under appropriate restrictions of the
functions along the boundary, we may extend the convergence of the Jacobian
to hold up to the boundary and hence recover boundary defects as well. 

To be more precise we will provide an explicit statement of our result. We begin by describing our variational problem. 
Note that an extended discussion about notation used in the paper can be
found in Section \ref{Background}. We  let $\Omega\subseteq\mathbb{R}^{2}$ be a bounded connected open set with
$C^{2,1}$-boundary.
In addition, we suppose that $\partial\Omega$ has $b+1$ connected components
where $b\ge0$. We introduce the function spaces
\begin{align*}
    W_{T}^{1,2}(\Omega;\mathbb{R}^{2})
    &\coloneqq\Bigl\{u\in{}W^{1,2}(\Omega;\mathbb{R}^{2}):u_{T}:=u\cdot\tau=0\Bigr\},
    \\
    W_{N}^{1,2}(\Omega;\mathbb{R}^{2})
    &\coloneqq\Bigl\{u\in{}W^{1,2}(\Omega;\mathbb{R}^{2}):u_{N}:=u\vert_{\partial\Omega}-u_{T}=0\Bigr\}.
\end{align*}

With this, we consider the Ginzburg-Landau energy defined on the above function spaces:
\noindent{}For $\varepsilon>0$, we let 
\begin{equation*}
E_{\varepsilon}(u)\coloneqq\int_{\Omega}\!{}
\frac{1}{2}\bigl|\nabla{}u\bigr|^{2}
+\frac{1}{4\varepsilon^{2}}\bigl(|u|^{2}-1\bigr)^{2}.
\end{equation*}

We are ready to state our theorem:

\begin{theorem}\label{ZerothOrderNormal}
\hspace{2pt}\vspace{-2pt}
\begin{enumerate}
    \item\label{Compactness}
        Suppose $\{u_{\varepsilon}\}_{\varepsilon\in(0,1]}\subseteq{}
        W_{T}^{1,2}(\Omega;\mathbb{R}^{2})$ satisfies
        $E_{\varepsilon}(u_{\varepsilon})
        \le{}C\mathopen{}\left|\log(\varepsilon)\right|\mathclose{}$
        for all $\varepsilon\in(0,1]$ and some $C>0$.
        Then, up to a subsequence that we do not relabel,
        we have that there is a signed Radon measure, $J_{*}$, supported on
        $\overline{\Omega}$ such that
        \begin{equation}\label{HolderDualConvergence}
            \lim_{\varepsilon\to0^{+}}\lVert\star{}J(u_{\varepsilon})-J_{*}
            \rVert_{(C^{0,\alpha}(\Omega))^{*}}=0
        \end{equation}
        for all $0<\alpha\le1$.
        In particular, we can express the limit, $J_{*}$, in terms of $M_1$ interior defects and of $M_{2,j}$ boundary defects around the $j$-th connected component of $\partial\Omega$:
        \begin{equation}\label{eq:LimitMeasure}
            J_{*}=\pi\sum_{i=1}^{M_{1}}d_{i}\delta_{a_{i}}
            +\frac{\pi}{2}\sum_{j=0}^{b}\sum\limits_{k=1}^{M_{2,j}}
            d_{jk}\delta_{c_{jk}},
        \end{equation}
        where $d_{i}$ and $d_{jk}$ are non-zero
        integers for $i=1,2,\ldots,M_{1}$, $j=0,1,\ldots,b$,
        and $k=1,2,\ldots,M_{2,j}$,
        $a_{i}\in\Omega$ and $c_{jk}\in(\partial\Omega)_{j}$ for
        $i=1,2,\ldots,M_{1}$, $j=0,1,\ldots,b$, 
        and $k=1,2,\ldots,M_{2,j}$, and
        \begin{equation}\label{TopologicalRestriction}
            \sum_{i=1}^{M_{1}}d_{i}
            +\frac{1}{2}\sum_{j=0}^{b}\sum_{k=1}^{M_{2,j}}d_{jk}=
            \chi_{Euler}(\Omega).
        \end{equation}
        In addition, we have that
        \begin{equation}\label{EvenInteger}
            \frac{1}{2}\sum_{k=1}^{M_{2,j}}d_{jk}\in\mathbb{Z}
        \end{equation}
        for each $j=0,1,\ldots,b$.
    \item\label{LiminfIneq}
        If $\{u_{\varepsilon}\}_{\varepsilon\in(0,1]}\subseteq{}
        W_{T}^{1,2}(\Omega;\mathbb{R}^{2})$ and
        $J_{*}=\pi\sum\limits_{i=1}^{M_{1}}d_{i}\delta_{a_{i}}
        +\dfrac{\pi}{2}\sum\limits_{j=0}^{b}
        \sum\limits_{k=1}^{M_{2,j}}d_{jk}\delta_{c_{jk}}$
        is a signed measure as described below \eqref{eq:LimitMeasure}
        satisfying \eqref{TopologicalRestriction} and
        \eqref{EvenInteger} as well as
        \begin{equation*}
            \lim_{\varepsilon\to0^{+}}\left\|\star{}J(u_{\varepsilon})
            -J_{*}\right\|_{(C^{0,\alpha}(\Omega))^{*}}=0,
        \end{equation*}
        for some $0<\alpha\le1$, then
        \begin{equation*}
            \pi\sum_{i=1}^{M_{1}}|d_{i}|
            +\frac{\pi}{2}\sum_{j=0}^{b}\sum_{k=1}^{M_{2,j}}|d_{jk}|
            \le
            \liminf_{\varepsilon\to0^{+}}\frac{E_{\varepsilon}
            (u_{\varepsilon})}
            {\mathopen{}\left|\log(\varepsilon)\right|\mathclose{}}.
        \end{equation*}
    \item\label{LimsupIneq}
        For each signed measure $J_{*}=\pi\sum\limits_{i=1}^{M_{1}}d_{i}\delta_{a_{i}}
        +\dfrac{\pi}{2}\sum\limits_{j=0}^{b}
        \sum\limits_{k=1}^{M_{2,j}}d_{jk}\delta_{c_{jk}}$ as described below
        \eqref{eq:LimitMeasure}
        satisfying \eqref{TopologicalRestriction} and \eqref{EvenInteger} we
        can find
        $\{u_{\varepsilon}\}_{\varepsilon\in(0,1]}\subseteq{}W_{T}^{1,2}(\Omega;\mathbb{R}^{2})$
        such that
        \begin{equation*}
            \lim_{\varepsilon\to0^{+}}\lVert\star{}J(u_{\varepsilon})
            -J_{*}\rVert_{(C^{0,\alpha}(\Omega))^{*}}
            =0,\,\,\,\forall\,0<\alpha\le1,\hspace{20pt}
            \limsup_{\varepsilon\to0^{+}}\frac{E_{\varepsilon}(u_{\varepsilon})}
            {\mathopen{}\left|\log(\varepsilon)\right|\mathclose{}}
            =\left\|J_{*}\right\|.
        \end{equation*}
\end{enumerate}
\end{theorem}

\begin{remark}
We obtain a corresponding theorem for functions
$u\in{}W_{N}^{1,2}(\Omega;\mathbb{R}^{2})$ since the function, $v$,
defined by
\begin{equation*}
    v(x)\coloneqq(-u_{2}(x),u_{1}(x))
    \in{}W_{T}^{1,2}(\Omega;\mathbb{R}^{2})
\end{equation*}
satisfies
\begin{equation*}
    E_{\varepsilon}(v)=E_{\varepsilon}(u),\hspace{10pt}
    Jv=Ju.
\end{equation*}
\end{remark}

The proof of Theorem \ref{ZerothOrderNormal} makes use of the
reflection technique from \cite{JMS} and the results of \cite{JS}.
However, since we make no restrictions on our domain beyond connectedness
and boundary regularity, and since we permit singularities to develop along the
boundary, we encounter a few additional obstacles.
\begin{enumerate}[label=(\alph*)]
    \item
        In the proof of compactness we needed to demonstrate that
        along each boundary component one half of the sum of the degrees of
        boundary vortices is an integer.
        This is a necessary restriction in order for the boundary vector field
        to be well defined.
        The natural integer to compare this quantity to is the degree near
        the boundary.
        However, since the Jacobian is only known to converge in a suitable
        weak norm, some care is needed in order to facilitate this comparison.
    \item
        In order to relate the degrees of vortices to topological restrictions
        we needed to demonstrate, in a suitable weak sense, convergence of the
        tangential portion of the current, $ju$, along the boundary to the
        signed curvature of the boundary components.
        In order to estimate error terms, which involved the modulus of our
        sequence of functions, we needed to use an adapted slicing
        argument similar to the one found in \cite{JS} to estimate the size
        of sets where the modulus was not close to $1$.
    \item
        In order to find a suitable sequence for the
        proof of the upper bound
        we need to construct, for our setting, the canonical harmonic map.
        Unfortunately, since we permit boundary singularities the construction
        from \cite{BBH} is not directly applicable.
        However, after some adapting, their construction extends to our
        setting.
        Note that this obstacle has been encountered in other papers, for
        instance see \cite{IgKur1}, but the techniques used there are not
        suitable for our setting since we do not impose restrictions on our
        domain beyond connectedness and boundary regularity.
\end{enumerate}

        We note that while there have been many generalizations of the interior convergence
        results for the Ginzburg-Landau energy $E_{\varepsilon}(u)$ including \cite{AlBaOr} for domains in every dimension,
        \cite{IgJ} for two-dimensional compact manifolds,
        and \cite{ABM} for convergence for a modified functional, there
        have not been many results in the Ginzburg-Landau literature
        addressing convergence up to the boundary.
        
        A couple of results that did consider convergence of the Ginzburg-Landau energy $E_{\varepsilon}(u)$ up to the boundary
        were \cite{JMS} and \cite{AlPo}.
        Through a counterexample presented in \cite{JMS} one can see the
        importance of boundary conditions on the functions considered.
        In particular, in the absence of boundary conditions the convergence
        of the Jacobian from \cite{JS} cannot be extended up to the boundary.
        On the other hand, the paper of \cite{AlPo} shows that under the
        assumption of \emph{full Dirichlet} boundary conditions one can obtain
        up to the boundary convergence.
        In particular, \cite{AlPo} goes on to consider first order
        $\Gamma$-convergence of the Ginzburg-Landau functional. A related problem was studied in \cite{IgKur1} in a different regime where boundary vortices are energetically preferable than interior ones and they introduce the notion of global jacobian to obtain a first order
        $\Gamma$-convergence of their functional.

We provide a brief overview of the organization of the paper and the content
of each section.
In Section \ref{Background} we outline notation to be used and some preliminary
concepts.
Here, we provide extended discussions about concepts and notation that will be
used throughout the document.
In particular, we provide exposition regarding a particular tangent-normal
coordinate system that we make use of.
In Section \ref{Lemmas} we provide a number of lemmas needed to prove Theorem
\ref{ZerothOrderNormal}.
In Section \ref{Goal} we provide a proof of Theorem \ref{ZerothOrderNormal}
in three subsections.
Each subsection is dedicated to one of compactness, the lower bound, and the
upper bound.
In addition, we provide an \hyperref[Appendix]{Appendix} for some basic topological results that
we need.

\section*{Acknowledgements}
We would like to thank Prof. Robert Jerrard for his very helpful and fruitful discussions.

\section{Preliminaries}\label{Background}
\subsection{Notation}

For $x=(x_{1},x_{2})\in\mathbb{R}^{2}$ we define
$x^{\perp}\coloneqq(-x_{2},x_{1})$.
For $(x_{1},x_{2}),(y_{1},y_{2})\in\mathbb{R}^{2}$ we also define
\begin{equation*}
    (x_{1},x_{2})\times(y_{1},y_{2})\coloneqq{}x_{1}y_{2}-x_{2}y_{1}.
\end{equation*}
We will use $\Omega\subseteq\mathbb{R}^{2}$ to denote an open bounded set whose
boundary is
of $C^{2,1}$ regularity.
In addition, we will assume that $\partial\Omega$ has $b+1$ connected
components, for $b\ge0$, denoted $(\partial\Omega)_{i}$ for
$i=0,1,\ldots,b$.
We assume that these components are indexed so that $(\partial\Omega)_{0}$
coincides with the outermost boundary component and
$(\partial\Omega)_{i}$ coincides with an interior boundary component for
$i=1,2,\ldots,b$.
For a function $u\colon\Omega\to\mathbb{R}^{2}$,
with well-defined trace, we let $u_{T}$, $u_{N}$ denote, respectively,
the \emph{tangential part} of $u$ along
$\partial\Omega$ and the \emph{normal part} of $u$ along $\partial\Omega$.
These are defined, respectively, by
\begin{equation*}
    u_{T}\coloneqq{}u\cdot\tau,
    \hspace{15pt}u_{N}\coloneqq{}u\vert_{\partial\Omega}-u_{T}
\end{equation*}
where $\tau$ is a locally defined unit tangent vector.
For $\Omega$ as above we introduce the function spaces
\begin{align*}
    W_{T}^{1,2}(\Omega;\mathbb{R}^{2})
    &\coloneqq\biggl\{u\in{}W^{1,2}(\Omega;\mathbb{R}^{2}):u_{T}=0\biggr\},
    \\
    W_{N}^{1,2}(\Omega;\mathbb{R}^{2})
    &\coloneqq\biggl\{u\in{}W^{1,2}(\Omega;\mathbb{R}^{2}):u_{N}=0\biggr\}.
\end{align*}
\noindent{}For appropriate functions $u\colon\Omega\to\mathbb{R}^{2}$ we let
\begin{equation*}
    \nabla\times{}u\coloneqq\frac{\partial{}u^{2}}{\partial{}x_{1}}
    -\frac{\partial{}u^{1}}{\partial{}x_{2}}.
\end{equation*}
In addition, for appropriate functions $\varphi\colon\Omega\to\mathbb{R}$ we let
\begin{equation*}
    \nabla^{\perp}\varphi\coloneqq(\nabla\varphi)^{\perp}.
\end{equation*}

\noindent{}For $\varepsilon>0$ we let the symbol $E_{\varepsilon}(u,\Omega)$
denote the following energy
\begin{equation*}
E_{\varepsilon}(u,\Omega)\coloneqq\int_{\Omega}\!{}e_{\varepsilon}(u),
\hspace{15pt}
e_{\varepsilon}(u)\coloneqq\frac{1}{2}\bigl|\nabla{}u\bigr|^{2}
+\frac{1}{4\varepsilon^{2}}\bigl(|u|^{2}-1\bigr)^{2}.
\end{equation*}
We will also use this notation for more general measurable sets
$A\subseteq\mathbb{R}^{2}$.
We will consider this energy over $W_{T}^{1,2}(\Omega;\mathbb{R}^{2})$ and
$W_{N}^{1,2}(\Omega;\mathbb{R}^{2})$.
We let $C(\Omega)$ denote the space of continuous functions on
$\Omega$ into $\mathbb{R}$ and we pair this with
$\left\|\,\cdot\,\right\|_{L^{\infty}(\Omega)}$.
Next, we let, for $0<\alpha\le1$ and $\varphi\colon\Omega\to\mathbb{R}$,
$[\varphi]_{\alpha}$ denote
\begin{equation*}
    [\varphi]_{\alpha}\coloneqq\sup_{\substack{x,y\in\Omega\\{}x\neq{}y}}
    \biggl\{\frac{|\varphi(x)-\varphi(y)|}{|x-y|^{\alpha}}\bigg\}
\end{equation*}
and we let $C^{0,\alpha}(\Omega)$ denote
\begin{equation*}
    C^{0,\alpha}(\Omega)\coloneqq
    \{\varphi\in{}C(\Omega):[\varphi]_{\alpha}<\infty\}.
\end{equation*}
We pair $C^{0,\alpha}(\Omega)$, for $0<\alpha\le1$, with the norm
\begin{equation*}
    \left\|\,\cdot\,\right\|_{C^{0,\alpha}(\Omega)}\coloneqq
    \max\bigl\{\left\|\,\cdot\,\right\|_{L^{\infty}(\Omega)},
    [\,\cdot\,]_{\alpha}\bigr\}.
\end{equation*}
When we want to refer to functions in $C^{0,\alpha}(\Omega)$
which have compact
support, we use the notation $C_{c}^{0,\alpha}(\Omega)$ for the space
and
$\left\|\,\cdot\,\right\|_{C_{c}^{0,\alpha}(\Omega)}$
for the norm
\begin{equation}\label{CompactSupportNorm}
    \left\|\,\cdot\,\right\|_{C_{c}^{0,\alpha}(\Omega)}
    \coloneqq{}[\,\cdot\,]_{\alpha}.
\end{equation}
Since we will be concerned with boundary behaviour we let, for $r>0$,
$B_{r,+}(0)$ denote
\begin{equation*}
    B_{r,+}(0)\coloneqq\{(x_{1},x_{2})\in\mathbb{R}^{2}:|(x_{1},x_{2})|<r,
    \hspace{5pt}x_{2}>0\}.
\end{equation*}
For $y\in\mathbb{R}$ we also use the notation
\begin{equation*}
    B_{r,+}(y,0)\coloneqq{}(y,0)+B_{r,+}(0).
\end{equation*}
We also introduce, for $0<\alpha\le1$ and $r>0$, the space of functions
\begin{equation*}
    \mathcal{A}_{\alpha,r}\coloneqq
    \bigl\{\varphi\in{}C^{0,\alpha}(B_{r,+}(0)):
    \substack{\varphi\equiv0\text{ in a neighbourhood}\\
    \text{of }\partial{}B_{r,+}(0)\cap\{x_{2}>0\}}\bigr\}
\end{equation*}
which we equip with the norm, $\left\|\,\cdot\,\right\|_{\mathcal{A}_{\alpha,r}}$
defined to be the same as \eqref{CompactSupportNorm}.
For $0<\alpha\le1$ we let $(C^{0,\alpha}(\Omega))^{*}$
be the dual space to $C^{0,\alpha}(\Omega)$ paired with the norm
\begin{equation*}
    \left\|\mu\right\|_{(C^{0,\alpha}(\Omega))^{*}}
    \coloneqq\sup_{\substack{\varphi\in{}C^{0,\alpha}(\Omega)\\
    \left\|\varphi\right\|_{C^{0,\alpha}(\Omega)}\le1}}
    \big\{\bigl<\mu,\varphi\bigr>\bigr\}
\end{equation*}
where $\bigl<\cdot,\cdot\bigr>$ denotes a duality pairing.
For each of $C_{c}^{0,\alpha}(\Omega)$, $\mathcal{A}_{\alpha,r}$, and
$C(\Omega)$ we define respective dual spaces.\\

Given a function $u\in{}W^{1,2}(\Omega;\mathbb{R}^{2})$ we define the
\emph{Jacobian}, $\star{}J(u)$, to be the measure in
$\bigl(C^{0,\alpha}(\Omega)\bigr)^{*}$,
for $0<\alpha\le1$, defined by
\begin{equation*}
    \bigl<\star{}J(u),\varphi\bigr>
    \coloneqq\int_{\Omega}\!{}\varphi{}Ju
\end{equation*}
where $\varphi\in{}C^{0,\alpha}(\Omega)$ and
$Ju\coloneqq\det(\nabla{}u)$.
We use this definition of the Jacobian for $(C(\Omega))^{*}$ as well.
For $u\in{}W^{1,2}(\Omega;\mathbb{R}^{2})$ we also let $ju$ denote
\begin{equation*}
    ju\coloneqq\biggl(u\times{}\frac{\partial{}u}{\partial{}x_{1}},
    u\times{}\frac{\partial{}u}{\partial{}x_{2}}\biggr).
\end{equation*}
Note that since $Ju=\frac{1}{2}\nabla\times{}ju$ then, for
$\varphi\in{}C^{0,1}(\Omega)$, integrating by parts gives:
\begin{equation*}
    \int_{\Omega}\!{}\varphi{}Ju
    =-\frac{1}{2}\int_{\Omega}\!{}\nabla^{\perp}\varphi\cdot{}ju
    +\frac{1}{2}\int_{\partial\Omega}\!{}(ju\cdot\tau)\varphi
\end{equation*}
where $\tau(x)$ denotes, for each $x\in\partial\Omega$, the unit tangent vector at $x$
such that $\{\mathbf{n}(x),\tau(x)\}$ is a positively oriented basis where $\mathbf{n}(x)$
is the outward unit normal to $\partial\Omega$ at $x$.
We provide a detailed discussion of notation for tangential and normal unit vectors in
Section \ref{TangentNormalCoordinates}.

\subsection{Coordinates}\label{TangentNormalCoordinates}
\par{}Here we will construct a coordinate system for a neighbourhood of
$\partial\Omega$ which uses the unit tangent and unit inward normal.
Our construction will follow the ideas found in Section $2$ of \cite{BaTi}.
See also \cite{RaIg} for similar ideas.\\

To construct the desired coordinate system we will need to consider each of the
connected components of $\partial\Omega$ separately.
We parametrize each $(\partial\Omega)_{i}$, for $i=0,1,\ldots,b$, by its
arclength, $L_{i}$, using a $C^{2,1}$ curve
$\gamma_{i}(y_{1})=(\gamma_{i,1}(y_{1}),\gamma_{i,2}(y_{1}))$
where $\gamma_{i}\colon\mathbb{R}\slash{}L_{i}\mathbb{Z}\to(\partial\Omega)_{i}$.
We define the unit tangent and unit normal vectors, $\tau_{i}$ and
$\nu_{i}$, respectively, as
\begin{align*}
    \tau_{i}(y_{1})&\coloneqq(\gamma_{i,1}'(y_{1}),\gamma_{i,2}'(y_{1})),
    \\
    \nu_{i}(y_{1})&\coloneqq(-\gamma_{i,2}'(y_{1}),\gamma_{i,1}'(y_{1})).
\end{align*}
We also let, for $i=0,1,\ldots,b$,
$\mathbf{n}_{i}\coloneqq-\nu_{i}$ denote the outward unit normal along
$(\partial\Omega)_{i}$.
By perhaps reversing the orientation on $\gamma_{i}$ we may ensure that
$\nu_{i}$ coincides with the inward unit normal to $(\partial\Omega)_{i}$ for each
$i=0,1,\ldots,b$.
Notice that this choice of orientation matches the induced orientation from $\Omega$ which is clockwise
on $(\partial\Omega)_{i}$ for $i=1,2,\ldots,b$ and counterclockwise on $(\partial\Omega)_{0}$.
\begin{center}
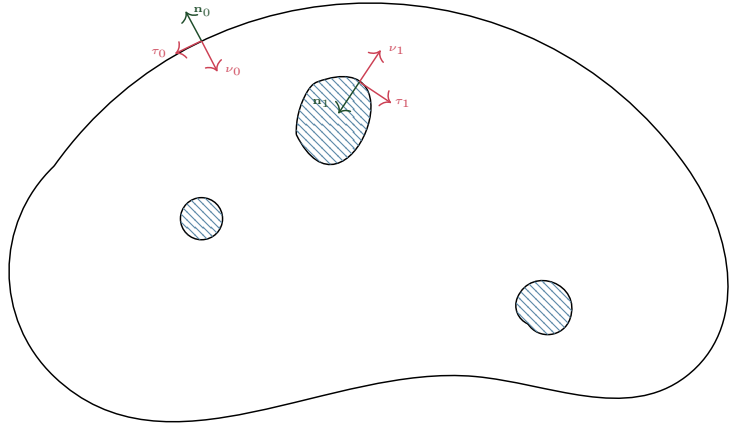
\begin{figure}
\begin{center}
\scalebox{1.4}{
\begin{tikzpicture}[use Hobby shortcut]
\draw[rotate=-90,scale=2]
(3,0) .. +(1,0) .. +(1,2) .. +(1,3) .. +(0,3) .. (3,0); 
\draw[->,brickred](1.4,-4.81) .. (1.55,-5.1);
\draw[->,brickred](1.4,-4.81) .. (1.15,-4.93045454545);
\draw[->,calpolypomonagreen](1.4,-4.81) .. (1.25,-4.53133333333);
\draw(1.55,-5.1)node[right,brickred,scale=0.6]{\tiny $\nu_{0}$};
\draw(1.15,-4.93045454545)node[left,brickred,scale=0.6]{\tiny $\tau_{0}$};
\draw(1.25,-4.53133333333)node[right,calpolypomonagreen,scale=0.6]{\tiny $\mathbf{n}_{0}$};
\begin{scope}
\clip[draw] (2.3,-5.7) .. (2.5,-5.2) .. (2.4,-5.3) .. (2.3,-5.7);
\fill[pattern=north west lines, pattern color=airforceblue](2.5,-5.5) circle(1);
\end{scope}
\draw[->,brickred](2.9,-5.2) .. (3.1,-4.9);
\draw[->,brickred](2.9,-5.2) .. (3.2,-5.4);
\draw[->,calpolypomonagreen](2.9,-5.2) .. (2.7,-5.5);
\draw(3.1,-4.9)node[right,brickred,scale=0.6]{\tiny $\nu_{1}$};
\draw(3.16,-5.4)node[right,brickred,scale=0.6]{\tiny $\tau_{1}$};
\draw(2.7,-5.4)node[left,calpolypomonagreen,scale=0.6]{\tiny $\mathbf{n}_{1}$};
\begin{scope}
\clip[draw] (4.5,-7.5) .. (4.9,-7.45) .. (4.4,-7.25) .. (4.5,-7.5);
\fill[pattern=north west lines, pattern color=airforceblue](4.5,-7.5) circle(1);
\end{scope}
\begin{scope}
\clip[draw] (1.4,-6.5) coordinate (M) circle[radius=2mm];
\fill[pattern=north west lines, pattern color=airforceblue](1.4,-6.5) circle(1);
\end{scope}
\end{tikzpicture}
}
\caption{Depiction of tangent and normal vectors.}
\end{center}
\end{figure}
\end{center}
Recall Chapter $2.2$ of \cite{Pr} gives
\begin{equation}\label{DerivativeIdentities}
    \tau_{i}'(y_{1})=\kappa_{i}(y_{1})\nu_{i}(y_{1}),\hspace{15pt}
    \nu_{i}'(y_{1})=-\kappa_{i}(y_{1})\tau_{i}(y_{1}),
\end{equation}
where $\kappa_{i}\colon\mathbb{R}\slash{}L_{i}\mathbb{Z}\to\mathbb{R}$, for
$i=0,1,\ldots,b$ is the signed curvature of $(\partial\Omega)_{i}$.
Notice that $\kappa_{i}\in{}C^{0,1}(\mathbb{R}\slash{}L_{i}\mathbb{Z})$
since $\partial\Omega$ is $C^{2,1}$.

Next we will define a coordinate chart on $\Omega$ which extends to a larger domain.
Suppose $0<r\le{}r_{0}<\frac{1}{2}\text{inj}(\partial\Omega)$, where
$\text{inj}(\partial\Omega)$ denotes the injectivity radius of
$\partial\Omega$, and set
\begin{align*}
    \Omega_{i,r}&\coloneqq
    \{x\in\Omega:0<\text{dist}(x,(\partial\Omega)_{i})<r\},\\
    \overline{\Omega}_{i,r}&\coloneqq
    \{x\in\overline{\Omega}:0\le\text{dist}(x,(\partial\Omega)_{i})\le{}r\}.
\end{align*}
We define, for $i=0,1,\ldots,b$, $C^{1,1}$ maps
$X_{i}\colon(\mathbb{R}\slash{}L_{i}\mathbb{Z})\times(0,r_{0})\to
\Omega_{i,r_{0}}$
by
\begin{equation}\label{TangentNormal:Def}
    X_{i}(y_{1},y_{2})\coloneqq\gamma_{i}(y_{1})+y_{2}\nu_{i}(y_{1}).
\end{equation}
We observe that by \eqref{DerivativeIdentities} we have
\begin{equation}\label{PartialsOfX}
    \frac{\partial{}X_{i}}{\partial{}y_{1}}
    =\tau_{i}(y_{1})-y_{2}\kappa_{i}(y_{1})\tau_{i}(y_{1}),\hspace{15pt}
    \frac{\partial{}X_{i}}{\partial{}y_{2}}
    =\nu_{i}(y_{1}).
\end{equation}
From \eqref{PartialsOfX} we have
\begin{equation*}
    JX_{i}=1-y_{2}\kappa_{i}(y_{1}).
\end{equation*}
By perhaps shrinking $r_{0}$ we may ensure $JX_{i}$ is bounded away from zero on
$(\mathbb{R}\slash{}L_{i}\mathbb{Z})\times(0,r_{0})$.
Next we let
\begin{equation*}
    r_{1}\coloneqq\frac{1}{2}\min\Bigl\{\min_{i=0}^{b}\{L_{i}\},r_{0}\Bigr\}
\end{equation*}
and notice that for each $y\in(\mathbb{R}\slash{}L_{i}\mathbb{Z})$
we can define a map $i_{y}\colon{}B_{r_{1},+}(0)\to{}B_{r_{1},+}(y,0)$ as
\begin{equation*}
    i_{y}(z)\coloneqq{}(y,0)+z.
\end{equation*}
Since $(\mathbb{R}\slash{}L_{i}\mathbb{Z})\times[0,r_{1}\slash2]$
is compact for each $i=0,1,2,\ldots,b$ then we can find, for each $i=0,1,\ldots,b$,
finitely many points
$\{y_{i,j}\}_{j=1}^{N_{i}}\subseteq(\mathbb{R}\slash{}L_{i}\mathbb{Z})$
such that
\begin{equation*}
    (\mathbb{R}\slash{}L_{i}\mathbb{Z})\times[0,r_{1}\slash2]
    \subseteq\bigcup_{j=1}^{N_{i}}B_{r_{1},+}(y_{i,j},0).
\end{equation*}
Next we define, for each $i=0,1,\ldots,b$ and $j=1,2,\ldots,N_{i}$, the
$C^{1,1}$ map $\psi_{i,j}\colon{}B_{r_{1},+}(0)\to\mathcal{U}_{i,j}$ by
\begin{equation*}
    \psi_{i,j}\coloneqq{}X_{i}\circ{}i_{y_{i,j}}
\end{equation*}
where $\mathcal{U}_{i,j}\coloneqq{}X_{i}(B_{r_{1},+}(y_{i,j},0))$.
Notice that, for each $i=0,1,\ldots,b$ and $j=1,2,\ldots,N_{i}$,
$J\psi_{i,j}$ is Lipschitz and bounded away from zero by our choice of
$r_{0}$, $\psi_{i,j}$ is proper (i.e. preimages of compact sets are
compact), and $\mathcal{U}_{i,j}$ is simply
connected.
By Theorem $6.2.8$ of \cite{KrPa} or Theorem $B$ of \cite{Go} we conclude that
$\psi_{i,j}$ is invertible with differentiable inverse.
It follows from
\begin{equation*}
    D\psi_{i,j}^{-1}(x)=\bigl[D\psi_{i,j}(\psi_{i,j}^{-1}(x))\bigr]^{-1}
\end{equation*}
and that $J\psi_{i,j}^{-1}$ is bounded away from zero, that $D\psi_{i,j}^{-1}$ has
Lipschitz components.
For each $i=0,1,\ldots,b$ we now obtain a local chart for
$\overline{\Omega}_{i,\frac{r_{1}}{2}}\setminus(\partial\Omega)_{i}$ from
$\{(\mathcal{U}_{i,j},\psi_{i,j})\}_{j=1}^{N_{i}}$.
Collecting all of these local charts gives an atlas for
$\{x\in\Omega:\text{dist}(x,\partial\Omega)\le\frac{r_{1}}{2}\}$.
Observe that we may adjoin
$\mathcal{U}_{0,0}\coloneqq\{x\in\Omega:\text{dist}(x,\partial\Omega)
>\frac{r_{1}}{4}\}$
paired with the identity map $\psi_{0,0}$ to obtain an atlas for $\Omega$.
We also notice that since the parametrizing curves $\gamma_{i}$ were not
reparametrized in a way that distorts length, we retain that
$(\partial\Omega)_{i}$, for $i=0,1,\ldots,b$, is still described
by an arclength parametrized curve in $\mathcal{U}_{i,j}$.
We observe that this atlas extends to an atlas for $\overline{\Omega}$
with similar properties.

Next, we let $\widetilde{\Omega}_{i,r}$ denote, for $i=0,1,2,\ldots,b$ and
$0<r\le{}r_{0}$,
\begin{equation}\label{def:DomainiExt}
    \widetilde{\Omega}_{i,r}\coloneqq
    \{x\in\mathbb{R}^{2}:0\le\text{dist}(x,(\partial\Omega)_{i})<r\}.
\end{equation}
Notice that we can extend $X_{i}$, for $i=0,1,\ldots,b$, to a map
$\widetilde{X}_{i}\coloneqq(\mathbb{R}\slash{}L_{i})\times(-r_{0},r_{0})
\to\widetilde{\Omega}_{i,r_{0}}$ by the same definition as in
\eqref{TangentNormal:Def}.
This allows us to view each $\psi_{i,j}$, for $i=0,1,\ldots,b$ and
$j=1,2,\ldots,N_{i}$, as the restriction to $B_{r_{1},+}(0)$ of a map
$\widetilde{\psi}_{i,j}\colon{}B_{r_{1}}(0)\to\widetilde{\mathcal{U}}_{i,j}$,
where $\widetilde{\mathcal{U}}_{i,j}
\coloneqq\widetilde{X}_{i}(B_{r_{1}}(y_{i,j},0))$.
By definition of the extension of $\widetilde{X}_{i}$ and the chart
$\widetilde{\psi}_{i,j}$ we
have, for $(y_{1},s)\in\psi_{i,j}^{-1}(\mathcal{U}_{i,j})$, that
\begin{equation*}
    \text{dist}(\widetilde{\psi}_{i,j}(y_{1},-s),(\partial\Omega)_{i})
    =\text{dist}(\widetilde{\psi}_{i,j}(y_{1},s),(\partial\Omega)_{i}).
\end{equation*}
Let
$\widetilde{\mathcal{U}}_{0,0}
\coloneqq\{x\in\Omega:\text{dist}(x,\partial\Omega)>
\frac{r_{1}}{4}\}$.
We may pair $\widetilde{\mathcal{U}}_{0,0}$ with the identity map,
$\widetilde{\psi}_{0,0}(y)\coloneqq{}y$, in order to extend the atlas for
$\bigcup_{i=0}^{b}\widetilde{\Omega}_{i,\frac{r_{1}}{2}}$ to one for $\Omega$.
We let $\{\widetilde{\rho}_{i,j}\}_{i=0,1,\ldots,b}^{j=1,2,\ldots,N_{i}}\cup
\{\widetilde{\rho}_{0,0}\}$
be a smooth partition of unity subordinate to the cover
$\{\widetilde{\mathcal{U}}_{i,j}\}_{i=0,1,\ldots,b}^{j=1,2,\ldots,N_{i}}
\cup\{\widetilde{\mathcal{U}}_{0,0}\}$.
We observe that since, for each $i_{1}=0,1,\ldots,b$, 
$(\partial\Omega)_{i_{1}}\cap\widetilde{\mathcal{U}}_{i_{2},j_{2}}=\varnothing$
for $i_{1}\neq{}i_{2}$, $i_{2}=0,1,\ldots,b$ and $j_{2}=1,2,\ldots,N_{i_{2}}$
and since $\widetilde{\mathcal{U}}_{0,0}\subset\subset\Omega$ then
\begin{equation}\label{DisjointPartitionOfUnity}
    \sum_{j_{1}=1}^{N_{i_{1}}}\widetilde{\rho}_{i_{1},j_{1}}=1\hspace{15pt}
    \text{on }(\partial\Omega)_{i_{1}},\,i_{1}=0,1,\ldots,b.
\end{equation}
Next, we notice that, due to the choice of orientation about
$(\partial\Omega)_{i}$ for
$i=0,1,\ldots,b$, we have, by Lemma \ref{FilledInDomain}
and the Gauss-Bonnet theorem, that
\begin{equation}\label{eq:OrientedBoundary}
    \int_{(\partial\Omega)_{0}}\!{}\widetilde{\kappa}_{0}=2\pi,\hspace{15pt}
    \int_{(\partial\Omega)_{i}}\!{}\widetilde{\kappa}_{i}=-2\pi,
    \,\,\text{for }i=1,2,\ldots,b,
\end{equation}
where $\widetilde{\kappa}_{i}\coloneqq
\kappa_{i}\circ(\widetilde{\psi}_{i,j}^{-1})^{1}$.
By Corollary \ref{EulerCharacteristic} or the Gauss-Bonnet theorem we now
have
\begin{equation*}
    \int_{\partial\Omega}\!{}\widetilde{\kappa}=2\pi\chi_{Euler}(\Omega).
\end{equation*}
We introduce some additional notation for expressing and extending functions
$u_{\varepsilon}$ using the previously described coordinates on $\Omega$:
\begin{align}
    z_{i,j,\varepsilon}&\coloneqq{}u_{\varepsilon}\circ\psi_{i,j},
    \label{def:zRep}\\
    w_{i,j,\varepsilon}&\coloneqq{}
    z_{i,j,\varepsilon,\tau}e_{1}
    +z_{i,j,\varepsilon,\nu}e_{2},
    \label{def:wRep}
\end{align}
where $z_{i,j,\varepsilon,\tau}$ and $z_{i,j,\varepsilon,\nu}$ are,
respectively, the components of $z_{i,j,\varepsilon}$ in the $\tau$ and $\nu$
directions.
Note that these functions are defined over $B_{r_{1},+}(0)$.

\section{Lemmas}\label{Lemmas}
\subsection{Extension for \texorpdfstring{$B_{r,+}(0)$}{}}
\begin{lemma}\label{Extension}
    Suppose $u\in{}W_{T}^{1,2}(B_{r,+}(0);\mathbb{R}^{2})$ and
    $\varphi\in{}\mathcal{A}_{\alpha,r}$ where $0<\alpha\le1$ and $r>0$.
    Then there is an extension $\tilde{u}\in{}W^{1,2}(B_{r}(0);\mathbb{R}^{2})$ of
    $u$ and an extension $\tilde{\varphi}\in{}C_{c}^{0,\alpha}(B_{r}(0))$
    of $\varphi$ such that
    \begin{equation}\label{eq:JacobianReflection}
        \bigl<\star{}J(\tilde{u}),\tilde{\varphi}\bigr>
        =2\bigl<\star{}J(u),\varphi\bigr>
    \end{equation}
    and
    \begin{align}
        \lVert\tilde{u}\rVert_{W^{1,2}(B_{r}(0);\mathbb{R}^{2})}&=
        2\lVert{}u\rVert_{W^{1,2}(B_{r,+}(0);\mathbb{R}^{2})},
        \label{SobolevDouble}\\
        \lVert\tilde{\varphi}\rVert_{C_{c}^{0,\alpha}(B_{r}(0))}
        &=\lVert\varphi\rVert_{\mathcal{A}_{\alpha,r}},
        \label{HolderExtensionNorm}\\
        E_{\varepsilon}(\tilde{u},B_{r}(0))&=2E_{\varepsilon}(u,B_{r,+}(0)).
        \label{eq:EnergyDouble}
    \end{align}
\end{lemma}

\begin{proof}
Since $u_{T}=0$ then if $u=u^{1}e_{1}+u^{2}e_{2}$
then $u^{1}(x_{1},0)=0$ for $|x_{1}|<r$.
We define $\tilde{u}\colon{}B_{r}(0)\to\mathbb{R}^{2}$ by
\begin{equation*}
    \tilde{u}(x_{1},x_{2})\coloneqq
    \begin{cases}
        u^{1}(x_{1},x_{2})e_{1}+u^{2}(x_{1},x_{2})e_{2},& \text{if }|x_{1}|<r
        \text{ and }x_{2}\ge0\\
        -u^{1}(x_{1},-x_{2})e_{1}+u^{2}(x_{1},-x_{2})e_{2},&
        \text{if }|x_{1}|<r\text{ and }x_{2}<0.
    \end{cases}
\end{equation*}
Note that we are
able to extend $u^{1}$ as an odd function because it is zero
along $\bigl\{(x_{1},0):|x_{1}|\le{}r\bigr\}$.
We note that $\tilde{u}\in{}W^{1,2}(B_{r}(0);\mathbb{R}^{2})$ and
that both \eqref{SobolevDouble} and \eqref{eq:EnergyDouble}
are satisfied.
Next we compute the Jacobian of $\tilde{u}$.
Doing this, we see that
\begin{equation}\label{ExtendedJacobian}
    J\tilde{u}(x_{1},x_{2})=
    \begin{cases}
        Ju(x_{1},x_{2}),& \text{if }|x_{1}|<r\text{ and }x_{2}\ge0\\
        Ju(x_{1},-x_{2}),& \text{if }|x_{1}|<r\text{ and }x_{2}<0.
    \end{cases}
\end{equation}
Next, consider a function $\varphi\in\mathcal{A}_{\alpha,r}$, where 
$\alpha\in(0,1]$ and $r>0$.
We extend $\varphi$ to $B_{r}(0)$ in the following way:
\begin{equation}\label{ExtendedTestFunction}
    \tilde{\varphi}(x_{1},x_{2})\coloneqq\varphi(x_{1},|x_{2}|)=
    \begin{cases}
        \varphi(x_{1},x_{2}),& \text{if }|x_{1}|<r\text{ and }x_{2}\ge0\\
        \varphi(x_{1},-x_{2}),& \text{if }|x_{1}|<r\text{ and }x_{2}<0
    \end{cases}
\end{equation}
and note that \eqref{HolderExtensionNorm} holds.
We also note that $\tilde{\varphi}\in{}C_{c}^{0,\alpha}(B_{r}(0))$
since $\varphi\equiv0$ in a neighbourhood of
$\partial{}B_{r,+}(0)\cap\{x_{2}>0\}$.
Finally, \eqref{eq:JacobianReflection} now follows from
\eqref{ExtendedJacobian}, \eqref{ExtendedTestFunction},
and a change of variables.
\end{proof}

\begin{lemma}\label{ExtensionCompactness}
    Suppose
    $\{u_{\varepsilon}\}_{\varepsilon\in(0,1]}\subseteq{}W_{T}^{1,2}(B_{r,+}(0);
    \mathbb{R}^{2})$ satisfies
    \begin{equation*}
        E_{\varepsilon}(u_{\varepsilon},B_{r,+}(0))
        \le{}C\mathopen{}\left|\log(\varepsilon)\right|\mathclose{}.
    \end{equation*}
    Then there exists a subsequence $\{\varepsilon_{k}\}_{k=1}^{\infty}$,
    non-zero integers $d_{i}$ and $d_{j}$ for
    $i=1,2,\ldots,{}M_{1}$ and $j=1,2,\ldots,M_{2}$ such that
    for all $0<\alpha\le1$ we have
    \begin{equation*}
        \Biggl\lVert\star{}J(u_{\varepsilon_{k}})
        -\pi\sum_{i=1}^{M_{1}}d_{i}\delta_{x_{i}}
        -\frac{\pi}{2}\sum_{j=1}^{M_{2}}d_{j}\delta_{x_{j}}
        \Biggr\rVert_{\mathcal{A}_{\alpha,r}^{*}}
        \longrightarrow{}0^{+},
    \end{equation*}
    where $x_{i}\in{}B_{r,+}(0)$ for $i=1,2,\ldots,M_{1}$ and
    $x_{j}\in\{(x,0):|x|<r\}$ for
    $j=1,2,\ldots,M_{2}$.
\end{lemma}

\begin{proof}
It follows from Theorem $3.1$ of \cite{JS} that there is a subsequence
$\{u_{\varepsilon_{k}}\}_{k=1}^{\infty}\subseteq{}
W_{T}^{1,2}(B_{r,+}(0);\mathbb{R}^{2})
$, points $\{x_{i}\}_{i=1}^{M_{1}}\subseteq{}B_{r,+}(0)$, and non-zero integers
$\{d_{i}\}_{i=1}^{M_{1}}$ such that
\begin{equation}\label{HalfConvergence}
    \Biggl\lVert\star{}J(u_{\varepsilon_{k}})-
    \pi\sum_{i=1}^{M_{1}}d_{i}\delta_{x_{i}}\Biggr\rVert_
    {(C_{c}^{0,\alpha}(B_{r,+}(0)))^{*}}\longrightarrow0^{+},
\end{equation}
for each $0<\alpha\le1$ as $k\to\infty$.
We now demonstrate that we can extend this convergence, after an appropriate
modification to the limiting measure, to hold for $\mathcal{A}_{\alpha,r}^{*}$, the dual of $\mathcal{A}_{\alpha,r}$.
By Lemma \ref{Extension} we may extend
each $u_{\varepsilon_{k}}$ to a function $\tilde{u}_{\varepsilon_{k}}$
defined over $B_{r}(0)$.
We also have by \eqref{eq:EnergyDouble} of Lemma \ref{Extension} that
\begin{equation*}
    E_{\varepsilon_{k}}(\tilde{u}_{\varepsilon_{k}},B_{r}(0))
    \le{}2C\mathopen{}\left|\log(\varepsilon_{k})\right|\mathclose{}
\end{equation*}
for each $k\in\mathbb{N}$.
Applying Theorem $3.1$ of \cite{JS} again we obtain a further subsequence
$\{\tilde{u}_{\varepsilon_{k}}\}_{k=1}^{\infty}$, sequence with $\varepsilon_{k}$,
points $\{\tilde{x}_{i}\}_{i=1}^{\tilde{M}}\subseteq{}B_{r}(0)$,
and non-zero integers $\{\tilde{d}_{i}\}_{i=1}^{\tilde{M}}$ such that
\begin{equation}\label{TotalConvergence}
    \Biggl\lVert\star{}J(\tilde{u}_{\varepsilon_{k}})
    -\pi\sum_{i=1}^{\tilde{M}}\tilde{d}_{i}\delta_{\tilde{x}_{i}}
    \Biggr\rVert_{(C_{c}^{0,\alpha}(B_{r}(0)))^{*}}\longrightarrow{}0^{+}
\end{equation}
for each $0<\alpha\le1$ as $k\to\infty$.
We let
$J_{0}\coloneqq\pi\sum\limits_{i=1}^{\tilde{M}}\tilde{d}_{i}\delta_{\tilde{x}_{i}}$.
Observe that we may decompose $J_{0}$ as
\begin{equation*}
    J_{0}
    =\pi\sum_{i=1}^{\tilde{M}_{1}}\tilde{d}_{i}\delta_{\tilde{x}_{i}}
    +\pi\sum_{i=1}^{\tilde{M}_{2}}\tilde{d}_{i}\delta_{\tilde{x}_{i}}
    +\pi\sum_{i=1}^{\tilde{M}_{3}}\tilde{d}_{i}\delta_{\tilde{x}_{i}}
\end{equation*}
where
\begin{align*}
    \{\tilde{x}_{i}\}_{i=1}^{\tilde{M}_{1}}&\subseteq{}B_{r,+}(0),\\
    \{\tilde{x}_{i}\}_{i=1}^{\tilde{M}_{2}}&\subseteq{}
    \{(x,0):|x|<r\},\\
    \{\tilde{x}_{i}\}_{i=1}^{\tilde{M}_{3}}&\subseteq{}
    B_{r}(0)\setminus\overline{B_{r,+}(0)}
\end{align*}
and $\tilde{M}=\tilde{M}_{1}+\tilde{M}_{2}+\tilde{M}_{3}$.
Combining \eqref{HalfConvergence} and \eqref{TotalConvergence} we find that
$\tilde{M}_{1}=M_{1}$, as well as that $\tilde{x}_{i}=x_{i}$ and
$\tilde{d}_{i}=d_{i}$ for $i=1,2,\ldots,M_{1}$.
In addition, by a symmetry argument combined with the identity
\eqref{eq:JacobianReflection}  we conclude that
$\tilde{M}_{3}=M_{1}$, $\tilde{d}_{i}=d_{j}$ for
$i=1,2,\ldots,\tilde{M}_{3}$ and some $j=1,2,\ldots,M_{1}$,
and each $\tilde{x}_{i}$, for
$i=1,2,\ldots,\tilde{M}_{3}$ is the reflection of some point $x_{j}$ for
$j=1,2,\ldots,M_{1}$.
\\

Thus, after possibly reindexing some of the points, we have
\begin{equation*}
    J_{0}
    =\pi\sum_{i=1}^{M_{1}}d_{i}\bigl(\delta_{x_{i}}+\delta_{\bar{x}_{i}}\bigr)
    +\pi\sum_{i=1}^{\tilde{M}_{2}}\tilde{d}_{i}\delta_{\tilde{x}_{i}},
\end{equation*}
where $\bar{x}$ denotes the reflection of $x$ across the $x$-axis.
With this in place we now prove the desired convergence.
Let $\varphi\in{}\mathcal{A}_{\alpha,r}$ be such that
$\left\|\varphi\right\|_{\mathcal{A}_{\alpha,r}}\le1$.
Observe that
\begin{align*}
    \biggl<\pi\sum_{i=1}^{M_{1}}d_{i}\delta_{x_{i}}
    +\frac{\pi}{2}\sum_{j=1}^{\tilde{M}_{2}}\tilde{d}_{j}\delta_{x_{j}},
    \varphi\biggr>
    &=
    \biggl<\pi\sum_{i=1}^{M_{1}}d_{i}\delta_{x_{i}}
    +\frac{\pi}{2}\sum_{j=1}^{\tilde{M}_{2}}\tilde{d}_{j}\delta_{x_{j}},
    \tilde{\varphi}\biggr>\\
    &=\biggl<\frac{\pi}{2}\sum_{i=1}^{M_{1}}d_{i}\bigl(\delta_{x_{i}}
    +\delta_{\bar{x}_{i}}\bigr)
    +\frac{\pi}{2}\sum_{j=1}^{\tilde{M}_{2}}\tilde{d}_{j}\delta_{x_{j}},
    \tilde{\varphi}\biggr>\\
    &=\frac{1}{2}\bigl<J_{0},
    \tilde{\varphi}\bigr>.
\end{align*}
Hence, by combining the above observation with
\eqref{eq:JacobianReflection} and \eqref{HolderExtensionNorm} we have
\begin{align*}
    \biggl<\star{}J(u_{\varepsilon_{k}})-
    \pi\sum_{i=1}^{M_{1}}d_{i}\delta_{x_{i}}
    -\frac{\pi}{2}\sum_{j=1}^{\tilde{M}_{2}}\tilde{d}_{j}\delta_{\tilde{x}_{j}},
    \varphi\biggr>
    &=\frac{1}{2}\bigl<\star{}J(\tilde{u}_{\varepsilon_{k}})-
    J_{0},
    \tilde{\varphi}
    \bigr>\\
    &\le\frac{1}{2}\lVert\star{}J(\tilde{u}_{\varepsilon_{k}})-
    J_{0}
    \rVert_{(C_{c}^{0,\alpha}(B_{r}(0)))^{*}}.
\end{align*}
Taking the supremum over such $\varphi$ and using \eqref{TotalConvergence} gives
the desired result.
\end{proof}

\subsection{Interpolation}
Here we prove an analog of Lemma $3.3$ of \cite{JS} for $B_{r,+}(0)$.
Specifically, we show that for each signed Radon measure on $B_{r,+}(0)$,
$\nu$, we can estimate its norm in $\mathcal{A}_{\alpha,r}^{*}$, for
$0<\alpha<1$ in terms of its norm in $\mathcal{A}_{1,r}^{*}$ and
$(C(B_{r,+}(0)))^{*}$ as well as the constant $\alpha$.\\

We introduce some notation in order to prove the desired result.
We let $\eta\colon\mathbb{R}^{2}\to\mathbb{R}$ be a non-negative, smooth,
symmetric function with $\text{supp}(\eta)\subseteq{}B_{1}(0)$ satisfying
$\int_{B_{1}(0)}\!{}\eta=1$.
We define, for $\delta>0$, $r>0$, $0<\alpha\le1$, and
$\varphi\in\mathcal{A}_{\alpha,r}$, the mollified function
$\varphi_{\delta}\colon{}B_{r,+}(0)\to\mathbb{R}$ by
\begin{equation}\label{Mollifier:Def}
    \varphi_{\delta}(y)\coloneqq
    \int_{B_{\delta}(y)\cap{}B_{r,+}(0)}\!{}
    \frac{1}{\delta^{2}}\eta\biggl(\frac{y-z}{\delta}\biggr)\varphi(z)\,
    \mathrm{d}z.
\end{equation}
We use the mollified function to obtain an approximation to $\varphi$
whose Lipschitz norm can be explicitly estimated.
We begin with a preparatory lemma.
\begin{lemma}\label{MollifierBoundaryValues}
    Suppose $r>0$, $0<\alpha\le1$, $\varphi\in\mathcal{A}_{\alpha,r}$, and
    $\delta>0$.
    Suppose further that $s\in(0,r)$ is such that
    \begin{equation*}
        \emph{supp}(\varphi)\subseteq{}B_{s,+}(0)
    \end{equation*}
    and $M_{\delta}>0$ is chosen so that
    \begin{equation}\label{eq:MollifierLip}
        [\varphi_{\delta}]_{1}\le{}M_{\delta}.
    \end{equation}
    Then the functions $\sigma_{1,\delta},\sigma_{2,\delta}\colon{}B_{r,+}(0)\to\mathbb{R}$
    defined by
    \begin{align*}
        \sigma_{1,\delta}(y)&\coloneqq
        \sup_{x\in{}B_{r,+}(0)\setminus{}B_{s,+}(0)}
        \max\bigl\{\varphi_{\delta}(x)-M_{\delta}|x-y|,
        0\bigr\},\\
        \sigma_{2,\delta}(y)&\coloneqq
        \sup_{x\in{}B_{r,+}(0)\setminus{}B_{s,+}(0)}
        \min\bigl\{\varphi_{\delta}(x)+M_{\delta}|x-y|,
        0\bigr\},
    \end{align*}
    satisfy
    \begin{align}
        \varphi_{\delta}=\sigma_{1,\delta}-\sigma_{2,\delta}\hspace{15pt}
        &\text{on }B_{r,+}(0)\setminus{}B_{s,+}(0),
        \label{BoundaryAgreement}\\
        \lVert\sigma_{1,\delta}\rVert_{L^{\infty}(B_{r,+}(0))}
        &+\lVert\sigma_{2,\delta}\rVert_{L^{\infty}(B_{r,+}(0))}
        \le2\delta^{\alpha}\lVert\varphi\rVert_{C^{0,\alpha}(B_{r,+}(0))},
        \label{SupEstimate}\\
        [\sigma_{1,\delta}]_{1}+[\sigma_{2,\delta}]_{1}&\le2M_{\delta}.
        \label{LipschitzEstimate}
    \end{align}
\end{lemma}

\begin{proof}
The proof is similar to the one found in the proof of Lemma $3.3$ of
\cite{JS}.
\end{proof}

\begin{lemma}\label{InterpolationEstimate}
    Suppose $\nu$ is a signed Radon measure on $B_{r,+}(0)$ and $0<\alpha<1$ and
    $r>0$.
    Then
    \begin{equation*}
        \lVert\nu\rVert_{\mathcal{A}_{\alpha,r}^{*}}
        \le{}
        C
        \lVert\nu\rVert_{\mathcal{A}_{1,r}^{*}}^{\frac{\alpha}{1+\alpha}}
        \lVert\nu\rVert_{(C(B_{r,+}(0)))^{*}}^{\frac{1}{1+\alpha}}.
    \end{equation*}
\end{lemma}
\begin{proof}
    Without loss of generality we may assume
    $0<\lVert\nu\rVert_{(C(B_{r,+}(0)))^{*}}<\infty$ and
    $0<\lVert\nu\rVert_{\mathcal{A}_{1,r}^{*}}<\infty$.
    Let $\varphi\in\mathcal{A}_{\alpha,r}$ for $0<\alpha<1$.
    We let $\delta>0$ denote a scale of regularization to be chosen later.
    We prove a preliminary estimate by considering two cases.\\
    
    \noindent\boxed{\text{Case }1\text{: }0<\delta<r}\\
    \par{}Let $\zeta\colon[0,\infty)\to\mathbb{R}$ denote a smooth non-increasing
    function satisfying
    $\zeta\ge0$, $\zeta\equiv1$ on $[0,\delta]$, $\zeta\equiv0$
    on $[2\delta,\infty)$, and
    $\left\|\zeta'\right\|_{L^{\infty}}\le\frac{2}{\delta}$.
    Observe that
    \begin{align*}
        \int_{B_{r,+}(0)}\!{}\varphi(y)\,\mathrm{d}\nu
        &=\int_{B_{r,+}(0)}\!{}
        (1-\zeta(y_{2}))\varphi(y)\,\mathrm{d}\nu
        +\int_{B_{r,+}(0)}\!{}
        \zeta(y_{2})[\varphi(y)-\varphi(y_{1},0)]\,\mathrm{d}\nu
        \\
        &
        +\int_{B_{r,+}(0)}\!{}
        \zeta(y_{2})\varphi(y_{1},0)\,\mathrm{d}\nu\\
        &=I_{1}+I_{2}+I_{3}.
    \end{align*}
    Then notice that we can estimate
    $I_{2}$ as
    \begin{align*}
        |I_{2}|
        &\le\lVert\zeta(y_{2})[\varphi(y)-\varphi(y_{1},0)]\rVert
        _{L^{\infty}(B_{r,+}(0))}
        \lVert\nu\rVert_{(C(B_{r,+}(0)))^{*}}\\
        &\le2^{\alpha}\delta^{\alpha}
        \lVert\varphi\rVert_{C^{0,\alpha}(B_{r,+}(0))}
        \lVert\nu\rVert_{(C(B_{r,+}(0)))^{*}}.
    \end{align*}
    We let $\varphi_{\delta}$ be as defined in \eqref{Mollifier:Def}
    and we observe that for $y=(y_{1},0)$, using that $\varphi=0$ in a
    neighbourhood of
    $\partial{}B_{r,+}(0)\cap\{y_{2}>0\}$ as well as symmetry of $\eta$,
    we have
    \begin{equation}\label{HalfEstimate}
        \biggl\lVert\frac{1}{2}\varphi(y_{1},0)-\varphi_{\delta}(y_{1},0)
        \biggr\rVert_{L^{\infty}(-r,r)}
        \le\delta^{\alpha}\lVert\varphi\rVert_{C^{0,\alpha}(B_{r,+}(0))}.
    \end{equation}
    We now decompose $I_{3}$ as
    \begin{align*}
        I_{3}&=2\int_{B_{r,+}(0)}\!{}
        \zeta(y_{2})\frac{\varphi(y_{1},0)}{2}\,\mathrm{d}\nu\\
        &=2\int_{B_{r,+}(0)}\!{}
        \zeta(y_{2})\Bigl[\frac{1}{2}\varphi(y_{1},0)
        -\varphi_{\delta}(y_{1},0)\Bigr]\,
        \mathrm{d}\nu\\
        &+2\int_{B_{r,+}(0)}\!{}
        \zeta(y_{2})[\varphi_{\delta}(y_{1},0)
        -\varphi_{\delta}(y_{1},y_{2})]\,\mathrm{d}\nu\\
        &+2\int_{B_{r,+}(0)}\!{}\zeta(y_{2})\bigl[
        \varphi_{\delta}(y_{1},y_{2})-\sigma_{1,\delta}(y_{1},y_{2})
        +\sigma_{2,\delta}(y_{1},y_{2})\bigr]\,
        \mathrm{d}\nu\\
        &+2\int_{B_{r,+}(0)}\!{}\zeta(y_{2})\bigl[\sigma_{1,\delta}(y_{1},y_{2})
        -\sigma_{2,\delta}(y_{1},y_{2})\bigr]\,\mathrm{d}\nu\\
        &=A+B+C+D.
    \end{align*}
    By \eqref{HalfEstimate} we have
    \begin{equation*}
        |A|\le2\delta^{\alpha}\lVert\varphi\rVert
        _{C^{0,\alpha}(B_{r,+}(0))}
        \lVert\nu\rVert_{(C(B_{r,+}(0)))^{*}}.
    \end{equation*}
    Since $\text{supp}(\zeta)\subseteq[0,2\delta]$ then
    \begin{equation*}
        \lVert\zeta(y_{2})[\varphi_{\delta}(y_{1},0)
        -\varphi_{\delta}(y_{1},y_{2})]\rVert_{L^{\infty}(B_{r,+}(0))}
        \le2^{\alpha}\delta^{\alpha}
        \lVert\varphi\rVert_{C^{0,\alpha}(B_{r,+}(0))}
    \end{equation*}
    and hence
    \begin{equation*}
        |B|\le2^{\alpha+1}\delta^{\alpha}
        \lVert\varphi\rVert_{C^{0,\alpha}(B_{r,+}(0))}
        \lVert\nu\rVert_{(C(B_{r,+}(0)))^{*}}.
    \end{equation*}
    Next by \eqref{SupEstimate} we have
    \begin{equation*}
        |D|\le4\delta^{\alpha}\left\|\varphi\right\|
        _{C^{0,\alpha}(B_{r,+}(0))}
        \lVert\nu\rVert_{(C(B_{r,+}(0)))^{*}}.
    \end{equation*}
    We observe that, since $\varphi=0$ in a neighbourhood of
    $\partial{}B_{r,+}(0)\cap\{y_{2}>0\}$, for each
    $(y_{1},y_{2})\in{}B_{r,+}(0)$ we have
    \begin{align*}
        \biggl|
        \frac{\partial\varphi_{\delta}}{\partial{}y_{1}}(y_{1},y_{2})
        \biggr|
        &=\biggl|\int_{B_{\delta}(y_{1},y_{2})\cap{}B_{r,+}(0)}\!{}
        \frac{1}{\delta^{3}}
        \frac{\partial\eta}{\partial{}y_{1}}
        \Bigl(\frac{(y_{1},y_{2})-z}{\delta}\Bigr)
        \varphi(z)\,\mathrm{d}z\biggr|\\
        &\le{}\delta^{-1}r^{\alpha}
        \lVert\varphi\rVert_{C^{0,\alpha}(B_{r,+}(0))}
        \lVert\eta\rVert_{W^{1,1}(\mathbb{R}^{2})}
    \end{align*}
    using that
    \begin{equation*}
        \varphi\Bigl(z_{1},\sqrt{r^{2}-z_{1}^{2}}\Bigr)=0,\hspace{20pt}
        \Bigl|\sqrt{r^{2}-z_{1}^{2}}-z_{2}\Bigr|\le{}r.
    \end{equation*}
    A similar estimate holds for the partial derivative in $y_{2}$ and
    we conclude that
    \begin{equation}\label{MollifierLipschitz}
        [\varphi_{\delta}]_{1}
        \le2\delta^{-1}r^{\alpha}
        \lVert\varphi\rVert_{C^{0,\alpha}(B_{r,+}(0))}
        \lVert\eta\rVert_{W^{1,1}(\mathbb{R}^{2})}.
    \end{equation}
    Now we estimate $C$.
    Setting $M_{\delta}\coloneqq
    2\delta^{-1}r^{\alpha}\lVert\varphi\rVert_{C^{0,\alpha}(B_{r,+}(0))}
    \lVert\eta\rVert_{W^{1,1}(\mathbb{R}^{2})}$
    we see that by \eqref{eq:MollifierLip}, \eqref{BoundaryAgreement},
    \eqref{SupEstimate},
    and \eqref{LipschitzEstimate} we have
    $\zeta[\varphi_{\delta}-\sigma_{1,\delta}+\sigma_{2,\delta}]\in\mathcal{A}_{1,r}$.
    and
    \begin{align*}
        |C|&\le\Bigl[12\delta^{-1}r^{\alpha}
        \lVert\eta\rVert_{W^{1,1}(\mathbb{R}^{2})}
        +8\delta^{\alpha-1}
        +4r^{\alpha}\delta^{-1}
        \Bigr]
        \lVert\varphi\rVert_{C^{0,\alpha}(B_{r,+}(0))}
        \lVert\nu\rVert_{\mathcal{A}_{1,r}^{*}}\\
        &\le{}C
        \delta^{-1}\lVert\varphi\rVert_{C^{0,\alpha}(B_{r,+}(0))}
        \lVert\nu\rVert_{\mathcal{A}_{1,r}^{*}}.
    \end{align*}
    Now we estimate $I_{1}$.
    \begin{align*}
        I_{1}&=\int_{B_{r,+}(0)}\!{}(1-\zeta(y_{2}))
        [\varphi(y)-\varphi_{\delta}(y)]\,\mathrm{d}\nu\\
        &+\int_{B_{r,+}(0)}\!{}(1-\zeta(y_{2}))
        [\varphi_{\delta}(y)-\sigma_{1,\delta}(y)+\sigma_{2,\delta}(y)]\,
        \mathrm{d}\nu\\
        &+\int_{B_{r,+}(0)}\!{}(1-\zeta(y_{2}))
        [\sigma_{1,\delta}(y)-\sigma_{2,\delta}(y)]\,\mathrm{d}\nu\\
        &=(i)+(ii)+(iii)
    \end{align*}
    where $\sigma_{1,\delta}$ and $\sigma_{2,\delta}$ are the functions
    introduced in Lemma \ref{MollifierBoundaryValues}.
    Since $\varphi\equiv0$ in a neighbourhood of
    $\partial{}B_{r,+}(0)\cap\{y_{2}>0\}$ and since for
    $(y_{1},y_{2})\in{}B_{r,+}(0)\cap\{y_{2}\ge\delta\}$ we have
    $B_{\delta}(y_{1},y_{2})\cap\{y_{2}=0\}=\varnothing$ then
    \begin{equation*}
        \varphi(y)-\varphi_{\delta}(y)
        =\int_{B_{\delta}(y)}\!{}\frac{1}{\delta^{2}}
        \eta\Bigl(\frac{y-z}{\delta}\Bigr)
        [\varphi(y)-\varphi(z)]\,\mathrm{d}z
        \le\delta^{\alpha}\lVert\varphi\rVert_{C^{0,\alpha}(B_{r,+}(0))}
    \end{equation*}
    for $B_{r,+}(0)\cap\{y_{2}\ge\delta\}$.
    Since $\text{supp}(1-\zeta(y_{2}))
    \subseteq{}B_{r,+}(0)\cap\{y_{2}\ge\delta\}$
    then
    \begin{equation*}
        \lVert(1-\zeta(y_{2}))[\varphi-\varphi_{\delta}]\rVert_{L^{\infty}
        (B_{r,+}(0))}
        \le\delta^{\alpha}\lVert\varphi\rVert_{C^{0,\alpha}(B_{r,+}(0))}.
    \end{equation*}
    Thus,
    \begin{equation*}
        |(i)|
        \le\delta^{\alpha}\lVert\varphi\rVert_{C^{0,\alpha}(B_{r,+}(0))}
        \lVert\nu\rVert_{(C(B_{r,+}(0)))^{*}}.
    \end{equation*}
    Using \eqref{eq:MollifierLip}, \eqref{BoundaryAgreement}, \eqref{SupEstimate},
    \eqref{LipschitzEstimate}, and \eqref{MollifierLipschitz} we have
    \begin{align*}
        |(ii)|&\le\Bigl[
        6\delta^{-1}r^{\alpha}
        \lVert\eta\rVert_{W^{1,1}(\mathbb{R}^{2})}
        +4\delta^{\alpha-1}
        +2r^{\alpha}\delta^{-1}
        \Bigr]
        \lVert\varphi\rVert_{C^{0,\alpha}(B_{r,+}(0))}
        \lVert\nu\rVert_{\mathcal{A}_{1,r}^{*}}\\
        &\le{}C\delta^{-1}\lVert\varphi\rVert_{C^{0,\alpha}(B_{r,+}(0))}
        \lVert\nu\rVert_{\mathcal{A}_{1,r}^{*}}.
    \end{align*}
    Finally, by \eqref{SupEstimate} we have
    \begin{equation*}
        |(iii)|\le2\delta^{\alpha}
        \lVert\varphi\rVert_{C^{0,\alpha}(B_{r,+}(0))}
        \lVert\nu\rVert_{(C(B_{r,+}(0)))^{*}}.
    \end{equation*}
    The above gives
    \begin{equation*}
        \int_{B_{r,+}(0)}\!{}\varphi\,\mathrm{d}\nu
        \le{}C\Bigl[\delta^{\alpha}
        \lVert\nu\rVert_{(C(B_{r,+}(0)))^{*}}
        +\delta^{-1}
        \lVert\nu\rVert_{\mathcal{A}_{1,r}^{*}}\Bigr].
    \end{equation*}
    \noindent\boxed{\text{Case }2\text{: }\delta\ge{}r}\\
    \par{}In this case we write
    \begin{align*}
        \int_{B_{r,+}(0)}\!{}\varphi(y_{1},y_{2})\,\mathrm{d}\nu
        &=\int_{B_{r,+}(0)}\!{}
        [\varphi(y_{1},y_{2})-\varphi(y_{1},0)]\,\mathrm{d}\nu
        +\int_{B_{r,+}(0)}\!{}\varphi(y_{1},0)\,\mathrm{d}\nu\\
        &=E+F.
    \end{align*}
    Similar to the estimate for
    $I_{2}$ we observe that, since $|y_{2}|<r$
    on $B_{r,+}(0)$ and since $\delta\ge{}r$, we have
    \begin{equation*}
        |E|\le{}r^{\alpha}\lVert\varphi\rVert_{C^{0,\alpha}(B_{r,+}(0))}
        \lVert\nu\rVert_{(C(B_{r,+}(0)))^{*}}
        \le\delta^{\alpha}\lVert\varphi\rVert_{C^{0,\alpha}(B_{r,+}(0))}
        \lVert\nu\rVert_{(C(B_{r,+}(0)))^{*}}.
    \end{equation*}
    Using the mollification, $\varphi_{\delta}$, from
    \eqref{Mollifier:Def} as well as the
    functions $\sigma_{1,\delta}$ and $\sigma_{2,\delta}$ from Lemma
    \ref{MollifierBoundaryValues} we can rewrite $F$ similar to the estimate
    for 
    $I_3$ to obtain
    \begin{align*}
        F&=2\int_{B_{r,+}(0)}\!{}
        \frac{1}{2}\varphi(y_{1},0)\,\mathrm{d}\nu\\
        &=2\int_{B_{r,+}(0)}\!{}
        \Bigl[\frac{1}{2}\varphi(y_{1},0)
        -\varphi_{\delta}(y_{1},0)\Bigr]\,
        \mathrm{d}\nu
        +2\int_{B_{r,+}(0)}\!{}
        [\varphi_{\delta}(y_{1},0)
        -\varphi_{\delta}(y_{1},y_{2})]\,\mathrm{d}\nu\\
        &+2\int_{B_{r,+}(0)}\!{}
        [\varphi_{\delta}(y)-\sigma_{1,\delta}(y)
        +\sigma_{2,\delta}(y)]\,\mathrm{d}\nu
        +2\int_{B_{r,+}(0)}\!{}[\sigma_{1,\delta}(y)-\sigma_{2,\delta}(y)]\,
        \mathrm{d}\nu\\
        &=FA+FB+FC+FD.
    \end{align*}
    After possibly extending $\varphi$ by zero across
    $\partial{}B_{r,+}(0)\cap\{y_{2}>0\}$ we have, by symmetry of $\eta$
    and \eqref{HalfEstimate} that
    \begin{equation*}
        |FA|\le2\delta^{\alpha}\lVert\varphi\rVert_{C^{0,\alpha}(B_{r,+}(0))}
        \lVert\nu\rVert_{(C(B_{r,+}(0)))^{*}}.
    \end{equation*}
    Next, we estimate $FB$ similar to 
    $B$ and use that $r\le\delta$
    to obtain
    \begin{equation*}
        |FB|\le
        2\delta^{\alpha}\lVert\varphi\rVert_{C^{0,\alpha}(B_{r,+}(0))}
        \lVert\nu\rVert_{(C(B_{r,+}(0)))^{*}}.
    \end{equation*}
    We estimate $FD$ using \eqref{SupEstimate} as
    \begin{equation*}
        |FD|\le4\delta^{\alpha}
        \lVert\varphi\rVert_{C^{0,\alpha}(B_{r,+}(0))}
        \lVert\nu\rVert_{(C(B_{r,+}(0)))^{*}}.
    \end{equation*}
    Finally, \eqref{BoundaryAgreement}, \eqref{LipschitzEstimate}, and
    \eqref{MollifierLipschitz} gives that
    $\varphi_{\delta}-\sigma_{1,\delta}+\sigma_{2,\delta}\in\mathcal{A}_{\alpha,r}$
    and hence
    \begin{equation*}
        |FC|\le{}12\delta^{-1}r^{\alpha}
        \lVert\eta\rVert_{W^{1,1}(\mathbb{R}^{2})}
        \lVert\varphi\rVert_{C^{0,\alpha}(B_{r,+}(0))}
        \lVert\nu\rVert_{\mathcal{A}_{\alpha,r}^{*}}.
    \end{equation*}
    Putting the estimates of this case together we have
    \begin{equation*}
        \int_{B_{r,+}(0)}\!{}\varphi\,\mathrm{d}\nu
        \le{}C\Bigl[\delta^{\alpha}
        \lVert\nu\rVert_{(C(B_{r,+}(0)))^{*}}
        +\delta^{-1}
        \lVert\nu\rVert_{\mathcal{A}_{1,r}^{*}}\Bigr]
        \lVert\varphi\rVert_{C^{0,\alpha}(B_{r,+}(0))}.
    \end{equation*}
    Putting the two cases together and noting that
    $\varphi\in\mathcal{A}_{\alpha,r}$ was arbitrary then we obtain
    \begin{equation*}
        \lVert\nu\rVert_{\mathcal{A}_{\alpha,r}^{*}}
        \le{}C
        \bigl[\delta^{\alpha}\lVert\nu\rVert_{(C(B_{r,+}(0)))^{*}}
        +\delta^{-1}
        \lVert\nu\rVert_{\mathcal{A}_{1,r}^{*}}\bigr]
    \end{equation*}
    for all $\delta>0$.
    Taking
    $\delta=
    \lVert\nu\rVert_{\mathcal{A}_{1,r}^{*}}^{\frac{1}{1+\alpha}}
    \lVert\nu\rVert_{(C(B_{r,+}(0)))^{*}}^{\frac{-1}{1+\alpha}}$ we obtain
    \begin{equation*}
        \lVert\nu\rVert_{\mathcal{A}_{\alpha,r}^{*}}
        \le{}C\
        \lVert\nu\rVert_{\mathcal{A}_{1,r}^{*}}^{\frac{\alpha}{1+\alpha}}
        \lVert\nu\rVert_{(C(B_{r,+}(0)))^{*}}^{\frac{1}{1+\alpha}}.
    \end{equation*}
\end{proof}

\subsection{Slicing}\label{sec:Slicing}
In this subsection we prove, through a slicing argument, that the
portion of the vortex set near the boundary cannot meet many normal lines
to the boundary.
This will be used in order to estimate the size of error terms involving the
modulus.
\begin{center}
\begin{figure}
\begin{center}
\begin{tikzpicture}[use Hobby shortcut,decoration={
    markings,
    mark=between positions 0 and 1 step 0.0125 with {\draw[thin] (0,0)--(0,0.2);},
  }]  
\draw[rotate=-90,scale=2,postaction={decorate}] (3,0) .. +(1,0) .. +(1,2) .. +(1,3) .. 
+(0,3) .. (3,0);   
\begin{scope}  
\clip[draw] (2.3,-5.7) .. (2.5,-5.2) .. (2.4,-5.3) .. (2.3,-5.7);
\fill[pattern=north east lines](2.5,-5.5) circle(1);
\end{scope}
\begin{scope}  
\clip[draw] (4.5,-7.5) .. (4.9,-7.45) .. (4.4,-7.25) .. (4.5,-7.5);
\fill[pattern=north east lines](4.5,-7.5) circle(1);
\end{scope}
\begin{scope}  
\clip[draw] (1.4,-6.5) coordinate (M) circle[radius=2mm];
\fill[pattern=north east lines](1.4,-6.5) circle(1);
\end{scope}
\end{tikzpicture}
\caption{Depiction of extended normal lines.}
\end{center}
\end{figure}
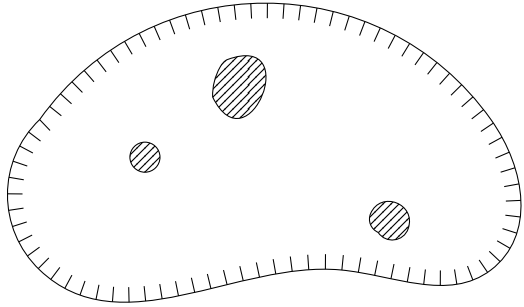
\end{center}
\begin{lemma}\label{Containment}
    Suppose $u\in{}W^{1,2}(\Omega;\mathbb{R}^{2})$
    satisfies $E_{\varepsilon}(u)\le{}
    C\mathopen{}\left|\log(\varepsilon)\right|\mathclose{}$.
    Then, for each $i=0,1,2,\ldots,b$ and $j=1,2,\ldots,N_{i}$, if we set
    \begin{equation*}
        \mathcal{B}_{i,j,\varepsilon}\coloneqq
        \biggl\{t\in(-r_{1},r_{1}):
        \exists{}s\in\Bigl[0,\sqrt{r_{1}^{2}-t^{2}}\Bigr)\,\,\emph{s.t.}\,\,
        \bigl||u(\psi_{i,j}(t,s))|-1\bigr|
        >\varepsilon^{\frac{1}{8}}\biggr\}
    \end{equation*}
    then we have
    \begin{equation*}
        \mathcal{L}^{1}(\mathcal{B}_{i,j,\varepsilon})
        \le{}C(\Omega)\varepsilon^{\frac{1}{4}}.
    \end{equation*}
\end{lemma}
\begin{proof}
    Notice that by the construction in Section \ref{TangentNormalCoordinates}
    there are $C^{1,1}$-coordinate charts
    $\psi_{i,j}\colon{}B_{r_{1},+}(0)\to\mathcal{U}_{i,j}$,
    for $i=0,1,\ldots,b$ and $j=1,2,\ldots,N_{i}$.
Our first ``bad" set is
\begin{equation*}
    \mathcal{B}_{0,i,j,\varepsilon}\coloneqq\biggl\{t\in(-r_{1},r_{1}):
    t\text{ is a singular value for }(\psi_{i,j}^{-1})^{1}\biggr\}.
\end{equation*}
Since $\psi_{i,j}$ is a $C^{1,1}$-diffeomorphism then by Sard's theorem,
see \cite{Ba}, we have
that
\begin{equation*}
    \mathcal{L}^{1}(\mathcal{B}_{0,i,j,\varepsilon})=0.
\end{equation*}
We let $\mathcal{G}_{0,i,j,\varepsilon}\coloneqq(-r_{1},r_{1})\setminus
\mathcal{B}_{0,i,j,\varepsilon}$.
Next, we consider the set
\begin{equation*}
    \mathcal{B}_{1,i,j,\varepsilon}\coloneqq\biggl\{t\in(-r_{1},r_{1}):
    \int_{\Gamma_{i,j,t}}\!{}\bigl|\nabla|u|\bigr|^{2}>
    \frac{1}{\varepsilon^{\frac{1}{4}}}
    \mathopen{}\left|\log(\varepsilon)\right|\mathclose{}\biggr\},
\end{equation*}
where we have set
$\Gamma_{i,j,t}\coloneqq\bigl\{x\in\mathcal{U}_{i,j}:
((\psi_{i,j}^{-1})^{1})(x)=t\bigr\}$.
By Chebyshev's inequality and the coarea formula we have
\begin{equation*}
    \mathcal{L}^{1}(\mathcal{B}_{1,i,j,\varepsilon})\le{}
    \frac{\displaystyle\int_{\mathcal{B}_{1,i,j,\varepsilon}}\int_{\Gamma_{i,j,t}}\!{}
    \bigl|\nabla|u|\bigr|^{2}}
    {\frac{\mathopen{}\left|\log(\varepsilon)\right|\mathclose{}}{\varepsilon^{\frac{1}{4}}}}
    \le\frac{\varepsilon^{\frac{1}{4}}}{\mathopen{}\left|\log(\varepsilon)\right|\mathclose{}}
    \cdot\int_{\mathcal{U}_{i,j}}\!{}\bigl|\nabla|u|\bigr|^{2}
    |\nabla(\psi_{i,j}^{-1})^{1}|
    \le{}C(\Omega)\varepsilon^{\frac{1}{4}}.
\end{equation*}
We let $\mathcal{G}_{1,i,j,\varepsilon}\coloneqq(-r_{1},r_{1})
\setminus\mathcal{B}_{1,i,j,\varepsilon}$.
The final ``bad" set is
\begin{equation*}
    \mathcal{B}_{2,i,j,\varepsilon}\coloneqq\bigg\{t\in\mathcal{G}_{0,i,j,\varepsilon}
    \cap\mathcal{G}_{1,i,j.\varepsilon}:
    \exists{}y\in\Bigl[0,\sqrt{r_{1}^{2}-t^{2}}\Bigr)\text{ such that }
    \bigl||u(\psi_{i,j}(t,y))|-1\bigr|
    >\varepsilon^{\frac{1}{8}}\biggr\}.
\end{equation*}
To estimate this we let $t\in\mathcal{B}_{2,i,j,\varepsilon}$ and find $y$ such that
\begin{equation*}
    \bigl||u(\psi_{i,j}(t,y))|-1\bigr|
    >\varepsilon^{\frac{1}{8}}.
\end{equation*}
Suppose there is $y'\in[0,\sqrt{r_{1}^{2}-t^{2}})$ such that
$\bigl||u(\psi_{i,j}(t,y'))|-1\bigr|
\le\frac{\varepsilon^{\frac{1}{8}}}{2}$.
Then since $t\in\mathcal{G}_{0,i,j,\varepsilon}\cap\mathcal{G}_{1,i,j,\varepsilon}$ we have
\begin{align*}
    \frac{\varepsilon^{\frac{1}{8}}}{2}
    &\le\Bigl|\bigl||u(\psi_{i,j}(t,y))|-1\bigr|-
    \bigl||u(\psi_{i,j}(t,y'))|-1\bigr|\Bigr|
    \\
    &\le{}C(\Omega)\biggl(\int_{\Gamma_{i,j,t}}\!{}
    \bigl|\nabla|u|\bigr|^{2}\biggr)^{\frac{1}{2}}
    |\psi_{i,j}(t,y)-\psi_{i,j}(t,y')|^{\frac{1}{2}}\\
    &\le\frac{C(\Omega)\mathopen{}\left|\log(\varepsilon)\right|\mathclose{}
    ^{\frac{1}{2}}}
    {\varepsilon^{\frac{1}{8}}}
    |\psi_{i,j}(t,y)-\psi_{i,j}(t,y')|^{\frac{1}{2}}
\end{align*}
and hence
\begin{equation*}
    |\psi_{i,j}(t,y)-\psi_{i,j}(t,y')|
    \ge\frac{C(\Omega)\varepsilon^{\frac{1}{2}}}
    {\mathopen{}\left|\log(\varepsilon)\right|\mathclose{}}.
\end{equation*}
We conclude that the closest point along $\Gamma_{i,j,t}$ satisfying
$\bigl||u|-1\bigr|\le\frac{\varepsilon^{\frac{1}{8}}}{2}$ is
at least of distance
$\frac{C(\Omega)\varepsilon^{\frac{1}{2}}}
{\mathopen{}\left|\log(\varepsilon)\right|\mathclose{}}$
away and hence
\begin{equation*}
    \mathcal{H}^{1}\biggl(
    \Gamma_{i,j,t}\cap\Bigl\{\bigl||u|-1\bigr|
    >\frac{\varepsilon^{\frac{1}{8}}}{2}\Bigr\}\biggr)
    \ge\frac{C(\Omega)\varepsilon^{\frac{1}{2}}}
    {\mathopen{}\left|\log(\varepsilon)\right|\mathclose{}}.
\end{equation*}
From this we conclude that
\begin{align*}
    \int_{\mathcal{B}_{2,i,j,\varepsilon}}\int_{\Gamma_{i,j,t}}\!{}
    \frac{(1-|u|^{2})^{2}}{4\varepsilon^{2}}
    &\ge\int_{\mathcal{B}_{2,i,j,\varepsilon}}\int_{\Gamma_{i,j,t}}\!{}
    \frac{\bigl|1-|u|\bigr|^{2}}{4\varepsilon^{2}}\\
    &\ge\frac{\varepsilon^{\frac{1}{4}}}{16\varepsilon^{2}}
    \int_{\mathcal{B}_{2,i,j,\varepsilon}}\!{} \mathcal{H}^{1}\biggl(
    \Gamma_{i,j,t}\cap\Bigl\{\bigl||u|-1\bigr|
    >\frac{\varepsilon^{\frac{1}{8}}}{2}\Bigr\}\biggr)\\
    &\ge\frac{C(\Omega)\varepsilon^{\frac{3}{4}}}{\varepsilon^{2}
    \mathopen{}\left|\log(\varepsilon)\right|\mathclose{}}
    \cdot\mathcal{L}^{1}(\mathcal{B}_{2,i,j,\varepsilon}).
\end{align*}
By the coarea formula we have
\begin{equation*}
    \int_{\mathcal{B}_{2,i,j,\varepsilon}}\int_{\Gamma_{i,j,t}}\!{}
    \frac{(1-|u|^{2})^{2}}{4\varepsilon^{2}}
    \le\int_{\mathcal{U}_{i,j}}\!{}\frac{(1-|u|^{2})^{2}}
    {4\varepsilon^{2}}|\nabla(\psi_{i,j}^{-1})^{1}|
    \le{}C(\Omega)\mathopen{}\left|\log(\varepsilon)\right|\mathclose{}
\end{equation*}
and hence
\begin{equation*}
    \mathcal{L}^{1}(\mathcal{B}_{2,i,j,\varepsilon})\le{}
    C(\Omega)\varepsilon^{\frac{5}{4}}
    \mathopen{}\left|\log(\varepsilon)\right|^{2}\mathclose{}.
\end{equation*}
We set $\mathcal{G}_{2,i,j,\varepsilon}\coloneqq(-r_{1},r_{1})\setminus
\mathcal{B}_{2,i,j,\varepsilon}$.
Finally we set $\mathcal{G}_{i,j,\varepsilon}\coloneqq\mathcal{G}_{0,i,j,\varepsilon}\cap
\mathcal{G}_{1,i,j,\varepsilon}\cap\mathcal{G}_{2,i,j,\varepsilon}$ and
$\mathcal{B}_{i,j,\varepsilon}\coloneqq(-r_{1},r_{1})\setminus
\mathcal{G}_{i,j,\varepsilon}$.
We notice that
\begin{equation*}
    \mathcal{B}_{i,j,\varepsilon}
    \subseteq\mathcal{B}_{0,i,j,\varepsilon}\cup\mathcal{B}_{1,i,j,\varepsilon}
    \cup\mathcal{B}_{2,i,j,\varepsilon}
\end{equation*}
which has measure
\begin{equation*}
    \mathcal{L}^{1}(\mathcal{B}_{i,j})
    \le\mathcal{L}^{1}(\mathcal{B}_{0,i,j})
    +\mathcal{L}^{1}(\mathcal{B}_{1,i,j})
    +\mathcal{L}^{1}(\mathcal{B}_{2,i,j})
    \le{}C(\Omega)\varepsilon^{\frac{1}{4}}
\end{equation*}
and we have $\bigl||u|^{2}-1\bigr|\le{}C(\Omega)
\varepsilon^{\frac{1}{8}}$ on $\mathcal{G}_{i,j,\varepsilon}$.
\end{proof}

\begin{lemma}\label{lem:SmallBoundaryError}
    Suppose $u\in{}W^{1,2}(\Omega;\mathbb{R}^{2})$
    satisfies $E_{\varepsilon}(u)\le{}
    C\mathopen{}\left|\log(\varepsilon)\right|\mathclose{}$.
    Then
    \begin{equation*}
        \bigl\lVert|u|^{2}-1\bigr\rVert_{L^{2}(\partial\Omega)}
        \le{}C(\Omega)\varepsilon^{\frac{1}{16}}
        \mathopen{}\left|\log(\varepsilon)\right|\mathclose{}.
    \end{equation*}
\end{lemma}

\begin{proof}
    By appealing to a partition of unity and compactness of $\partial\Omega$
    it suffices to verify
    \begin{equation*}
        \bigl\lVert|u|^{2}-1\bigr\rVert
        _{L^{2}(\psi_{i,j}((-r_{1},r_{1})\times\{0\})
        )}\le{}C(\Omega)\varepsilon^{\frac{1}{16}}
        \mathopen{}\left|\log(\varepsilon)\right|\mathclose{}
    \end{equation*}
    for an arbitrary $i=0,1,\ldots,b$ and $j=1,2,\ldots,N_{i}$.
    Suppose that we let $\mathcal{G}_{i,j,\varepsilon}\coloneqq(-r_{1},r_{1})
    \setminus\mathcal{B}_{i,j,\varepsilon}$ where
    $\mathcal{B}_{i,j,\varepsilon}$ is as in Lemma \ref{Containment}.
    It then suffices to estimate each of
    \begin{equation*}
        \int_{\psi_{i,j}(\mathcal{B}_{i,j,\varepsilon})}\!{}
        (|u|^{2}-1)^{2},\hspace{20pt}
        \int_{\psi_{i,j}(\mathcal{G}_{i,j,\varepsilon})}\!{}
        (|u|^{2}-1)^{2}.
    \end{equation*}
    Since $\bigl||u|-1\bigr|<\varepsilon^{\frac{1}{8}}$ on
    $\psi_{i,j}(\mathcal{G}_{i,j,\varepsilon})$ then
    \begin{equation}\label{eq:GoodPieceEstimate}
        \biggl|\int_{\psi_{i,j}(\mathcal{G}_{i,j,\varepsilon})}\!{}
        (|u|^{2}-1)^{2}\biggr|
        \le{}C(\Omega)\varepsilon^{\frac{1}{4}}.
    \end{equation}
    Since $(|u|^{2}-1)\in{}W^{1,p}(\Omega)$ for each $1\le{}p<2$ and
    since
    $W^{1,p}(\Omega)\hookrightarrow{}W^{1-\frac{1}{p},p}(\partial\Omega)
    \hookrightarrow{}
    L^{\frac{p}{2-p}}(\partial\Omega)$ for all $1\le{}p<2$ then for
$p=\frac{8}{5}$ we have
\begin{align*}
    \biggl|\int_{\psi_{i,j}(\mathcal{B}_{i,j,\varepsilon})}\!{}
    (|u|^{2}-1)^{2}\biggr|
    &\le(\mathcal{H}^{1}(\psi_{i,j}(\mathcal{B}_{i,j,\varepsilon})))
    ^{\frac{1}{2}}
    \biggl\||u_{\varepsilon_{k}}|^{2}-1\biggr\|_{L^{4}(\partial\Omega)}
    ^{2}\\
    &\le{}C(\Omega)
    (\mathcal{H}^{1}(\psi_{i,j}(\mathcal{B}_{i,j,\varepsilon})))
    ^{\frac{1}{2}}
    \mathopen{}\left|\log(\varepsilon)\right|\mathclose{}
    ^{2}.
\end{align*}
By Lemma \ref{Containment} and the fact that $\psi_{i,j}$ is Lipschitz we now
have
\begin{equation}\label{eq:BadPieceEstimate}
    \biggl|\int_{\psi_{i,j}(\mathcal{B}_{i,j,\varepsilon})}\!{}
    (|u|^{2}-1)^{2}\biggr|
    \le{}C(\Omega)
    \varepsilon^{\frac{1}{8}}
    \mathopen{}\left|\log(\varepsilon)\right|\mathclose{}^{2}.
\end{equation}
Putting together \eqref{eq:GoodPieceEstimate} and \eqref{eq:BadPieceEstimate}
gives the result.
\end{proof}

\section{Proof of Theorem \ref{ZerothOrderNormal}}\label{Goal}
\subsection{Proof of Compactness}
\par{}We proceed using the idea of extension by reflection in order to make use
of the compactness result from \cite{JS} on a slightly larger open set.
This approach is a modification of the ideas found in Propositions $3.1$ and $3.2$
of \cite{JMS} (see also Theorem $6.1$ of \cite{AlPo}).\\

\subsubsection{Step \texorpdfstring{$1$}{}:}\label{DualConvergence}

We use the $C^{1,1}$-coordinates
$\{(\mathcal{U}_{i,j},\psi_{i,j})\}
_{i=0,1,\ldots,b}^{j=1,2,\ldots,N_{i}}\cup
\{(\mathcal{U}_{0,0},\psi_{0,0})\}$
introduced in Section \ref{TangentNormalCoordinates} which extend to the
boundary.
We also recall that in Section \ref{TangentNormalCoordinates} we introduced
a partition of unity
$\{\widetilde{\rho}_{i,j}\}_{i=0,1,\ldots,b}^{j=1,2,\ldots,N_{i}}\cup
\{\widetilde{\rho}_{0,0}\}$
subordinate to the extended open cover.
Next, for $0\le{}i\le{}b$ with $1\le{}j\le{}N_{i}$ or $i=0$ with $j=0$,
and $k\ge1$, we use the functions $z_{i,j,\varepsilon_{k}}$ and
$w_{i,j,\varepsilon_{k}}$, defined in \eqref{def:zRep} and \eqref{def:wRep}.
Since
$w_{i,j,\varepsilon_{k}}^{1}(y_{1},0)=z_{i,j,\varepsilon_{k},\tau}(y_{1},0)=0$
for all $y_{1}\in(-r_{1},r_{1})$ and since this is a function defined on
$B_{r_{1},+}(0)$ then the extension results of Lemma \ref{Extension} apply.
After extending $w_{i,j,\varepsilon_{k}}$ by reflection,
perhaps passing to a subsequence, and applying Lemma \ref{ExtensionCompactness},
we may assume that
\begin{equation}\label{LocalFlatConvergence}
    \Biggl\lVert\star{}J(w_{i,j,\varepsilon_{k}})-
    \pi\sum_{p=1}^{M_{i,j,1}}d_{p}\delta_{y_{p}}
    -\frac{\pi}{2}\sum_{\ell=1}^{M_{i,j,2}}d_{\ell}\delta_{y_{\ell}}
    \Biggr\rVert_{\mathcal{A}_{\alpha,r_{1}}^{*}}\longrightarrow0^{+}
\end{equation}
for each $i=0,1,2,\ldots,b$ with $j=1,2,\ldots,N_{i}$ and each $0<\alpha\le1$.
By using Theorem $3.1$ of \cite{JS} and perhaps passing to a further subsequence
we may assume
\begin{equation}\label{InteriorFlatConvergence}
    \Biggl\lVert\star{}J(u_{\varepsilon_{k}})-
    \pi\sum_{l=1}^{M_{0,0,1}}d_{l}\delta_{x_{l}}
    \Biggr\rVert_{(C_{c}^{0,\alpha}(\mathcal{U}_{0,0}))^{*}}
    \longrightarrow0^{+}
\end{equation}
for all $0<\alpha\le1$.
We set
\begin{align*}
    V_{\Omega}&\coloneqq\bigcup_{\substack{i=0,1,\ldots,b\\ j=1,2,\ldots,N_{i}}}\bigcup_{p=1}^{M_{i,j,1}}\{\psi_{i,j}(y_{p})\}
    \cup\bigcup_{l=1}^{M_{0,0,1}}\{x_{l}\},\\
    V_{\partial\Omega}&\coloneqq\bigcup_{\substack{
    i=0,1,\ldots,b\\ j=1,2,\ldots,N_{i}}}
    \bigcup_{\ell=1}^{M_{i,j,2}}\{\psi_{i,j}(y_{\ell})\}
\end{align*}
which denote, respectively, the collection of interior and boundary vortices.
Note that there may have been duplicates due to overlapping charts but
$V_{\Omega}$ and $V_{\partial\Omega}$ consider such points only once.
Next, we set
\begin{equation*}
    J_{*}\coloneqq\pi\sum_{x\in{}V_{\Omega}}d_{x}\delta_{x}
    +\frac{\pi}{2}\sum_{x\in{}V_{\partial\Omega}}d_{x}\delta_{x}
\end{equation*}
where $d_{x}$ denotes the respective non-zero integer corresponding to
$x\in{}V_{\Omega}\sqcup{}V_{\partial\Omega}$.
We now show that $\star{}J(u_{\varepsilon_{k}})$ converges to $J_{*}$ in
$(C^{0,\alpha}(\Omega))^{*}$ for all $0<\alpha\le1$.
Let $\varphi\in{}C^{0,\alpha}(\Omega)$ and notice that since
\begin{equation*}
    \sum_{i=0}^{b}\sum_{j=1}^{N_{i}}\widetilde{\rho}_{i,j}
    +\widetilde{\rho}_{0,0}=1
    \hspace{15pt}\text{on }\overline{\Omega}
\end{equation*}
then
\begin{equation*}
    \bigl<\star{}J(u_{\varepsilon_{k}})-J_{*},\varphi\bigr>
    =\sum_{i=0}^{b}\sum_{j=1}^{N_{i}}
    \bigl<\star{}J(u_{\varepsilon_{k}})-J_{*},\widetilde{\rho}_{i,j}\varphi\bigr>
    +\bigl<\star{}J(u_{\varepsilon_{k}})-J_{*},\widetilde{\rho}_{0,0}\varphi\bigr>.
\end{equation*}
Since $\mathcal{U}_{0,0}\subset\subset\Omega$
\begin{align*}
    \bigl<\star{}J(u_{\varepsilon_{k}})-J_{*},\widetilde{\rho}_{0,0}\varphi\bigr>
    &=\biggl<\star{}J(u_{\varepsilon_{k}})
    -\pi\sum_{l=1}^{M_{0,0,1}}d_{\ell}\delta_{x_{\ell}},
    \widetilde{\rho}_{0,0}\varphi\biggr>\\
    &\le
    C(\Omega)
    \lVert\varphi\rVert_{C^{0,\alpha}(\Omega)}
    \Biggl\lVert\star{}J(u_{\varepsilon_{k}})-
    \pi\sum_{l=1}^{M_{0,0,1}}d_{l}\delta_{x_{l}}\Biggr\rVert_
    {(C_{c}^{0,\alpha}(\mathcal{U}_{0,0}))^{*}}
\end{align*}
where the last term tends to zero due to \eqref{InteriorFlatConvergence}.
For $0\le{}i\le{}b$ and $1\le{}j\le{}N_{i}$ we notice that since
$\mathcal{U}_{i,j}\cap(V_{\Omega}\sqcup{}V_{\partial\Omega})
=\biggl(\bigcup\limits_{p=1}^{M_{i,j,1}}\{\psi_{i,j}(y_{p})\}\biggr)\sqcup
\biggl(\bigcup\limits_{\ell=1}^{M_{i,j,2}}\{\psi_{i,j}(y_{\ell})\}\biggr)$ then
\begin{equation*}
    \bigl<\star{}J(u_{\varepsilon_{k}})-J_{*},\widetilde{\rho}_{i,j}\varphi\bigr>
    =\biggl<\star{}J(u_{\varepsilon_{k}})
    -\pi\sum_{p=1}^{M_{i,j,1}}d_{p}\delta_{\psi_{i,j}(y_{p})}
    -\frac{\pi}{2}\sum_{\ell=1}^{M_{i,j,2}}d_{\ell}\delta_{\psi_{i,j}(y_{\ell})},
    \widetilde{\rho}_{i,j}\varphi\biggr>.
\end{equation*}
Notice that
\begin{align*}
    \pi\sum_{p=1}^{M_{i,j,1}}d_{p}\delta_{\psi_{i,j}(y_{p})}
    +\frac{\pi}{2}\sum_{\ell=1}^{M_{i,j,2}}d_{\ell}\delta_{\psi_{i,j}(y_{\ell})}
    &=(\psi_{i,j})_{\#}\biggl(
    \pi\sum_{p=1}^{M_{i,j,1}}d_{p}\delta_{y_{p}}
    +\frac{\pi}{2}\sum_{\ell=1}^{M_{i,j,2}}d_{\ell}\delta_{y_{\ell}}\biggr),\\
    \star{}J(u_{\varepsilon_{k}})
    &=(\psi_{i,j})_{\#}\biggl(\star{}J(z_{i,j,\varepsilon_{k}})\biggr).
\end{align*}
Thus,
\begin{align*}
    &\biggl<\star{}J(u_{\varepsilon_{k}})
    -\pi\sum_{p=1}^{M_{i,j,1}}d_{p}\delta_{\psi_{i,j}(y_{p})}
    -\frac{\pi}{2}\sum_{\ell=1}^{M_{i,j,2}}d_{\ell}\delta_{\psi_{i,j}(y_{\ell})},
    \widetilde{\rho}_{i,j}\varphi\biggr>\\
    =&
    \biggl<\star{}J(z_{i,j,\varepsilon_{k}})
    -\pi\sum_{p=1}^{M_{i,j,1}}d_{p}\delta_{y_{p}}
    -\frac{\pi}{2}\sum_{\ell=1}^{M_{i,j,2}}d_{\ell}\delta_{y_{\ell}},
    (\widetilde{\rho}_{i,j}\varphi)\circ\psi_{i,j}\biggr>.
\end{align*}
Next we show, for each $i=0,1,\ldots,b$ and $j=1,2,\ldots,N_{i}$, that
\begin{equation}\label{Replacement}
    \lVert\star{}J(z_{i,j,\varepsilon_{k}})-\star{}J(w_{i,j,\varepsilon_{k}})
    \rVert_{\mathcal{A}_{\alpha,r_{1}}^{*}}\longrightarrow0^{+}
\end{equation}
as $k\to\infty$ for each $0<\alpha\le1$.
Observe that we can express $u_{\varepsilon_{k}}$ in
$\mathcal{U}_{i,j}$ as
\begin{equation*}
    u_{\varepsilon_{k}}(x)=u_{\varepsilon_{k}\tau}(x)
    \tau_{i}\bigl((\psi_{i,j}^{-1})^{1}(x)\bigr)
    +u_{\varepsilon_{k},\nu}(x)\nu_{i}\bigl((\psi_{i,j}^{-1})^{1}(x)\bigr)
\end{equation*}
where $u_{\varepsilon_{k},\tau}(x)$ and $u_{\varepsilon_{k},\nu}(x)$ are the
projections of $u_{\varepsilon_{k}}(x)$ onto the basis
$\{\tau_{i}\circ(\psi_{i,j}^{-1})^{1},\nu_{i}\circ(\psi_{i,j}^{-1})^{1}\}$
at the point $x$.
Next, since our coordinates $\psi_{i,j}$ are chosen so that
$\psi_{i,j}(y)=x$ then we have
\begin{equation}\label{ZRepresentation}
    z_{i,j,\varepsilon_{k}}
    =u_{\varepsilon_{k}}(\psi_{i,j}(y))
    =u_{\varepsilon_{k},\tau}(\psi_{i,j}(y))\tau_{i}(y_{1})
    +u_{\varepsilon_{k},\nu}(\psi_{i,j}(y))\nu_{i}(y_{1}).
\end{equation}
Taking partial derivatives of \eqref{ZRepresentation} and using
\eqref{DerivativeIdentities} gives
\begin{equation*}
    \begin{aligned}
    Jz_{i,j,\varepsilon_{k}}(y)=Jw_{i,j,\varepsilon_{k}}(y)
    &-\kappa_{i}(y_{1})z_{i,j,\varepsilon_{k},\tau}(y)
    \frac{\partial{}z_{i,j,\varepsilon_{k},\tau}}{\partial{}y_{2}}(y)\\
    &-\kappa_{i}(y_{1})z_{i,j,\varepsilon_{k},\nu}(y)
    \frac{\partial{}z_{i,j,\varepsilon_{k},\nu}}{\partial{}y_{2}}(y).
    \end{aligned}
\end{equation*}
which can be rewritten as:
\begin{align}
    Jz_{i,j,\varepsilon_{k}}(y)&=Jw_{i,j,\varepsilon_{k}}(y)
    -\kappa_{i}(y_{1})\frac{\partial}{\partial{}y_{2}}
    \biggl(\frac{(z_{i,j,\varepsilon_{k},\tau})^{2}
    +(z_{i,j,\varepsilon_{k},\nu})^{2}}{2}\biggr)
    \nonumber\\
    &=Jw_{i,j,\varepsilon_{k}}(y)
    -\kappa_{i}(y_{1})\frac{\partial}{\partial{}y_{2}}
    \biggl(\frac{|z_{i,j,\varepsilon_{k}}|^{2}}{2}\biggr).
    \label{JzRewrite}
\end{align}
First note that for $\phi\in\mathcal{A}_{1,r_{1}}$
\begin{align*}
    \int_{B_{r_{1},+}(0)}\!{}\phi(y)\bigl[Jz_{i,j,\varepsilon_{k}}(y)
    -Jw_{i,j,\varepsilon_{k}}(y)\bigr]
    &=-\int_{B_{r_{1},+}(0)}\!{}\phi(y)\kappa_{i}(y_{1})
    \frac{\partial}{\partial{}y_{2}}
    \biggl(\frac{|z_{i,j,\varepsilon_{k}}|^{2}}{2}\biggr)\\
    &=-\int_{B_{r_{1},+}(0)}\!{}\phi(y)\kappa_{i}(y_{1})
    \frac{\partial}{\partial{}y_{2}}
    \biggl(\frac{|z_{i,j,\varepsilon_{k}}|^{2}-1}{2}\biggr)\\
    &=(A).
\end{align*}
Integrating by parts gives
\begin{align*}
    (A)&=\int_{B_{r_{1},+}(0)}\!{}\kappa_{i}(y_{1})
    \frac{\partial\phi}{\partial{}y_{2}}(y)
    \cdot\frac{|z_{i,j,\varepsilon_{k}}|^{2}-1}{2}\\
    &-\int_{\partial{}B_{r_{1},+}(0)\cap\{y_{2}=0\}}\!{}
    \phi(y)\kappa_{i}(y_{1})
    \cdot\frac{|z_{i,j,\varepsilon_{k}}|^{2}-1}{2}\\
    &=(AA)+(AB).
\end{align*}
Next observe that since $\phi=0$ in a neighbourhood of
$\{(x,0):|x|<r_{1}\}$ we may estimate $(AA)$ as
\begin{align*}
    |(AA)|&\le
    C(\Omega)\lVert\phi\rVert_{\mathcal{A}_{1,r_{1}}}
    \biggl(\int_{B_{r_{1},+}(0)}\!{}
    \frac{(|z_{i,j,\varepsilon_{k}}|^{2}-1)^{2}}{4}\biggr)^{\frac{1}{2}}\\
    &=
    C(\Omega)\lVert\phi\rVert_{\mathcal{A}_{1,r_{1}}}
    \biggl(\int_{\psi_{i,j}(B_{r_{1},+}(0))}\!{}
    \frac{(|u_{\varepsilon_{k}}|^{2}-1)^{2}}{4}
    |J\psi_{i,j}^{-1}|\biggr)^{\frac{1}{2}}\\
    &\le
    C(\Omega)\lVert\phi\rVert_{\mathcal{A}_{1,r_{1}}}
    \varepsilon_{k}E_{\varepsilon_{k}}(u_{\varepsilon_{k}},\Omega)^{\frac{1}{2}}.
\end{align*}
To estimate $(AB)$ first note
\begin{equation*}
    |(AB)|
    \le{}
    C(\Omega)\lVert\phi\rVert_{\mathcal{A}_{1,r_{1}}}
    \int_{\partial{}B_{r_{1},+}(0)\cap\{y_{2}=0\}}\!{}
    \frac{\bigl||z_{i,j,\varepsilon_{k}}|^{2}-1\bigr|}{2}
\end{equation*}
and then we apply Lemma \ref{lem:SmallBoundaryError} after a coordinate
change.
The above then shows that
\begin{equation*}
    \lVert{}Jz_{i,j,\varepsilon_{k}}-Jw_{i,j,\varepsilon_{k}}\rVert
    _{\mathcal{A}_{1,r_{1}}^{*}}
    \le{}C(\Omega)\varepsilon_{k}^{\frac{1}{8}}.
\end{equation*}
Next observe that for $\phi\in{}C(B_{r_{1},+}(0))$ we have
\begin{equation*}
    |(A)|\le{}C(\Omega)\lVert\phi\rVert_{L^{\infty}(B_{r_{1},+}(0))}
    \bigl[\varepsilon_{k}E_{\varepsilon_{k}}(u_{\varepsilon_{k}},\Omega)
    +E_{\varepsilon_{k}}(u_{\varepsilon_{k}},\Omega)\bigr]
\end{equation*}
and hence
\begin{equation*}
    \lVert{}Jz_{i,j,\varepsilon_{k}}-Jw_{i,j,\varepsilon_{k}}\rVert{}_{(C(B_{r_{1},+}(0)))^{*}}
    \le{}C(\Omega)\mathopen{}\left|\log(\varepsilon_{k})\right|\mathclose{}.
\end{equation*}
By Lemma \ref{InterpolationEstimate}
we have for each $0<\alpha<1$ that
\begin{equation*}
    \lVert{}Jz_{i,j,\varepsilon_{k}}-Jw_{i,j,\varepsilon_{k}}\rVert
    _{\mathcal{A}_{\alpha,r_{1}}^{*}}
    \le{}C(\Omega)\varepsilon_{k}^{\frac{\alpha}{8(1+\alpha)}}
    \mathopen{}\left|\log(\varepsilon_{k})\right|\mathclose{}^
    {\frac{1}{1+\alpha}}.
\end{equation*}
Since $(\widetilde{\rho}_{i,j}\varphi)\circ\psi_{i,j}\in\mathcal{A}_{\alpha,r_{1}}$
then from the above
estimates we conclude \eqref{Replacement} by letting $k\to\infty$.
Thus, combined with the above it suffices to estimate
\begin{equation*}
    \biggl<\star{}J(w_{i,j,\varepsilon_{k}})
    -\pi\sum_{p=1}^{M_{i,j,1}}d_{p}\delta_{y_{p}}
    -\frac{\pi}{2}\sum_{\ell=1}^{M_{i,j,2}}d_{\ell}\delta_{y_{\ell}},
    (\widetilde{\rho}_{i,j}\varphi)\circ\psi_{i,j}\biggr>.
\end{equation*}
But this tends to zero due to \eqref{LocalFlatConvergence} since
$(\widetilde{\rho}_{i,j}\varphi)\circ\psi_{i,j}\in\mathcal{A}_{\alpha,r_{1}}$.\\

\subsubsection{Step \texorpdfstring{$2$}{}:}
\label{EulerCharacteristicCondition}
\par{}Now we show that if $\{u_{\varepsilon_{k}}\}_{k=1}^{\infty}$
denotes a convergent subsequence from Step $1$ and we have
\begin{equation}\label{eq:FlatConvergence}
    \Biggl\lVert\star{}J(u_{\varepsilon_{k}})
    -\pi\sum_{i=1}^{M_{1}}d_{i}\delta_{a_{i}}
    -\frac{\pi}{2}\sum_{j=0}^{b}
    \sum_{\ell=1}^{M_{2,j}}d_{j\ell}\delta_{c_{j\ell}}
    \Biggr\rVert_{(C^{0,\alpha}(\Omega))^{*}}\longrightarrow0^{+},
\end{equation}
for each $0<\alpha\le1$, where $a_{i}\in\Omega$ for $1\le{}i\le{}M_{1}$,
$c_{j\ell}\in\partial\Omega$ for $0\le{}j\le{}b$, $1\le{}\ell\le{}M_{2,j}$, and
$d_{i},d_{j\ell}\in\mathbb{Z}\setminus\{0\}$ then
\begin{equation*}
    \sum_{i=1}^{M_{1}}d_{i}
    +\frac{1}{2}\sum_{j=0}^{b}\sum_{\ell=1}^{M_{2,j}}d_{j\ell}=
    \chi_{Euler}(\Omega).
\end{equation*}

\noindent\underline{\bf{Substep 1}}\vspace{5pt}
\par{}Before we are able to demonstrate our desired result we first show that
\begin{equation*}
    \biggl|\int_{\partial\Omega}\!{}(ju_{\varepsilon_{k}}\cdot\tau)\varphi
    -\sum_{i=0}^{b}
    \int_{(\partial\Omega)_{i}}\!{}
    \kappa_{i}\varphi\biggr|\longrightarrow0^{+}
\end{equation*}
as $k\to\infty$ where $\varphi\in{}C^{0,\alpha}(\Omega)$ for $0<\alpha\le1$.
We use the atlas
$\{(\widetilde{\mathcal{U}}_{i,j},\widetilde{\psi}_{i,j})\}_{i=0,1,\ldots,b}
^{j=1,2,\ldots,N_{i}}$
from Section \ref{TangentNormalCoordinates}
to form a coordinate system for $\partial\Omega$.
We also use the partition of unity
$\{\widetilde{\rho}_{i,j}\}_{i=0,1,\ldots,b}^{j=1,2,\ldots,N_{i}}$ from Section
\ref{TangentNormalCoordinates}.
Specifically, we notice that $\partial\Omega\cap\widetilde{\mathcal{U}}_{i,j}$
is the image of the map
\begin{equation*}
    \gamma_{i,j}\colon{}(-r_{1},r_{1})\ni{}s
    \mapsto\widetilde{\psi}_{i,j}(s,0)=\gamma_{i}(s)
\end{equation*}
where $\gamma_{i}$, and hence $\gamma_{i,j}$, is an arclength parametrized curve.
We can then write
\begin{align*}
    \int_{\partial\Omega}\!{}(ju_{\varepsilon_{k}}\cdot\tau)\varphi
    =\sum_{i=0}^{b}\sum_{j=1}^{N_{i}}
    \int_{(\partial\Omega)_{i}\cap\widetilde{\mathcal{U}}_{i,j}}\!{}
    \Bigl(u_{\varepsilon_{k}}\times
    \frac{\partial{}u_{\varepsilon_{k}}}{\partial\tau}\Bigr)
    \widetilde{\rho}_{i,j}\varphi.
\end{align*}
Observe that since $(u_{\varepsilon_{k}})_{T}=0$ then
$u_{\varepsilon_{k}}=u_{\varepsilon_{k},\nu}\nu_{i}\circ(\widetilde{\psi}_{i,j}^{-1})^{1}$ along $(\partial\Omega)_{i}$ for each $i=0,1,\ldots,b$.
By \eqref{DerivativeIdentities}
we have using local coordinates that
\begin{equation*}
    \frac{\partial{}u_{\varepsilon_{k}}}{\partial\tau}
    =\frac{\partial{}z_{i,j,\varepsilon_{k},\nu}}{\partial{}y_{1}}(y_{1},0)
    \nu_{i}(y_{1})-
    z_{i,j,\varepsilon_{k},\nu}(y_{1},0)\kappa_{i}(y_{1})\tau_{i}(y_{1}).
\end{equation*}
Thus,
\begin{equation*}
    u_{\varepsilon_{k}}\times
    \frac{\partial{}u_{\varepsilon_{k}}}{\partial{}\tau}
    =|z_{i,j,\varepsilon_{k}}|^{2}(y_{1},0)\kappa_{i}(y_{1})
\end{equation*}
using that $z_{i,j,\varepsilon_{k},\tau}=0$ in the last equality.
From this we conclude that
\begin{align*}
    \int_{(\partial\Omega)_{i}\cap\widetilde{\mathcal{U}}_{i,j}}\!{}
    \Bigl(u_{\varepsilon_{k}}\times
    \frac{\partial{}u_{\varepsilon_{k}}}{\partial\tau}\Bigr)
    \widetilde{\rho}_{i,j}\varphi
    &=\int_{(\partial\Omega)_{i}\cap\widetilde{\mathcal{U}}_{i,j}}\!{}
    |u_{\varepsilon_{k}}|^{2}\widetilde{\kappa}_{i}
    \widetilde{\rho}_{i,j}\varphi\\
    &=\int_{(\partial\Omega)_{i}\cap\widetilde{\mathcal{U}}_{i,j}}\!{}
    \bigl(|u_{\varepsilon_{k}}|^{2}-1\bigr)
    \widetilde{\kappa}_{i}\widetilde{\rho}_{i,j}\varphi
    +\int_{\partial\Omega\cap\widetilde{\mathcal{U}}_{i,j}}\!{}
    \widetilde{\kappa}_{i}\widetilde{\rho}_{i,j}\varphi
\end{align*}
where $\widetilde{\kappa}_{i}=\kappa_{i}\circ(\widetilde{\psi}_{i,j}^{-1})^{1}$.
Summing in $j$ and using \eqref{DisjointPartitionOfUnity} we now have
\begin{align*}
    \int_{(\partial\Omega)_{i}}\!{}(ju_{\varepsilon_{k}}\cdot\tau)\varphi
    =\sum_{j=1}^{N_{i}}\int_{(\partial\Omega)_{i}\cap\widetilde{\mathcal{U}}_{i,j}}
    \!{}
    \bigl(|u_{\varepsilon_{k}}|^{2}-1\bigr)\widetilde{\kappa}_{i}\widetilde{\rho}_{i,j}\varphi
    +\int_{(\partial\Omega)_{i}}\!{}\widetilde{\kappa}_{i}\varphi.
\end{align*}
Now summing in $i$, with consideration to orientation,
and using that $(\partial\Omega)_{i}$ for $i=0,1,\ldots,b$
are disjoint we have
\begin{align*}
    \int_{\partial\Omega}\!{}(ju_{\varepsilon_{k}}\cdot\tau)\varphi
    &=
    \sum_{i=0}^{b}\sum_{j=1}^{N_{i}}
    \int_{(\partial\Omega)_{i}\cap\widetilde{\mathcal{U}}_{i,j}}\!{}
    \bigl(|u_{\varepsilon_{k}}|^{2}-1\bigr)\widetilde{\kappa}_{i}\widetilde{\rho}_{i,j}\varphi\\
    &+\int_{(\partial\Omega)_{0}}\!{}\widetilde{\kappa}_{0}\varphi
    +\sum_{i=1}^{b}\int_{(\partial\Omega)_{i}}\!{}\widetilde{\kappa}_{i}\varphi.
\end{align*}
To prove the desired result we note that
$L^{2}(\partial\Omega)\subseteq(C^{0,\alpha}(\partial\Omega))^{*}$
for each $0<\alpha\le1$ and apply Lemma \ref{lem:SmallBoundaryError} to
\begin{equation*}
    \sum_{i=0}^{b}\sum_{j=1}^{N_{i}}
    \int_{(\partial\Omega)_{i}\cap\widetilde{\mathcal{U}}_{i,j}}\!{}
    \bigl(|u_{\varepsilon_{k}}|^{2}-1\bigr)\widetilde{\kappa}_{i}\widetilde{\rho}_{i,j}\varphi.
\end{equation*}

\noindent\underline{\bf{Substep 2}}\vspace{5pt}
\par{}Now we show the desired conclusion.
Observe that integrating by parts gives
\begin{equation*}
    0=-\frac{1}{2}\int_{\Omega}\!{}ju_{\varepsilon_{k}}\cdot\nabla^{\perp}(1)
    =-\frac{1}{2}\int_{\partial\Omega}\!{}(ju_{\varepsilon_{k}}\cdot\tau)
    +\int_{\Omega}\!{}Ju_{\varepsilon_{k}}.
\end{equation*}
Since $\phi\equiv1$ is a member of $C^{0,\alpha}(\Omega)$ for each $0<\alpha\le1$ then the above work
shows that as $k\to\infty$ we have by the Gauss-Bonnet theorem, see
\cite{DoC}\footnote{Here, we use the regular region $R\subseteq\mathbb{R}^{3}$
given by $\iota(\Omega)$ where $\iota\colon\Omega\to{}R$ is given by
$\iota(x)=(x,0)$.}, that
\begin{equation*}
    \pi\sum_{i=1}^{M_{1}}d_{i}
    +\frac{\pi}{2}\sum_{j=1}^{b}\sum_{\ell=1}^{M_{2,j}}d_{j\ell}
    =\frac{1}{2}\int_{\partial\Omega}\!{}\widetilde{\kappa}
    =\pi\chi_{Euler}(\overline{\Omega})
    =\pi\chi_{Euler}(\Omega)
\end{equation*}
using that the Gaussian curvature of $\Omega$ is zero.
\subsubsection{Step \texorpdfstring{$3$}{}:}
\par{}Finally,
we show that for each $j=0,1,\ldots,b$ we have
\begin{equation*}
    \frac{1}{2}\sum_{\ell=1}^{M_{2,j}}d_{j\ell}\in\mathbb{Z}
\end{equation*}
where $d_{j\ell}$ is the degree associated to
$c_{j\ell}\in(\partial\Omega)_{j}$.
To do this, we use Lipschitz test functions and the convergence in
H\"{o}lder dual spaces to isolate a level set of the distance
to a boundary component of $\Omega$ for which the degree is defined and relate
it to one half of the sum of boundary degrees along this component.
Specifically, we proceed as follows:
\begin{enumerate}
    \item
        Following ideas from \cite{CJ} we first make use of a carefully chosen
        test function which allows us to use the dual convergence of the
        Jacobian in order to argue that on most level sets,
        $(\partial\Omega)_{j,t}$, of the distance to a boundary component we have
        that
        \begin{equation*}
            \frac{1}{2}\int_{(\partial\Omega)_{j,t}}\!{}ju\cdot{}t_{j}
            =\frac{\pi}{2}\sum_{\ell=1}^{M_{2,j}}d_{j\ell}+\pi+o(1).
        \end{equation*}
    \item
        By making use of a separate test function we can now use the work of
        \cite{JS} to show that for most level sets $(\partial\Omega)_{j,t}$
        we have
        \begin{equation*}
            \frac{1}{2}\int_{(\partial\Omega)_{j,t}}\!{}ju\cdot{}t_{j}
            =\pi\text{deg}(u,(\partial\Omega)_{j,t})+o(1).
        \end{equation*}
        Combined with the previous statement we obtain the desired
        statement for $\varepsilon>0$ sufficiently small.
\end{enumerate}

We now prove the desired statement.
We first show that
\begin{align}\label{eq:SlicingConvergence}
    \int_{0}^{r_{1}}\!
    \biggl|\frac{1}{2}\int_{(\partial\Omega)_{j,t}}\!{}ju\cdot{}t_{j}-
    \frac{\pi}{2}\sum_{k=1}^{M_{2,j}}d_{jk}-\pi\biggr|\mathrm{d}t
    &\le{}
    \left\|Ju-\pi\sum_{i=1}^{M_{1}}d_{i}\delta_{a_{i}}
    -\frac{\pi}{2}\sum_{j=0}^{b}\sum_{k=1}^{M_{2,j}}d_{jk}\delta_{c_{jk}}
    \right\|_{(C^{0,\alpha}(\Omega))^{*}}\nonumber\\
    &+\frac{1}{2}\biggl|
    \int_{(\partial\Omega)_{j}}\!{}ju\cdot{}\tau_{j}
    +2\pi\biggr|\nonumber\\
    &\eqqcolon{}\sigma(\varepsilon)
\end{align}
where $t_{j}$ is the tangent vector to $(\partial\Omega)_{j,t}$ so that
$\{\nu_{j},t_{j}\}$ is positively oriented.
As a result, we may conclude that for a subset of the levels sets of the distance
to a boundary component of size $r_{1}-\sqrt{s(\varepsilon)}$ we have
\begin{equation}\label{eq:MainObservation}
    \biggl|\frac{1}{2}\int_{(\partial\Omega)_{j,t}}\!{}ju\cdot{}t_{j}-
    \frac{\pi}{2}\sum_{k=1}^{M_{2,j}}d_{jk}-\pi\biggr|\le{}\sqrt{\sigma(\varepsilon)}
\end{equation}
and, hence, that ``most" slices have $\frac{1}{2}\int_{(\partial\Omega)_{j,t}}\!{}
ju\cdot{}t_{j}$ close to $\frac{\pi}{2}\sum\limits_{k=1}^{M_{2,j}}d_{jk}+\pi$.
To do this we use a technique similar to the one found in Lemma $1$ of \cite{CJ}.
We let
\begin{equation*}
    \varphi(x)\coloneqq\int_{\text{dist}(x,(\partial\Omega)_{j})}^{r_{1}}\!{}
    g(s)\mathrm{d}s,\hspace{15pt}
    g(s)\coloneqq
    \chi_{[0,r_{1}]}(s)\text{sgn}
    \biggl(\frac{1}{2}\int_{(\partial\Omega)_{j,s}}\!{}ju\cdot{}t_{j}-
    \frac{\pi}{2}\sum_{k=1}^{M_{2,j}}d_{jk}-\pi\biggr)
\end{equation*}
and notice that $\text{supp}(\varphi)\subseteq\overline{\Omega}_{j,r_{1}}$.
By Lemma \ref{Containment} and Lemma \ref{lem:SmallBoundaryError} we have,
using the partition of unity $\{\widetilde{\rho}_{j\ell}\}_{\ell=1}^{N_{j}}$
subordinate to $\{\widetilde{\mathcal{U}}_{j,\ell}\}_{\ell=1}^{N_{j}}$, that
\begin{equation*}
    \int_{0}^{r_{1}}\!{}
    \biggl|\frac{1}{2}\int_{(\partial\Omega)_{j,t}}\!{}ju\cdot{}t_{j}-
    \frac{\pi}{2}\sum_{k=1}^{M_{2,j}}d_{jk}-\pi\biggr|\mathrm{d}t
    =\int_{0}^{r_{1}}\!{}
    g(t)
    \biggl[\frac{1}{2}\int_{(\partial\Omega)_{j,t}}\!{}ju\cdot{}t_{j}-
    \frac{\pi}{2}\sum_{k=1}^{M_{2,j}}d_{jk}-\pi\biggr]\mathrm{d}t.
\end{equation*}
Next parametrizing $(\partial\Omega)_{j,t}\cap\mathcal{U}_{j,\ell}$ using
the mapping
$(-a_{j,\ell}^{t},a_{j,\ell}^{t})\ni{}y_{1}\mapsto{}\psi_{j,\ell}(y_{1},t)$,
we have
\begin{align*}
    &=\sum_{\ell=1}^{N_{j}}\int_{0}^{r_{1}}\!{}
    \frac{g(t)}{2}\int_{-a_{j\ell}^{t}}^{a_{j\ell}^{t}}\!{}(ju)(\psi_{j,\ell}(y_{1},t))\cdot
    t_{j}(\psi_{j,\ell}(y_{1},t))
    \widetilde{\rho}_{j,\ell}(\psi_{j,\ell}(y_{1},t))
    (1-t\kappa_{j}(y_{1}))\mathrm{d}y_{1}\mathrm{d}t\\
    &-
    \biggl(\frac{\pi}{2}\sum_{k=1}^{M_{2,j}}d_{jk}\biggr)
    \int_{0}^{r_{1}}\!{}g(t)\mathrm{d}t
    -\pi\int_{0}^{r_{1}}\!{}g(t)\mathrm{d}t.
\end{align*}
Notice that since $\varphi(c_{jk})=\int_{0}^{r_{1}}\!{}g(t)\mathrm{d}t$ for
each $k=1,2,\ldots,M_{2,j}$
\begin{equation*}
    -\biggl(\frac{\pi}{2}\sum_{k=1}^{M_{2,j}}d_{jk}\biggr)
    \int_{0}^{r_{1}}\!{}g(t)\mathrm{d}t
    -\pi\int_{0}^{r_{1}}\!{}g(t)\mathrm{d}t
    =-\frac{\pi}{2}\sum_{k=1}^{M_{2,j}}d_{jk}\varphi(c_{jk})
    -\pi\varphi\bigl((\partial\Omega)_{j}\bigr)
\end{equation*}
where $\varphi\bigl((\partial\Omega)_{j}\bigr)$ denotes the value of
$\varphi$ at any point along the boundary.
If $x\in\mathcal{U}_{j,\ell}$ then since
$\text{dist}(x,(\partial\Omega)_{j})=(\psi_{j,\ell}^{-1})^{2}(x)$ then we have
\begin{equation*}
    \nabla\varphi(x)=-g\bigl((\psi_{j,\ell}^{-1})^{2}(x)\bigr)
    \nabla(\psi_{j,\ell}^{-1})^{2}(x)
    =-g\bigl((\psi_{j,\ell}^{-1})^{2}(x)\bigr)
    \nu_{j}\bigl((\psi_{j,\ell}^{-1})^{1}(x)\bigr).
\end{equation*}
Hence, we have by the above and the Change of Variables Theorem that
\begin{align*}
    &\sum_{\ell=1}^{N_{j}}\int_{0}^{r_{1}}\!{}
    \frac{g(t)}{2}\int_{-a_{j\ell}^{t}}^{a_{j\ell}^{t}}\!{}
    t_{j}(y_{1})\cdot{}
    ju(\psi_{j,\ell}(y_{1},t))\widetilde{\rho}_{j,\ell}(\psi_{j,\ell}(y_{1},t))
    (1-t\kappa_{j}(y_{1}))
    \mathrm{d}y_{1}\mathrm{d}t\\
    &=\sum_{\ell=1}^{N_{j}}\frac{-1}{2}
    \int_{0}^{r_{1}}\int_{-a_{j\ell}^{t}}^{a_{j\ell}^{t}}\!{}
    (\nabla^{\perp}\varphi)(\psi_{j,\ell}(y_{1},t))\cdot{}ju(\psi_{j,\ell}(y_{1},t))
    \widetilde{\rho}_{j,\ell}(\psi_{j,\ell}(y_{1},t))(1-t\kappa_{j}(y_{1}))
    \mathrm{d}y_{1}\mathrm{d}t\\
    &=\sum_{\ell=1}^{N_{i}}\frac{-1}{2}\int_{\Omega_{j,r_{1}}}\!{}
    \nabla^{\perp}\varphi(x)\cdot{}ju(x)\widetilde{\rho}_{j,\ell}(x)\mathrm{d}x\\
    &=\frac{-1}{2}
    \int_{\Omega_{j,r_{1}}}\!{}\nabla^{\perp}\varphi(x)\cdot{}ju(x)
    \mathrm{d}x.
\end{align*}
Integrating by parts now gives that
\begin{align*}
    \frac{-1}{2}\int_{\Omega_{j,r_{1}}}\!{}\nabla^{\perp}\varphi(x)
    \cdot{}ju(x)\mathrm{d}x
    &=\frac{-1}{2}\int_{(\partial\Omega)_{j}}\!{}\varphi(x)[ju\cdot\tau_{j}]
    +\int_{\Omega_{j,r_{1}}}\!{}\varphi(x)\cdot{}Ju(x)\\
    &=\frac{-\varphi\bigl((\partial\Omega)_{j}\bigr)}{2}
    \int_{(\partial\Omega)_{j}}\!{}ju\cdot\tau_{j}
    +\int_{\Omega_{j,r_{1}}}\!{}\varphi(x)\cdot{}Ju(x).
\end{align*}
Rearranging our previous work we now have that
\begin{align*}
    \int_{0}^{r_{1}}\!\biggl|\frac{1}{2}\int_{(\partial\Omega)_{j,t}}\!{}
    ju\cdot{}t_{j}-\frac{\pi}{2}\sum_{k=1}^{M_{2,j}}d_{j,k}-\pi\biggr|
    \mathrm{d}t
    &=\biggl[\int_{\Omega_{j,r_{1}}}\!{}\varphi(x)\cdot{}Ju(x)
    -\frac{\pi}{2}\sum_{k=1}^{M_{2,j}}d_{jk}\varphi(c_{jk})\biggr]\\
    &-\frac{\varphi\bigl((\partial\Omega)_{j}\bigr)}{2}
    \biggl[\int_{(\partial\Omega)_{j}}\!{}ju\cdot{}\tau_{j}
    +2\pi\biggr].
\end{align*}
By the conclusion of Substep $1$ of Step $2$ as well as
\eqref{eq:FlatConvergence} we have \eqref{eq:SlicingConvergence}.\\

Next, we show that $\frac{1}{2}\int_{(\partial\Omega)_{j,t}}ju\cdot{}t_{j}$ is
close to an integer multiple of $\pi$ on most slices $(\partial\Omega)_{j,t}$.
To do this we apply the work of Jerrard and Soner from \cite{JS}.
We consider the test function
\begin{equation*}
    \varphi_{0}(x)\coloneqq
    \max\Bigl\{\frac{r_{1}}{2}
    -\text{dist}(x,(\partial\Omega)_{j,\frac{r_{1}}{2}}),0\Bigr\}.
\end{equation*}
From the proof of Theorem $2.1$ and equation 2.15 of \cite{JS} there is a set
$A\subseteq\bigl[0,\frac{r_{1}}{2}\bigr]$, taking $\lambda=\frac{3}{2}$, such
that
$\mathcal{L}^{1}\bigl(\bigl[0,\frac{r_{1}}{2}\bigr]\setminus{}A\bigr)\le{}C
\varepsilon^{\frac{1}{6}}$
and
\begin{equation*}
    \int_{A}\biggl|\frac{1}{2}\int_{(\partial\Omega)_{j,\frac{r_{1}}{2}\pm{}t}}
    \!{}ju\cdot{}t_{j}-
    \pi\text{deg}(u,(\partial\Omega)_{j,\frac{r_{1}}{2}\pm{}t})
    \biggr|\mathrm{d}t\le{}C\varepsilon\mathopen{}\left|\log(\varepsilon)\right|\mathclose{}.
\end{equation*}
We conclude that except on a set of measure
$C\varepsilon^{\frac{1}{6}}$ we have that
$\frac{1}{2}\int_{(\partial\Omega)_{j,t}}\!{}ju\cdot{}t_{j}$
is close to an integer multiple of $\pi$.
Combined with \eqref{eq:MainObservation} we conclude, for $\varepsilon>0$
sufficiently small, that
\begin{equation*}
    \frac{1}{2}\sum_{k=1}^{M_{2,j}}d_{jk}\in\mathbb{Z}.
\end{equation*}

\subsection{Proof of Lower Bound}
\par{}
Since $(C^{0,\alpha}(\Omega))^{*}\subseteq(C^{0,1}(\Omega))^{*}$
it suffices to consider the case when the Jacobians converge in
$(C^{0,1}(\Omega))^{*}$.
We may also suppose that
\begin{equation}\label{LogBound}
    E_{\varepsilon}(u_{\varepsilon},\Omega)
    \le{}\biggl[
    \pi\sum_{i=1}^{M_{1}}|d_{i}|+\frac{\pi}{2}\sum_{j=1}^{M_{2}}|d_{j}|
    \biggr]\mathopen{}\left|\log(\varepsilon)\right|\mathclose{}+1
\end{equation}
for all $\varepsilon\in(0,1]$ since if \eqref{LogBound} fails for a collection,
$\mathcal{C}$, of
$\varepsilon\in(0,1]$ then
\begin{align*}
    \liminf_{\varepsilon\to0^{+}}
    \frac{E_{\varepsilon}(u_{\varepsilon},\Omega)}
    {\mathopen{}\left|\log(\varepsilon)\right|\mathclose{}}
    &=\lim_{\varepsilon_{0}\to0^{+}}
    \inf_{0<\varepsilon\le\varepsilon_{0}}
    \frac{E_{\varepsilon}(u_{\varepsilon},\Omega)}
    {\mathopen{}\left|\log(\varepsilon)\right|\mathclose{}}\\
    &\ge\lim_{\varepsilon_{0}\to0^{+}}
    \min\biggl\{
    \inf_{\varepsilon\in(0,\varepsilon_{0}]\cap((0,1]\setminus\mathcal{C})}
    \frac{E_{\varepsilon}(u_{\varepsilon},\Omega)}
    {\mathopen{}\left|\log(\varepsilon)\right|\mathclose{}},
    \pi\sum_{i=1}^{M_{1}}|d_{i}|+\frac{\pi}{2}\sum_{j=1}^{M_{2}}|d_{j}|
    \biggr\}\\
    &=\min\biggl\{
    \liminf_{\varepsilon\in(0,1]\setminus\mathcal{C}}
    \frac{E_{\varepsilon}(u_{\varepsilon},\Omega)}
    {\mathopen{}\left|\log(\varepsilon)\right|\mathclose{}},
    \pi\sum_{i=1}^{M_{1}}|d_{i}|+\frac{\pi}{2}\sum_{j=1}^{M_{2}}|d_{j}|
    \biggr\}
\end{align*}
and hence it would suffice to prove the desired result among the
collection $(0,1]\setminus{}\mathcal{C}$ where \eqref{LogBound} holds.\\

\subsubsection{Step \texorpdfstring{$1$}{}:}
\par{}First we show that we can extend convergence to be in
$(C_{c}^{0,1}(\widetilde{\Omega}_{\frac{r_{1}}{2}}))^{*}$
where $\widetilde{\Omega}_{\frac{r_{1}}{2}}$ is a slightly
larger open set than $\Omega$ in $\mathbb{R}^{2}$.
The strategy in extending $u_{\varepsilon}$ will be by reflection as in
\cite{RaIg}.
More specifically, if
we let
\begin{equation*}
    \widetilde{\Omega}_{\frac{r_{1}}{2}}\coloneqq\Omega\cup
    \bigcup\limits_{i=0}^{b}\widetilde{\Omega}_{i,\frac{r_{1}}{2}},
\end{equation*}
where $\widetilde{\Omega}_{i,r}$ is defined in
\eqref{def:DomainiExt}, then we will show
\begin{equation*}
    \Bigl\lVert\star{}J(\widetilde{u}_{\varepsilon})-\widetilde{J}\Bigr\rVert
    _{(C_{c}^{0,1}(\widetilde{\Omega}_{\frac{r_{1}}{2}}))^{*}}
    \longrightarrow0^{+}.
\end{equation*}
where $\widetilde{u}_{\varepsilon}$ is an extension of $u_{\varepsilon}$ to
$\widetilde{\Omega}_{\frac{r_{1}}{2}}$ and $\widetilde{J}$ is an extension of
$J$ to $\widetilde{\Omega}_{\frac{r_{1}}{2}}$ as defined below.\\

We use the atlas,
$\{(\widetilde{\mathcal{U}}_{i,j},\widetilde{\psi}_{i,j})\}_{i=0,1,\ldots,b}
^{j=1,2,\ldots,N_{i}}\cup
\{(\widetilde{\mathcal{U}}_{0,0},\widetilde{\psi}_{0,0})\}$,
constructed in Section \ref{TangentNormalCoordinates}, which covers
$\widetilde{\Omega}_{\frac{r_{1}}{2}}$,
as well as the smooth partition of unity $\{\widetilde{\rho}_{i,j}\}_{i=0,1,\ldots,b}^{j=1,2,\ldots,N_{i}}
\cup\{\widetilde{\rho}_{0,0}\}$
subordinate to this cover.
We index the chart functions so that
$\text{supp}(\widetilde{\rho}_{i,j})\subseteq\widetilde{\mathcal{U}}_{i,j}$.
For each $i=0,1,\ldots,b$ and $j=1,2,\ldots,N_{i}$ we have
\begin{equation*}
    \text{supp}(J)\cap\widetilde{\mathcal{U}}_{i,j}=
    \biggl(\bigcup_{k=1}^{M_{i,j,1}}\{x_{k}^{i,j}\}\biggr)\cup
    \biggl(\bigcup_{\ell=1}^{M_{i,j,2}}\{x_{\ell}^{i,j}\}\biggr)
\end{equation*}
where $x_{k}^{i,j}\in\Omega$ for $k=1,2,\ldots,M_{i,j,1}$ and
$x_{\ell}^{i,j}\in(\partial\Omega)_{i}$
for $\ell=1,2,\ldots,M_{i,j,2}$.
We also have
\begin{equation*}
    \text{supp}(J)\cap\widetilde{\mathcal{U}}_{0,0}=
    \bigcup_{k=1}^{M_{0,0,1}}\{x_{k}^{0,0}\}.
\end{equation*}
For each $x_{k}^{i,j}$, where $k=1,2,\ldots,M_{i,j,1}$, we have
$x_{k}^{i,j}=\widetilde{\psi}_{i,j}(y_{k}^{i,j})$ for some
$y_{k}^{i,j}\coloneqq(y_{1,k}^{i,j},y_{2,k}^{i,j})
\in{}B_{\frac{r_{1}}{2},+}(0)$.
In addition, for each $x_{\ell}^{i,j}$, where $\ell=1,2,\ldots,M_{i,j,2}$,
we have $x_{\ell}^{i,j}=\widetilde{\psi}_{i,j}(y_{\ell}^{i,j})$ where
$y_{\ell}^{i,j}\coloneqq(y_{1,\ell}^{i,j},0)$.
For each $i=0,1,\ldots,b$ and $j=1,2,\ldots,N_{i}$ we let
$\bar{x}_{k}\in\widetilde{\Omega}_{\frac{r_{1}}{2}}$ denote
\begin{equation*}
    \bar{x}_{k}\coloneqq\widetilde{\psi}_{i,j}(\bar{y}_{k}^{i,j})
\end{equation*}
where $\bar{y}_{k}^{i,j}\coloneqq(y_{1,k}^{i,j},-y_{2,k}^{i,j})$.
Next we set
\begin{align*}
    V_{\Omega}&\coloneqq\biggl(\bigcup_{i=0}^{b}
    \bigcup_{j=1}^{N_{i}}\bigcup_{k=1}^{M_{i,j,1}}
    \{\widetilde{\psi}_{i,j}(y_{k}^{i,j})\}\biggr)
    \cup\biggl(\bigcup_{k=1}^{M_{0,0,1}}\{x_{k}\}\biggr)\\
    V_{\partial\Omega}&\coloneqq\bigcup_{i=0}^{b}\bigcup_{j=1}^{N_{i}}
    \bigcup_{\ell=1}^{M_{i,j,2}}
    \{\widetilde{\psi}_{i,j}(y_{1,\ell}^{i,j},0)\}\\
    V_{\widetilde{\Omega}_{\frac{r_{1}}{2}}\setminus\overline{\Omega}}&\coloneqq
    \bigcup_{i=0}^{b}\bigcup_{j=1}^{N_{i}}
    \bigcup_{k=1}^{M_{i,j,1}}
    \{\widetilde{\psi}_{i,j}(\bar{y}_{k}^{i,j})\}.
\end{align*}
\begin{center}
\begin{figure}[h!]
\begin{center}
\scalebox{1.4}{
\begin{tikzpicture}[use Hobby shortcut]  
\draw[rotate=-90,scale=2,postaction={decorate}] (3,0) .. +(1,0) .. +(1,2) .. +(1,3) .. 
+(0,3) .. (3,0);   
\draw[rotate=-90,dashed,scale=2,postaction={decorate}] (2.7,0) .. +(0.1,-0.1) .. +(0.2,-0.18)
.. +(0.3,-0.26) .. +(0.4,-0.32) .. +(0.55,-0.39) .. +(0.65,-0.43) .. +(0.72,-0.44) .. +(0.78,-0.445)
.. +(0.83,-0.447) .. +(0.88,-0.447) .. +(0.92,-0.445) .. +(0.96,-0.444) .. +(0.99,-0.441)
.. +(1.2,-0.36) .. +(1.5,-0.16) .. +(1.75,0.35) .. +(1.77,0.55) .. +(1.76,0.75) .. +(1.74,0.9)
.. +(1.69,1.1) .. +(1.64,1.3) .. +(1.59,1.5) .. +(1.55,1.65) .. +(1.53,1.8) .. +(1.55,2)
.. +(1.59,2.2) .. +(1.63,2.4) .. +(1.66,2.6) .. +(1.65,2.8) .. +(1.57,3) .. +(1.45,3.15)
.. +(1.34,3.25) .. +(1.23,3.31) .. +(1.12,3.35) .. +(1.03,3.37) .. +(0.94,3.38) .. +(0.85,3.38)
.. +(0.76,3.37) .. +(0.67,3.35) .. +(0.56,3.32) .. +(0.45,3.28) .. +(0.31,3.22) .. +(0.17,3.15)
.. +(0.03,3.06)
.. (2.7,0);
\fill (2.05,-4.23)node[right]{\tiny $V_{\widetilde{\Omega}_{\frac{r_{1}}{2}}\setminus\overline{\Omega}}$} circle [radius=1pt];
\fill (3.4,-6.5)node[right]{\tiny $V_{\Omega}$} circle [radius=1pt];
\fill (0,-8)node[right]{\tiny $V_{\partial\Omega}$} circle [radius=1pt];
\fill (2.3,-5.7)node[left]{\tiny $V_{\partial\Omega}$} circle [radius=1pt];
\draw[white](5.7,-7.5)node[right,black]{\small $\Omega$} to (5.7,-8);
\draw[white](1.7,-8.6)node[right,black]{\small $\widetilde{\Omega}$} to (3.7,-8);
\begin{scope}  
\clip[draw] (2.3,-5.7) .. (2.5,-5.2) .. (2.4,-5.3) .. (2.3,-5.7);
\fill[pattern=north east lines](2.5,-5.5) circle(1);
\end{scope}
\begin{scope}  
\clip[draw] (4.5,-7.5) .. (4.9,-7.45) .. (4.4,-7.25) .. (4.5,-7.5);
\fill[pattern=north east lines](4.5,-7.5) circle(1);
\end{scope}
\begin{scope}  
\clip[draw] (1.4,-6.5) coordinate (M) circle[radius=2mm];
\fill[pattern=north east lines](1.4,-6.5) circle(1);
\end{scope}
\end{tikzpicture}
}
\caption{Illustration of the vortex positions corresponding to $V_{\Omega}$, $V_{\partial\Omega}$, and
$V_{\widetilde{\Omega}_{\frac{r_{1}}{2}}}$.}
\label{Fig:VortexPositions}
\end{center}
\end{figure}
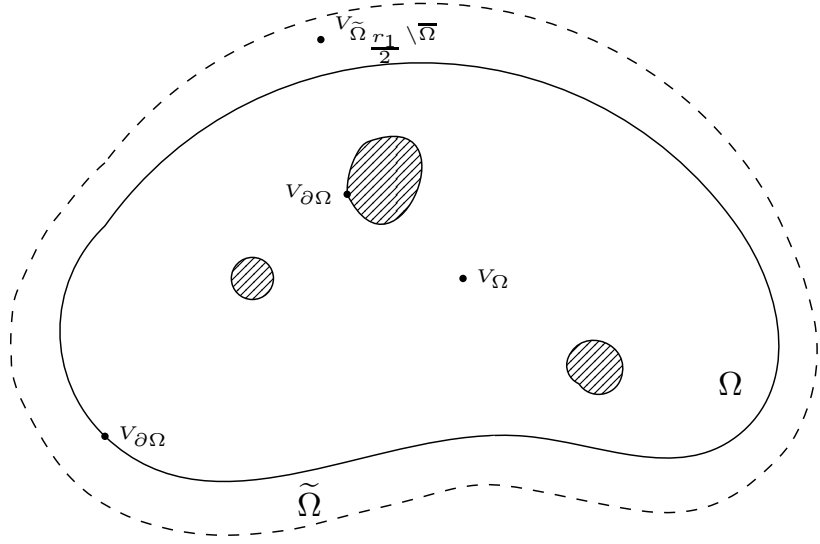
\end{center}
We then define $\widetilde{J}$ to be the measure given by
\begin{equation*}
    \widetilde{J}\coloneqq
    \pi\sum_{x\in{}V_{\Omega}\sqcup{}V_{\partial\Omega}}d_{x}\delta_{x}
    +\pi\sum_{\bar{x}\in{}
    V_{\widetilde{\Omega}_{\frac{r}{2}}\setminus\overline{\Omega}}}
    d_{\bar{x}}\delta_{\bar{x}}
\end{equation*}
where $d_{x}$ denotes the corresponding non-zero integer associated to $x$.
Next we extend $u_{\varepsilon}$ to a function on
$\widetilde{\Omega}_{\frac{r_{1}}{2}}$.
For each $i=0,1,\ldots,b$ and $j=1,2,\ldots,N_{i}$ we set
\begin{align*}
    \widetilde{z}_{i,j,\varepsilon}&\coloneqq{}
    u_{\varepsilon}\circ\widetilde{\psi}_{i,j}\\
    \widetilde{w}_{i,j,\varepsilon}&\coloneqq{}
    \widetilde{z}_{i,j,\varepsilon,\tau}e_{1}
    +\widetilde{z}_{i,j,\varepsilon,\nu}e_{2}\\
    \widetilde{Z}_{i,j,\varepsilon}(y_{1},y_{2})&\coloneqq
    \text{sgn}(y_{2})\widetilde{z}_{i,j,\varepsilon,\tau}(y_{1},|y_{2}|)
    \widetilde{\tau}_{i,j}
    +\widetilde{z}_{i,j,\varepsilon,\nu}(y_{1},|y_{2}|)\widetilde{\nu}_{i,j}
\end{align*}
where $\widetilde{z}_{i,j,\varepsilon,\tau}$ and
$\widetilde{z}_{i,j,\varepsilon,\nu}$ denote,
respectively, the components of
$\widetilde{z}_{i,j,\varepsilon}$ in the $\tau$ and $\nu$
directions and where $\widetilde{\tau}_{i,j}$ and $\widetilde{\nu}_{i,j}$ are
\begin{equation*}
    \widetilde{\tau}_{i,j}\coloneqq
    \tau_{i,j}((\widetilde{\psi}_{i,j}^{-1})^{1}(x)),
    \hspace{15pt}
    \widetilde{\nu}_{i,j}\coloneqq
    \nu_{i,j}((\widetilde{\psi}_{i,j}^{-1})^{1}(x)).
\end{equation*}
We define the extension $\widetilde{u}$ by
\begin{equation*}
    \widetilde{u}(x)=
    \begin{cases}
        u(x)& \text{if }x\in\overline{\Omega}\\
        \widetilde{Z}_{i,j,\varepsilon}(\widetilde{\psi}_{i,j}^{-1}(x))
        & \text{if } (\widetilde{\psi}_{i,j}^{-1})^{2}(x)<0.
    \end{cases}
\end{equation*}
This is well defined since if $x\in\partial\Omega$ then there is a local chart
$\widetilde{\psi}_{i,j}$ such that $x=\widetilde{\psi}_{i,j}(y_{1},0)$ and
\begin{equation*}
    u_{\varepsilon}(x)=u_{\varepsilon}\circ\widetilde{\psi}_{i,j}(y_{1},0)
    =z_{i,j,\varepsilon}(y_{1},0)
    =z_{i,j,\varepsilon,\nu}(y_{1},0)\widetilde{\nu}_{i,j}
    =z_{i,j,\varepsilon,\nu}\bigl(\widetilde{\psi}_{i,j}^{-1}(x)\bigr)
    \widetilde{\nu}_{i,j}
\end{equation*}
where we have used that $z_{i,j,\varepsilon,\tau}=0$ on $\partial\Omega$.
We now claim that
\begin{equation*}
    \Bigl\lVert\star{}J(\widetilde{u}_{\varepsilon})-\widetilde{J}\Bigr\rVert
    _{(C_{c}^{0,1}(\widetilde{\Omega}_{\frac{r_{1}}{2}}))^{*}}\longrightarrow0^{+}.
\end{equation*}
Let $\varphi\in{}C_{c}^{0,1}(\widetilde{\Omega}_{\frac{r_{1}}{2}})$
satisfy $\lVert\varphi\rVert
_{C_{c}^{0,1}(\widetilde{\Omega}_{\frac{r_{1}}{2}})}\le1$
and observe that
\begin{equation*}
    \bigl<\star{}J(\widetilde{u}_{\varepsilon})-\widetilde{J},\varphi\bigr>
    =\sum_{i=0}^{b}\sum_{j=1}^{N_{i}}
    \bigl<\star{}J(\widetilde{u}_{\varepsilon})-\widetilde{J},
    \widetilde{\rho}_{i,j}\varphi\bigr>
    +\bigl<\star{}J(\widetilde{u}_{\varepsilon})-\widetilde{J},
    \widetilde{\rho}_{0,0}\varphi\bigr>.
\end{equation*}
For $i=0$ and $j=0$ we have
\begin{equation*}
    \bigl<\star{}J(\widetilde{u}_{\varepsilon})-\widetilde{J},
    \widetilde{\rho}_{0,0}\varphi\bigr>
    =\bigl<\star{}J(u_{\varepsilon})-J,\widetilde{\rho}_{0,0}\varphi\bigr>
    \le{}C(\Omega)
    \lVert\star{}J(u_{\varepsilon})-J\rVert_{(C_{c}^{0,1}(\Omega))^{*}}.
\end{equation*}
For $i=0,1,\ldots,b$ and $j=1,2,\ldots,N_{i}$ we have
\begin{align*}
    \bigl<\widetilde{J},\rho_{i,j}\varphi\bigr>
    &=\biggl<\pi\sum_{k=1}^{M_{i,j,1}}
    d_{k}^{i,j}(\delta_{x_{k}^{i,j}}+\delta_{\bar{x}_{k}^{i,j}})
    +\pi\sum_{\ell=1}^{M_{i,j,2}}d_{\ell}^{i,j}\delta_{x_{\ell}^{i,j}},
    \widetilde{\rho}_{i,j}\varphi\biggr>\\
    &=\biggl<(\widetilde{\psi}_{i,j})_{\#}\biggl(
    \pi\sum_{k=1}^{M_{i,j,1}}d_{k}^{i,j}(\delta_{y_{k}^{i,j}}
    +\delta_{\bar{y}_{k}^{i,j}})
    +\pi\sum_{\ell=1}^{M_{i,j,2}}d_{\ell}^{i,j}\delta_{y_{\ell}^{i,j}}\biggr),
    \widetilde{\rho}_{i,j}\varphi\biggr>.
\end{align*}
Next observe that
\begin{align*}
    \bigl<\star{}J(\widetilde{u}_{\varepsilon}),\widetilde{\rho}_{i,j}\varphi\bigr>
    =\bigl<(\widetilde{\psi}_{i,j})
    _{\#}(\star{}J(\widetilde{u}\circ\widetilde{\psi}_{i,j})),
    \widetilde{\rho}_{i,j}\varphi\bigr>
\end{align*}
and hence
\begin{align*}
    \bigl<\star{}J(\widetilde{u}_{\varepsilon})-\widetilde{J},
    \widetilde{\rho}_{i,j}\varphi\bigr>
    &=\biggl<(\widetilde{\psi}_{i,j})
    _{\#}\biggl(\star{}J(\widetilde{u}\circ\widetilde{\psi}_{i,j})
    -\pi\sum_{k=1}^{M_{i,j,1}}d_{k}^{i,j}(\delta_{y_{k}^{i,j}}
    +\delta_{\bar{y}_{k}^{i,j}})
    -\pi\sum_{\ell=1}^{M_{i,j,2}}d_{\ell}^{i,j}\delta_{y_{\ell}^{i,j}}\biggr),
    \widetilde{\rho}_{i,j}\varphi\biggr>\\
    &=\biggl<\star{}J(\widetilde{u}\circ\widetilde{\psi}_{i,j})
    -\pi\sum_{k=1}^{M_{i,j,1}}d_{k}^{i,j}(\delta_{y_{k}^{i,j}}
    +\delta_{\bar{y}_{k}^{i,j}})
    -\pi\sum_{\ell=1}^{M_{i,j,2}}d_{\ell}^{i,j}\delta_{y_{\ell}^{i,j}},
    (\widetilde{\rho}_{i,j}\varphi)\circ\widetilde{\psi}_{i,j}\biggr>.
\end{align*}
We notice that
\begin{equation*}
    \widetilde{u}\circ\widetilde{\psi}_{i,j}(y_{1},y_{2})=
    \begin{cases}
        z_{i,j,\varepsilon,\tau}(y_{1},y_{2})\tau_{i,j}(y_{1})
        +z_{i,j,\varepsilon,\nu}(y_{1},y_{2})\nu_{i,j}(y_{1})&
        \text{if }y_{2}\ge0\\
        -z_{i,j,\varepsilon,\tau}(y_{1},-y_{2})\tau_{i,j}(y_{1})
        +z_{i,j,\varepsilon,\nu}(y_{1},-y_{2})\nu_{i,j}(y_{1})&
        \text{if }y_{2}<0.
    \end{cases}
\end{equation*}
Repeating the calculation from \eqref{JzRewrite}
we have
\begin{align*}
    &J(\widetilde{u}\circ\widetilde{\psi}_{i,j})(y_{1},y_{2})\\
    =&
    \begin{cases}
        Jw_{i,j,\varepsilon}(y_{1},y_{2})
        -\kappa_{i}(y_{1})
        \frac{\partial}{\partial{}y_{2}}
        \biggl(\frac{|z_{i,j,\varepsilon}(y_{1},y_{2})|^{2}-1}{2}\biggr)
        & \text{if }y_{2}\ge0\\
        Jw_{i,j,\varepsilon}(y_{1},-y_{2})
        +\kappa_{i}(y_{1})
        \frac{\partial}{\partial{}y_{2}}
        \biggl(\frac{|z_{i,j,\varepsilon}(y_{1},-y_{2})|^{2}-1}{2}\biggr)
        & \text{if }y_{2}<0.
    \end{cases}
\end{align*}
Combined with our previous calculations we have that
\begin{align*}
    &\biggl<\star{}J(\widetilde{u}\circ\widetilde{\psi}_{i,j})
    -\pi\sum_{k=1}^{M_{i,j,1}}d_{k}^{i,j}(\delta_{y_{k}^{i,j}}
    +\delta_{\bar{y}_{k}^{i,j}})
    -\pi\sum_{\ell=1}^{M_{i,j,2}}d_{\ell}^{i,j}\delta_{y_{\ell}^{i,j}},
    (\widetilde{\rho}_{i,j}\varphi)\circ\widetilde{\psi}_{i,j}\biggr>\\
    =&\biggl<\star{}J(w_{i,j,\varepsilon})(y_{1},|y_{2}|)
    -\pi\sum_{k=1}^{M_{i,j,1}}d_{k}^{i,j}(\delta_{y_{k}^{i,j}}
    +\delta_{\bar{y}_{k}^{i,j}})
    -\pi\sum_{\ell=1}^{M_{i,j,2}}d_{\ell}^{i,j}\delta_{y_{\ell}^{i,j}},
    (\widetilde{\rho}_{i,j}\varphi)\circ\widetilde{\psi}_{i,j}\biggr>\\
    -&\biggl<
    \text{sgn}(y_{2})\kappa_{i}(y_{1})
    \frac{\partial}{\partial{}y_{2}}\biggl(
    \frac{|z_{i,j,\varepsilon}|^{2}-1}{2}\biggr),
    (\widetilde{\rho}_{i,j}\varphi)\circ\widetilde{\psi}_{i,j}\biggr>.
\end{align*}
By symmetry and the argument from Step $1$ of the proof of \eqref{Compactness}
we find that
\begin{equation*}
    \left\|
    \text{sgn}(y_{2})\kappa_{i}(y_{1})
    \frac{\partial}{\partial{}y_{2}}\biggl(
    \frac{|z_{j,\varepsilon}|^{2}-1}{2}\biggr)\right\|
    _{(C_{c}^{0,1}(B_{\frac{r_{1}}{2}}(0)))^{*}}\longrightarrow0^{+}
\end{equation*}
as $\varepsilon\to0^{+}$.
Since $(\widetilde{\rho}_{i,j}\varphi)\circ\widetilde{\psi}_{i,j}\in{}
C_{c}^{0,1}(B_{\frac{r_{1}}{2}}(0))$
then it suffices to estimate the remaining term.
We notice that we can write
\begin{equation*}
    (\widetilde{\rho}_{i,j}\varphi)\circ\widetilde{\psi}_{i,j}
    =[(\widetilde{\rho}_{i,j}\varphi)\circ\widetilde{\psi}_{i,j}]_{e}
    +[(\widetilde{\rho}_{i,j}\varphi)\circ\widetilde{\psi}_{i,j}]_{o}
\end{equation*}
where for a function $f\colon{}B_{\frac{r_{1}}{2}}(0)\to\mathbb{R}^{2}$ we set
\begin{equation*}
    f_{e}(x_{1},x_{2})\coloneqq\frac{f(x_{1},x_{2})+f(x_{1},-x_{2})}{2},
    \hspace{15pt}
    f_{o}(x_{1},x_{2})\coloneqq\frac{f(x_{1},x_{2})-f(x_{1},-x_{2})}{2}.
\end{equation*}
That is, $f_{e}$ and $f_{o}$ denote, respectively, the even and odd part of $f$
in the second component.
We also observe that
\begin{equation*}
    \left\|f_{e}\right\|_{\mathcal{A}_{1,\frac{r_{1}}{2}}}=
    \left\|f_{e}\right\|_{C_{c}^{0,1}(B_{\frac{r_{1}}{2}}(0))}
    =\left\|f\right\|_{C_{c}^{0,1}(B_{\frac{r_{1}}{2}}(0))}
    .
\end{equation*}
Since $Jw_{i,j,\varepsilon}(y_{1},|y_{2}|)$ and the sum of delta masses are
even in $y_{2}$ and $B_{\frac{r_{1}}{2}}(0)$
is symmetric with respect to reflection in the $y_{2}$ variable then
\begin{align*}
    &\biggl<\star{}J(w_{i,j,\varepsilon})(y_{1},|y_{2}|)
    -\pi\sum_{k=1}^{M_{i,j,1}}d_{k}^{i,j}(\delta_{y_{k}^{i,j}}
    +\delta_{\bar{y}_{k}^{i,j}})
    -\pi\sum_{\ell=1}^{M_{i,j,2}}d_{\ell}^{i,j}\delta_{y_{\ell}^{i,j}},
    (\widetilde{\rho}_{i,j}\varphi)\circ\widetilde{\psi}_{i,j}\biggr>\\
    =&\biggl<\star{}J(w_{i,j,\varepsilon})(y_{1},|y_{2}|)
    -\pi\sum_{k=1}^{M_{i,j,1}}d_{k}^{i,j}(\delta_{y_{k}^{i,j}}
    +\delta_{\bar{y}_{k}^{i,j}})
    -\pi\sum_{\ell=1}^{M_{i,j,2}}d_{\ell}^{i,j}\delta_{y_{\ell}^{i,j}},
    [(\widetilde{\rho}_{i,j}\varphi)\circ\widetilde{\psi}_{i,j}]_{e}\biggr>\\
    =&2\biggl<\star{}J(w_{i,j,\varepsilon})(y_{1},y_{2})
    -\pi\sum_{k=1}^{M_{i,j,1}}d_{k}^{i,j}\delta_{y_{k}^{i,j}}
    -\frac{\pi}{2}\sum_{\ell=1}^{M_{i,j,2}}d_{\ell}^{i,j}\delta_{y_{\ell}^{i,j}},
    [(\widetilde{\rho}_{i,j}\varphi)\circ\widetilde{\psi}_{i,j}]_{e}\biggr>
\end{align*}
where the last equality uses symmetry.
In particular, the last equality is considered over $B_{\frac{r_{1}}{2},+}(0)$.
But we then note that for any $\phi\in{}\mathcal{A}_{1,\frac{r_{1}}{2}}$
with $\left\|\phi\right\|_{\mathcal{A}_{1,\frac{r_{1}}{2}}}\le1$
we have
\begin{align*}
    &\biggl<\star{}J(w_{i,j,\varepsilon})(y_{1},y_{2})
    -\pi\sum_{k=1}^{M_{i,j,1}}d_{k}^{i,j}\delta_{y_{k}^{i,j}}
    -\frac{\pi}{2}\sum_{\ell=1}^{M_{i,j,2}}d_{\ell}^{i,j}\delta_{y_{\ell}^{i,j}},
    \phi\biggr>\\
    \le&\biggl<\star{}J(z_{i,j,\varepsilon})(y_{1},y_{2})
    -\pi\sum_{k=1}^{M_{i,j,1}}d_{k}^{i,j}\delta_{y_{k}^{i,j}}
    -\frac{\pi}{2}\sum_{\ell=1}^{M_{i,j,2}}d_{\ell}^{i,j}\delta_{y_{\ell}^{i,j}},
    \phi\biggr>
    +\left\|\star{}J(w_{i,j,\varepsilon})-\star{}J(z_{i,j,\varepsilon})\right\|
    _{\mathcal{A}_{\frac{1,r_{1}}{2}}^{*}}
    \\
    =&\biggl<\star{}J(u_{\varepsilon})
    -\pi\sum_{k=1}^{M_{i,j,1}}d_{k}^{i,j}\delta_{x_{k}^{i,j}}
    -\frac{\pi}{2}\sum_{\ell=1}^{M_{i,j,2}}d_{\ell}^{i,j}\delta_{x_{\ell}^{i,j}},
    \phi\circ\widetilde{\psi}_{i,j}^{-1}\biggr>
    +\left\|\star{}J(w_{i,j,\varepsilon})-\star{}J(z_{i,j,\varepsilon})\right\|
    _{\mathcal{A}_{1,\frac{r_{1}}{2}}^{*}}
    \\
    \le&C(\Omega)\Biggl\|\star{}J(u_{\varepsilon})
    -\pi\sum_{k=1}^{M_{i,j,1}}d_{k}^{i,j}\delta_{x_{k}^{i,j}}
    -\frac{\pi}{2}\sum_{\ell=1}^{M_{i,j,2}}d_{\ell}^{i,j}
    \delta_{x_{\ell}^{i,j}}\Biggr\|
    _{(C^{0,1}(\Omega))^{*}}
    +\left\|\star{}J(w_{j,\varepsilon})-\star{}J(z_{j,\varepsilon})\right\|
    _{\mathcal{A}_{1,\frac{r_{1}}{2}}^{*}}
\end{align*}
where to obtain the last inequality we used that
$\phi\circ\widetilde{\psi}_{i,j}^{-1}$ is zero in a neighbourhood of
$\partial\widetilde{\mathcal{U}}_{i,j}\cap\Omega$.
The first term tends to zero by assumption while the second term
tends to zero by a similar argument to Step $1$ of \eqref{Compactness}.
Since $\widetilde{\rho}_{i,j}\circ\widetilde{\psi}_{i,j}=0$ in a neighbourhood of
$\partial{}B_{\frac{r_{1}}{2}}(0)$ then
$[(\widetilde{\rho}_{i,j}\varphi)\circ\widetilde{\psi}_{i,j}]_{e}\in\mathcal{A}_{1,\frac{r_{1}}{2}}$
and we obtain the
desired convergence since $i=0,1,\ldots,b$ and $j=1,2,\ldots,N_{i}$ was
arbitrary.
Thus, we now have
\begin{equation*}
    \left\|\star{}J(\widetilde{u}_{\varepsilon})-\widetilde{J}\right\|
    _{(C_{c}^{0,1}(\widetilde{\Omega}_{\frac{r_{1}}{2}}))^{*}}\longrightarrow0^{+}.
\end{equation*}
\subsubsection{Step \texorpdfstring{$2$}{}:}
\par{}Fix $0<r<r_{1}$.
By the previous step it follows that if
$\mathcal{O}_{\frac{r}{2}}\coloneqq\Omega\setminus
\biggl[\Omega\cap\bigcup\limits_{i=0}^{b}\widetilde{\Omega}_{i,\frac{r}{2}}\biggr]$
then
\begin{equation*}
    \left\|\star{}J(u_{\varepsilon})-J\right\|
    _{(C_{c}^{0,1}(\mathcal{O}_{\frac{r}{2}}))^{*}}
    \longrightarrow0^{+}
\end{equation*}
using that $\widetilde{u}=u$ and $\widetilde{J}=J$ on $\Omega$.
By Theorem $4.1$ of \cite{JS} or Theorem $1.1$ of \cite{AlBaOr} we have
\begin{equation}\label{InteriorLowerBound}
    \liminf_{\varepsilon\to0^{+}}\frac{E_{\varepsilon}(u_{\varepsilon},
    \mathcal{O}_{\frac{r}{2}})}
    {\mathopen{}\left|\log(\varepsilon)\right|\mathclose{}}
    \ge\pi\sum_{x\in\mathcal{O}_{\frac{r}{2}}}
    |d_{x}|.
\end{equation}
Similarly, we also have that the previous step gives
\begin{equation*}
    \left\|\star{}J(\widetilde{u}_{\varepsilon})-\widetilde{J}\right\|
    _{(C_{c}^{0,1}(\widetilde{\Omega}_{\frac{r}{2}}))^{*}}
    \longrightarrow0^{+},
\end{equation*}
where $\widetilde{\Omega}_{\frac{r}{2}}\coloneqq
\bigcup_{i=0}^{b}\widetilde{\Omega}_{i,\frac{r}{2}}$,
and hence
\begin{equation*}
    \liminf_{\varepsilon\to0^{+}}\frac{E_{\varepsilon}(\widetilde{u}_{\varepsilon},
    \widetilde{\Omega}_{\frac{r}{2}})}
    {\mathopen{}\left|\log(\varepsilon)\right|\mathclose{}}
    \ge2\pi\sum_{x\in\Omega_{\frac{r}{2}}}
    |d_{x}|+
    \pi\sum_{x\in\partial\Omega}|d_{x}|.
\end{equation*}
Computing in local coordinates shows that
\begin{equation*}
    E_{\varepsilon}(\widetilde{u}_{\varepsilon},
    \widetilde{\Omega}_{\frac{r}{2}})
    \le2\Bigl(1+C(\Omega)r\Bigr)
    E_{\varepsilon}(u_{\varepsilon},\Omega_{\frac{r}{2}})
\end{equation*}
and hence we obtain
\begin{equation*}
    2(1+C(\Omega)r)\liminf_{\varepsilon\to0^{+}}
    \frac{E_{\varepsilon}(u_{\varepsilon},
    \Omega_{\frac{r}{2}})}
    {\mathopen{}\left|\log(\varepsilon)\right|\mathclose{}}
    \ge2\pi\sum_{x\in\Omega_{\frac{r}{2}}}
    |d_{x}|+
    \pi\sum_{x\in\partial\Omega}|d_{x}|
\end{equation*}
or equivalently
\begin{equation}\label{LiminfNeighbourhood}
    (1+C(\Omega)r)\liminf_{\varepsilon\to0^{+}}
    \frac{E_{\varepsilon}(u_{\varepsilon},
    \Omega_{\frac{r}{2}})}
    {\mathopen{}\left|\log(\varepsilon)\right|\mathclose{}}
    \ge\pi\sum_{x\in\Omega_{\frac{r}{2}}}
    |d_{x}|+
    \frac{\pi}{2}\sum_{x\in\partial\Omega}|d_{x}|.
\end{equation}
Combining \eqref{InteriorLowerBound} and \eqref{LiminfNeighbourhood} we find
that
\begin{align*}
    (1+C(\Omega)r)\liminf_{\varepsilon\to0^{+}}
    \frac{E_{\varepsilon}(u_{\varepsilon},\Omega)}
    {\mathopen{}\left|\log(\varepsilon)\right|\mathclose{}}
    &\ge\liminf_{\varepsilon\to0^{+}}
    \frac{E_{\varepsilon}(u_{\varepsilon},
    \mathcal{O}_{\frac{r}{2}})}
    {\mathopen{}\left|\log(\varepsilon)\right|\mathclose{}}
    +(1+C(\Omega)r)\liminf_{\varepsilon\to0^{+}}
    \frac{E_{\varepsilon}(u_{\varepsilon},
    \Omega_{\frac{r}{2}})}
    {\mathopen{}\left|\log(\varepsilon)\right|\mathclose{}}
    \\
    &\ge\Biggl(
    \pi\sum_{x\in\mathcal{O}_{\frac{r}{2}}}
    |d_{x}|\Biggr)
    +\Biggl(\pi\sum_{x\in\Omega_{\frac{r}{2}}}
    |d_{x}|+
    \frac{\pi}{2}\sum_{x\in\partial\Omega}|d_{x}|\Biggr)
    \\
    &=\pi\sum_{i=1}^{M_{1}}|d_{i}|
    +\frac{\pi}{2}\sum_{j=1}^{M_{2}}
    |d_{j}|.
\end{align*}
Since $r$ was arbitrary we may let $r\to0^{+}$ to obtain the desired
inequality.

\subsection{Proof of Upper Bound}

\subsubsection{Step \texorpdfstring{$1$}{}:}
In this subsection we construct the canonical harmonic map with prescribed
singularities which has normal part zero.
We will use this map in order to build the recovery sequence for
the zeroth order $\Gamma$-convergence.
This generalizes the construction from \cite{IgKur1}, see also
\cite{BBH}, to consider interior
vortices as well as a general connected open subset of $\mathbb{R}^{2}$.
Our approach to building the canonical harmonic map is 
very much inspired by
the construction from \cite{BBH}.

Note that by rotating this map by $\frac{\pi}{2}$ radians we obtain the
canonical harmonic map with prescribed singularities which has tangential
part zero.\\

To construct the desired map we first solve the following boundary value
problem:
\begin{equation}\label{ConjugatePDE}
    \begin{cases}
        -\Delta{}\Psi=2\pi\sum\limits_{i=1}^{M_{1}}d_{i}\delta_{a_{i}}&
        \text{in }\Omega\\
        \hspace{8pt}
        \dfrac{\partial\Psi}{\partial{}\mathbf{n}}=
        \widetilde{\kappa}_{i}
        -\pi\sum\limits_{k=1}^{M_{2,j}}d_{jk}\delta_{c_{jk}}&
        \text{on }(\partial\Omega)_{j}\text{ for }j=1,2,\ldots,b\\
        \hspace{8pt}
        \dfrac{\partial\Psi}{\partial\mathbf{n}}=
        \widetilde{\kappa}_{0}-\pi\sum\limits_{k=1}^{M_{2,0}}d_{0k}\delta_{c_{0k}}&
        \text{on }(\partial\Omega)_{0}
    \end{cases}
\end{equation}
where $\mathbf{n}\coloneqq-\nu$ is the outward unit normal on
$(\partial\Omega)_{j}$ for $j=0,1,2,\ldots,b$.
We require that
$d_{i},d_{jk}\in\mathbb{Z}\setminus\{0\}$
for $i=1,2,\ldots,M_{1}$ and $k=1,2,\ldots,M_{2,j}$
for $j=0,1,\ldots,b$,
\begin{align}
    &\sum_{i=1}^{M_{1}}d_{i}
    +\frac{1}{2}\sum_{j=0}^{b}\sum_{k=1}^{M_{2,j}}d_{jk}
    =\chi_{Euler}(\Omega),\\
    &\frac{1}{2}\sum_{k=1}^{M_{2,j}}d_{jk}\in\mathbb{Z}
    \text{ for }j=0,1,\ldots,b,\label{IntegralSum}
\end{align}
as well as that $\{a_{i}\}_{i=1}^{M_{1}}\subseteq\Omega$ and
$\{c_{jk}\}_{k=1}^{M_{2,j}}\subseteq(\partial\Omega)_{j}$ for each
$j=0,1,\ldots,b$.
We let $\Psi_{1}\colon\Omega\to\mathbb{R}$ be defined by
\begin{equation}\label{InteriorLogs}
    \Psi_{1}(x)=\sum_{i=1}^{M_{1}}d_{i}\log(|x-a_{i}|).
\end{equation}
Notice that $\Psi_{1}$ solves
\begin{equation*}
    -\Delta\Psi_{1}=2\pi\sum_{i=1}^{M_{1}}d_{i}\delta_{a_{i}}
\end{equation*}
and so to solve \eqref{ConjugatePDE} it suffices to solve
\begin{equation*}
    \begin{cases}
        -\Delta{}\Psi=0&
        \text{in }\Omega\\
        \hspace{8pt}
        \dfrac{\partial\Psi}{\partial{}\mathbf{n}}=
        \widetilde{\kappa}_{i}
        -\pi\sum\limits_{k=1}^{M_{2,j}}d_{jk}\delta_{c_{jk}}
        -\dfrac{\partial\Psi_{1}}{\partial\mathbf{n}}&
        \text{on }(\partial\Omega)_{j}\text{ for }j=1,2,\ldots,b\\
        \hspace{8pt}
        \dfrac{\partial\Psi}{\partial\mathbf{n}}=
        \widetilde{\kappa}_{0}
        -\pi\sum\limits_{k=1}^{M_{2,0}}d_{0k}\delta_{c_{0k}}
        -\dfrac{\partial\Psi_{1}}{\partial\mathbf{n}}&
        \text{on }(\partial\Omega)_{0}.
    \end{cases}
\end{equation*}
Next, we let $\Psi_{2}\colon\Omega\to\mathbb{R}$ be defined by
\begin{equation}\label{BoundaryLogs}
    \Psi_{2}(x)=\sum_{p=0}^{b}\sum_{\ell=1}^{M_{2,p}}d_{p\ell}
    \log(|x-c_{p\ell}|).
\end{equation}
We claim that $\Psi_{2}$ satisfies
\begin{equation}\label{ReducedProblem}
    \begin{cases}
        -\Delta\Psi_{2}=0& \text{on }\Omega\\
        \hspace{11pt}
        \dfrac{\partial\Psi_{2}}{\partial\mathbf{n}}=
        -\pi\sum\limits_{\ell=1}^{M_{2,j}}d_{j\ell}\delta_{c_{j\ell}}+
        \sum\limits_{p=0}^{b}\sum\limits_{\ell=1}^{M_{2,p}}
        d_{p\ell}\frac{x-c_{p\ell}}{|x-c_{p\ell}|^{2}}\cdot\mathbf{n}&
        \text{on }(\partial\Omega)_{j}\text{ for }j=1,2,\ldots,b\\
        \hspace{11pt}
        \dfrac{\partial\Psi_{2}}{\partial\mathbf{n}}=
        -\pi\sum\limits_{\ell=1}^{M_{2,0}}d_{0\ell}\delta_{c_{0\ell}}+
        \sum\limits_{p=0}^{b}\sum\limits_{\ell=1}^{M_{2,p}}
        d_{p\ell}\frac{x-c_{p\ell}}{|x-c_{p\ell}|^{2}}\cdot\mathbf{n}&
        \text{on }(\partial\Omega)_{0}.
    \end{cases}
\end{equation}
To see this consider $\varphi\in{}W^{1,2}(\Omega)$ and observe that
\begin{align*}
    \int_{\Omega}\!{}\nabla\Psi_{2}\cdot\nabla\varphi
    &=\lim_{\delta\to0^{+}}
    \int_{\Omega\setminus
    \bigcup\limits_{j=0}^{b}\bigcup\limits_{m=1}^{M_{2,j}}
    B_{\delta}(c_{jm})}\!{}
    \nabla\Psi_{2}\cdot\nabla\varphi\\
    &=\lim_{\delta\to0^{+}}
    \biggl[
    \int_{(\partial\Omega)_{0}\setminus\bigcup\limits_{m=1}^{M_{2,0}}
    B_{\delta}(c_{0m})}\!{}
    \varphi
    \sum_{p=0}^{b}\sum_{\ell=1}^{M_{2,p}}d_{p\ell}
    \frac{x-c_{p\ell}}{|x-c_{p\ell}|^{2}}
    \cdot\mathbf{n}
    \\
    &+\sum_{j=1}^{b}
    \int_{(\partial\Omega)_{j}\setminus
    \bigcup\limits_{m=1}^{M_{2,j}}B_{\delta}(c_{jm})}\!{}
    \varphi
    \sum_{p=0}^{b}\sum_{\ell=1}^{M_{2,p}}d_{p\ell}
    \frac{x-c_{p\ell}}{|x-c_{p\ell}|^{2}}
    \cdot\mathbf{n}\\
    &+\sum_{m=1}^{M_{2,0}}
    \int_{\Omega\cap\partial{}B_{\delta}(c_{0m})}\!{}
    \varphi
    \sum_{p=0}^{b}\sum_{\ell=1}^{M_{2,p}}d_{p\ell}
    \frac{x-c_{p\ell}}{|x-c_{p\ell}|^{2}}
    \cdot\mathbf{n}\\
    &+\sum_{j=1}^{b}\sum_{m=1}^{M_{2,j}}
    \int_{\Omega\cap\partial{}B_{\delta}(c_{jm})}\!{}
    \varphi
    \sum_{p=0}^{b}\sum_{\ell=1}^{M_{2,p}}d_{p\ell}
    \frac{x-c_{p\ell}}{|x-c_{p\ell}|^{2}}
    \cdot\mathbf{n}
    \biggr].
\end{align*}
Next, notice that
\begin{align*}
    &\sum_{m=1}^{M_{2,0}}
    \int_{\Omega\cap\partial{}B_{\delta}(c_{0m})}\!{}
    \varphi
    \sum_{p=0}^{b}\sum_{\ell=1}^{M_{2,p}}d_{p\ell}
    \frac{x-c_{p\ell}}{|x-c_{p\ell}|^{2}}
    \cdot\mathbf{n}
    +\sum_{j=1}^{b}\sum_{m=1}^{M_{2,j}}
    \int_{\Omega\cap\partial{}B_{\delta}(c_{jm})}\!{}
    \varphi
    \sum_{p=0}^{b}\sum_{\ell=1}^{M_{2,p}}d_{p\ell}
    \frac{x-c_{p\ell}}{|x-c_{p\ell}|^{2}}
    \cdot\mathbf{n}\\
    =&-\sum_{m=1}^{M_{2,0}}\frac{d_{0m}}{\delta}
    \int_{\Omega\cap\partial{}B_{\delta}(c_{0m})}\!{}\varphi
    -\sum_{j=1}^{b}\sum_{m=1}^{M_{2,j}}\frac{d_{jm}}{\delta}
    \int_{\Omega\cap\partial{}B_{\delta}(c_{jm})}\!{}\varphi
    +O(\delta)
\end{align*}
and hence
\begin{align*}
    &\lim_{\delta\to0^{+}}\biggl[
    \sum_{m=1}^{M_{2,0}}\int_{\Omega\cap\partial{}B_{\delta}(c_{0m})}\!{}
    \varphi\sum_{p=0}^{b}\sum_{\ell=1}^{M_{2,p}}d_{p\ell}
    \frac{x-c_{p\ell}}{|x-c_{p\ell}|^{2}}\cdot\mathbf{n}
    +\sum_{j=1}^{b}\sum_{m=1}^{M_{2,j}}
    \int_{\Omega\cap\partial{}B_{\delta}(c_{jm})}\!{}
    \varphi\sum_{p=0}^{b}\sum_{\ell=1}^{M_{2,p}}d_{p\ell}
    \frac{x-c_{p\ell}}{|x-c_{p\ell}|^{2}}\cdot\mathbf{n}
    \biggr]\\
    =&-\pi\sum_{m=1}^{M_{2,0}}d_{0m}\varphi(c_{0m})
    -\pi\sum_{j=1}^{b}\sum_{m=1}^{M_{2,j}}d_{jm}\varphi(c_{jm}).
\end{align*}
Now we show that each
$\frac{x-c_{p\ell}}{|x-c_{p\ell}|^{2}}\cdot\mathbf{n}$ has a
finite limit at $c_{p\ell}$ along $(\partial\Omega)_{p}$ for $p=0,1,\ldots,b$.
It will also follow from this calculation that this function is Lipschitz
on $(\partial\Omega)_{p}$.
We only do the calculation when $p=1,2,\ldots,b$ as the case of $p=0$ is
similar.
We note that since $c_{q\ell}\notin(\partial\Omega)_{p}$ then this function
is smooth on $(\partial\Omega)_{q}$ and hence has a finite limit at
$c_{q\ell}$.
Now we consider $c_{p\ell}\in(\partial\Omega)_{p}$.
Using a local parametrization of
$(\partial\Omega)_{p}\cap{}B_{\delta}(c_{p\ell})$
similar to the one found in Section \ref{TangentNormalCoordinates}, except
centred at $c_{p\ell}$, we write
$(x-c_{p\ell})\cdot\mathbf{n}
=(\gamma_{p}(s)-\gamma_{p}(0))\cdot\tilde{\mathbf{n}}(s)$
where $\tilde{\mathbf{n}}(s)$ is the coordinate representation of
$\mathbf{n}$.
We introduce, by mollifying, a smooth approximation
$\tilde{\mathbf{n}}_{\eta}$ to
$\tilde{\mathbf{n}}$ which converges in $C^{1}$ and whose second derivative
is uniformly bounded in terms of the Lipschitz norm of
$\tilde{\mathbf{n}}'$.
Next we note that by a Taylor expansion to second order we have
\begin{equation*}
    (\gamma_{p}(s)-\gamma_{p}(0))\cdot\tilde{\mathbf{n}}_{\eta}(s)\\
    =s\gamma_{p}'(0)\cdot\tilde{\mathbf{n}}_{\eta}(0)
    +\frac{s^{2}}{2}\gamma_{p}''(0)\cdot\tilde{\mathbf{n}}_{\eta}(0)
    +\frac{s^{2}}{2}\gamma_{p}'(0)\cdot\tilde{\mathbf{n}}_{\eta}'(0)
    +O(s^{3}).
\end{equation*}
Using that $\tilde{\mathbf{n}}_{\eta}$ approximates $\tilde{\mathbf{n}}$
in $C^{1}$ as $\eta\to0^{+}$ we see that
\begin{equation*}
    (\gamma_{p}(s)-\gamma_{p}(0))\cdot\tilde{\mathbf{n}}(s)
    =s\gamma_{p}'(0)\cdot\tilde{\mathbf{n}}(0)
    +\frac{s^{2}}{2}\gamma_{p}''(0)\cdot\tilde{\mathbf{n}}(0)
    +\frac{s^{2}}{2}\gamma_{p}'(0)\cdot\tilde{\mathbf{n}}'(0)
    +O(s^{3}).
\end{equation*}
Noting that on $(\partial\Omega)_{p}$ we have
\begin{equation*}
    \tilde{\mathbf{n}}(s)=-\gamma_{p}'(s)^{\perp},\hspace{20pt}
    \tilde{\mathbf{n}}'(s)=\kappa_{p}(s)\gamma_{p}'(s)
\end{equation*}
we see that
\begin{align*}
    (\gamma_{p}(s)-\gamma_{p}(0))\cdot\tilde{\mathbf{n}}(s)
    &=-\frac{s^{2}}{2}\gamma_{p}''(0)\cdot\gamma_{p}'(0)^{\perp}
    +\frac{s^{2}}{2}\kappa_{p}(0)
    +O(s^{3})\\
    &=-\frac{\kappa_{p}(0)s^{2}}{2}+\frac{\kappa_{p}(0)s^{2}}{2}
    +O(s^{3})=O(s^{3}).
\end{align*}
We also observe that
\begin{equation*}
    |x-c_{p\ell}|^{2}=s^{2}+O(s^{3}).
\end{equation*}
We conclude that
\begin{equation}\label{FiniteLimit}
    \frac{x-c_{p\ell}}{|x-c_{p\ell}|^{2}}\cdot\mathbf{n}
    =O(s).
\end{equation}
Hence, the limit is finite as $x$ tends to $c_{p\ell}$.
One can note from the error terms that
$x\mapsto\frac{x-c_{p\ell}}{|x-c_{p\ell}|^{2}}\cdot\mathbf{n}$ is a Lipschitz
function on $(\partial\Omega)_{p}$ since $\Omega$ has $C^{2,1}$ boundary.
This permits us to conclude that
\begin{align*}
    &\lim_{\delta\to0^{+}}
    \biggl[
    \int_{(\partial\Omega)_{0}\setminus\bigcup\limits_{m=1}^{M_{2,0}}
    B_{\delta}(c_{0m})}\!{}
    \varphi
    \sum_{p=0}^{b}\sum_{\ell=1}^{M_{2,p}}d_{p\ell}
    \frac{x-c_{p\ell}}{|x-c_{p\ell}|^{2}}
    \cdot\mathbf{n}
    \\
    &+\sum_{j=1}^{b}
    \int_{(\partial\Omega)_{j}\setminus
    \bigcup\limits_{p=1}^{M_{2,j}}B_{\delta}(c_{jp})}\!{}
    \varphi
    \sum_{p=0}^{b}\sum_{\ell=1}^{M_{2,p}}d_{p\ell}
    \frac{x-c_{p\ell}}{|x-c_{p\ell}|^{2}}
    \cdot\mathbf{n}\biggr]\\
    &=\int_{(\partial\Omega)_{0}}\!{}
    \varphi
    \sum_{p=0}^{b}\sum_{\ell=1}^{M_{2,p}}d_{p\ell}
    \frac{x-c_{p\ell}}{|x-c_{p\ell}|^{2}}
    \cdot\mathbf{n}
    +\sum_{j=1}^{b}
    \int_{(\partial\Omega)_{j}}\!{}
    \varphi
    \sum_{p=0}^{b}\sum_{\ell=1}^{M_{2,p}}d_{p\ell}
    \frac{x-c_{p\ell}}{|x-c_{p\ell}|^{2}}
    \cdot\mathbf{n}.
\end{align*}
Putting this together gives that $\Psi_{2}$ solves \eqref{ReducedProblem}.
In addition, taking $\varphi\equiv1$ we obtain
\begin{align}
    &\int_{(\partial\Omega)_{0}}\!{}
    \sum\limits_{p=1}^{b}\sum\limits_{\ell=1}^{M_{2,p}}
    d_{p\ell}\frac{x-c_{p\ell}}{|x-c_{p\ell}|^{2}}\cdot\mathbf{n}
    +\sum_{j=1}^{b}
    \int_{(\partial\Omega)_{j}}\!{}
    \sum\limits_{p=1}^{b}\sum\limits_{\ell=1}^{M_{2,p}}
    d_{p\ell}\frac{x-c_{p\ell}}{|x-c_{p\ell}|^{2}}\cdot\mathbf{n}
    \nonumber\\
    =&\pi\sum_{j=0}^{b}\sum_{m=1}^{M_{2,j}}d_{jm}.
    \label{Compatibility}
\end{align}
Letting $\Bigl(\frac{\partial\Psi_{2}}{\partial\mathbf{n}}\Bigr)_{ac}$
denote
\begin{equation*}
    \Bigl(\frac{\partial\Psi_{2}}{\partial\mathbf{n}}\Bigr)_{ac}
    \coloneqq\sum\limits_{p=0}^{b}\sum\limits_{\ell=1}^{M_{2,p}}
    d_{p\ell}\frac{x-c_{p\ell}}{|x-c_{p\ell}|^{2}}\cdot\mathbf{n}
\end{equation*}
we see that \eqref{Compatibility} gives the desired
\begin{equation*}
    \int_{\partial\Omega}\!{}
    \Bigl(\frac{\partial\Psi_{2}}{\partial\mathbf{n}}\Bigr)_{ac}
    =2\pi\chi_{Euler}(\Omega)-
    2\pi\sum_{i=1}^{M_{1}}d_{i}.
\end{equation*}
With $\Psi_{2}$ it now suffices to solve
\begin{equation}\label{FurtherReducedProblem}
    \begin{cases}
        -\Delta{}\Psi=0&
        \text{in }\Omega\\
        \hspace{8pt}
        \dfrac{\partial\Psi}{\partial{}\mathbf{n}}=
        \widetilde{\kappa}_{i}
        -\Bigl(\dfrac{\partial\Psi_{2}}{\partial\mathbf{n}}\Bigr)_{ac}
        -\dfrac{\partial\Psi_{1}}{\partial\mathbf{n}}&
        \text{on }(\partial\Omega)_{i}\text{ for }i=1,2,\ldots,b\\[6pt]
        \hspace{8pt}
        \dfrac{\partial\Psi}{\partial\mathbf{n}}=
        \widetilde{\kappa}_{0}
        -\Bigl(\dfrac{\partial\Psi_{2}}{\partial\mathbf{n}}\Bigr)_{ac}
        -\dfrac{\partial\Psi_{1}}{\partial\mathbf{n}}&
        \text{on }(\partial\Omega)_{0}
    \end{cases}
\end{equation}
where the Neumann data is in $C^{0,1}(\partial\Omega)$.
Observe that if
$g(x)\coloneqq{}e^{i\sum\limits_{j=1}^{M_{1}}d_{j}\theta(x-a_{j})}$,
where $\theta$ is the argument function as in \cite{BBH}, then
\begin{equation*}
    \frac{\partial\Psi_{1}}{\partial\mathbf{n}}=g\times\partial_{\tau}g
\end{equation*}
and hence
\begin{equation*}
    \int_{\partial\Omega}\!{}\frac{\partial\Psi_{1}}{\partial\mathbf{n}}
    =2\pi\text{deg}(g,(\partial\Omega)_{0})
    -2\pi\sum_{i=1}^{b}\text{deg}(g,(\partial\Omega)_{i})
    =2\pi\sum_{i=1}^{M_{1}}d_{i}.
\end{equation*}
Finally, observe that
\begin{align*}
    &\int_{(\partial\Omega)_{0}}\!{}\biggl[\widetilde{\kappa}_{0}
    -\Bigl(\frac{\partial\Psi_{2}}{\partial\mathbf{n}}\Bigr)_{ac}
    -\frac{\partial\Psi_{1}}{\partial\mathbf{n}}\biggr]
    +\sum_{i=1}^{b}\int_{(\partial\Omega)_{i}}\!{}
    \biggl[\widetilde{\kappa}_{i}
    -\Bigl(\frac{\partial\Psi_{2}}{\partial\mathbf{n}}\Bigr)_{ac}
    -\frac{\partial\Psi_{1}}{\partial\mathbf{n}}\biggr]\\
    =&2\pi\chi_{Euler}(\Omega)
    -\int_{\partial\Omega}\!{}
    \Bigl(\frac{\partial\Psi_{2}}{\partial\mathbf{n}}\Bigr)_{ac}
    -\int_{\partial\Omega}\!{}\frac{\partial\Psi_{1}}{\partial\mathbf{n}}\\
    =&2\pi\chi_{Euler}(\Omega)-
    \biggl[2\pi\chi_{Euler}(\Omega)-2\pi\sum_{i=1}^{M_{1}}d_{i}\biggr]
    -2\pi\sum_{i=1}^{M_{1}}d_{i}\\
    =&0
\end{align*}
and that the boundary data is continuous.
Thus, \eqref{FurtherReducedProblem} has a solution $H_{\Omega}$.
As in the proof of Proposition $20$ of \cite{IgKur1} we have
$H_{\Omega}\in{}W^{1,2}(\Omega)$.
Letting $\Psi\coloneqq\Psi_{1}+\Psi_{2}+H_{\Omega}$ we now have a solution to
\eqref{ConjugatePDE} in $W^{1,p}(\Omega)$ for $1\le{}p<2$.\\

Next, we let $\bar{j}\colon\Omega\to\mathbb{R}^{2}$ denote the vector field
\begin{equation*}
    \bar{j}\coloneqq-\nabla^{\perp}\Psi.
\end{equation*}
We observe that $\bar{j}\in{}L^{p}(\Omega;\mathbb{R}^{2})$ for $1\le{}p<2$
and satisfies
\begin{equation}\label{CanonicalHarmonicEquation}
    \begin{cases}
        \text{div}(\bar{j})=0&\text{on }\Omega\\
        \text{curl}(\bar{j})
        =2\pi\sum\limits_{i=1}^{M_{1}}d_{i}\delta_{a_{i}}&\text{on }\Omega\\
        \bar{j}\cdot\tau=\widetilde{\kappa}_{i}
        -\pi\sum\limits_{k=1}^{M_{2,j}}d_{jk}\delta_{c_{jk}}&
        \text{on }(\partial\Omega)_{j}\text{ for }j=1,2,\ldots,b\\
        \bar{j}\cdot\tau=\widetilde{\kappa}_{0}
        -\pi\sum\limits_{k=1}^{M_{2,0}}d_{0k}\delta_{c_{0k}}&
        \text{on }(\partial\Omega)_{0}.
    \end{cases}
\end{equation}
Following the argument of Lemma $I.1$ on page $5$ of \cite{BBH}
we conclude from \eqref{IntegralSum} and the first, third, and fourth
conditions of \eqref{CanonicalHarmonicEquation} that, by integrating over
paths in $\Omega$, there is a function
$\bar{u}\in{}W^{1,p}(\Omega;\mathbb{S}^{1})$, for $1\le{}p<2$, such that
\begin{equation}\label{CurrentEquation}
    j\bar{u}=\bar{j}=-\nabla\Psi.
\end{equation}
As noted on page $8$ of \cite{BBH} the solution, $\bar{u}$, is unique up to a
constant phase.
Next, we define $g_{i}\in{}BV((\partial\Omega)_{i};\mathbb{S}^{1})$, for each
$i=0,1,\ldots,b$, so that
$g_{i}\times\partial_{\tau}g_{i}=\bar{j}\cdot\tau_{i}$.
The definition is done by an algorithm.
We start with the case of $i=1,2,\ldots,b$.
Let $x_{0}^{i}\in(\partial\Omega)_{i}\setminus\bigcup\limits_{j=1}^{M_{2,i}}
\{c_{ij}\}$ and define $g_{i}(x)=\tau_{i}(x)$ on the connected component
of $(\partial\Omega)_{i}\setminus\bigcup\limits_{j=1}^{M_{2,i}}\{c_{ij}\}$
containing $x_{0}^{i}$.
Following the boundary curve in the counterclockwise direction
we will encounter
$c_{ij_{1}}\in\bigcup\limits_{j=1}^{M_{2,i}}\{c_{ij}\}$ which has associated
degree $(-1)^{d_{ij_{1}}}$.
We define $g_{i}$ to be $(-1)^{-d_{ij_{1}}}\tau_{i}$ on the adjacent
connected component of
$(\partial\Omega)_{i}\setminus\bigcup\limits_{j=1}^{M_{2,i}}\{c_{ij}\}$
whose boundary contains $c_{ij_{1}}$ (i.e. with local argument
representation shifted by $-d_{ij_{1}}\pi$).
We then repeat our algorithm after we encounter another member of
$\bigcup\limits_{j=1}^{M_{2,i}}\{c_{ij}\}$.
Since $\bigcup\limits_{j=1}^{M_{2,i}}\{c_{ij}\}$ contains only finitely
many points and $(\partial\Omega)_{i}$ is a closed curve then we eventually
return to the connected component containing $x_{0}^{i}$.
Our function is well defined due to \eqref{IntegralSum}.
Observe that on
$(\partial\Omega)_{i}\setminus\bigcup\limits_{j=1}^{M_{2,i}}\{c_{ij}\}$
then
\begin{equation*}
    g_{i}\times\partial_{\tau}g_{i}=\widetilde{\kappa}_{i}.
\end{equation*}
In coordinates about each $c_{ij}$ for $j=1,2,\ldots,M_{2,i}$ we also see that
at $c_{ij}$:
\begin{equation*}
    g_{i}\times\partial_{\tau}g_{i}
    =\widetilde{\kappa}_{i}(c_{ij})-\pi{}d_{ij}
\end{equation*}
We conclude that
\begin{equation*}
    g_{i}\times\partial_{\tau}g_{i}=\widetilde{\kappa}_{i}-
    \pi\sum_{j=1}^{M_{2,i}}d_{ij}\delta_{c_{ij}}.
\end{equation*}
A similar construction works along $(\partial\Omega)_{0}$ except the singular
part of $g_{0}\times\partial_{\tau}g_{0}$ has jumps of
$-\pi{}d_{0j}$ at $c_{0j}\in\bigcup\limits_{j=1}^{M_{2,0}}\{c_{0j}\}$ due to
orientation.
Thus, $\bar{u}$ solves
\begin{equation*}
    \begin{cases}
        j\bar{u}=\bar{j}&\text{on }\Omega\\
        j\bar{u}=g_{i}\times\partial_{\tau}g_{i}&
        \text{on }(\partial\Omega)_{i}\text{ for }i=0,1,\ldots,b.
    \end{cases}
\end{equation*}
Arguing as in the proof of Theorem $I.3$ on page $9$ of \cite{BBH} there is
a phase $\theta_{0}\in[0,2\pi]$ so that
$\tilde{u}\coloneqq{}e^{-i\theta_{0}}\bar{u}$ satisfies
\begin{equation}\label{ProvisionalCurrentEquation}
    \begin{cases}
        j\tilde{u}=\bar{j}&\text{on }\Omega\\
        \tilde{u}=g_{0}&\text{on }(\partial\Omega)_{0}\\
        \tilde{u}=e^{i\theta_{i}}g_{i}&\text{on }(\partial\Omega)_{i}
    \end{cases}
\end{equation}
where $\theta_{i}\in[0,2\pi]$ for each $i=1,2,\ldots,b$.
Finally, we let, for $i=1,2,\ldots,b$, $\phi_{i}\colon\Omega\to\mathbb{R}$
be the solution to
\begin{equation*}
    \begin{cases}
        -\Delta\phi_{i}=0&\text{on }\Omega\\
        \phi_{i}(x)=1&\text{on }(\partial\Omega)_{i}\\
        \phi_{i}(x)=0&\text{on }(\partial\Omega)_{j}\text{ for }j\neq{}i.
    \end{cases}
\end{equation*}
and we let $\eta_{i}$, for $i=1,2,\ldots,b$ denote $\eta_{i}=\nabla\phi_{i}$.
Observe that $\eta_{i}\in{}H_{T}^{1}(\Omega)$, the space of harmonic $1$-forms
 with zero tangential part, for each $i=1,2,\ldots,b$ and
by Lemmas $9$ and $12$ of \cite{BJOS} form a basis for $H_{T}^{1}(\Omega)$.
We consider $u_{*}^{N}\colon\Omega\to\mathbb{S}^{1}$ defined by
\begin{equation}\label{def:TangentMap}
    u_{*}^{N}\coloneqq{}
    e^{-i\sum\limits_{i=1}^{b}\theta_{i}\phi_{i}}\tilde{u}.
\end{equation}
Observe that $u_{*}^{N}\in{}W^{1,p}(\Omega;\mathbb{S}^{1})$, for $1\le{}p<2$,
satisfies $(u_{*}^{N})_{N}=0$,
$ju_{*}^{N}=\bar{j}-\sum\limits_{i=1}^{b}\theta_{i}\eta_{i}$, as well as
\begin{equation*}
    \begin{cases}
        \text{div}(ju_{*}^{N})=0&\text{on }\Omega\\
        \text{curl}(ju_{*}^{N})
        =2\pi\sum\limits_{i=1}^{M_{1}}d_{i}\delta_{a_{i}}&\text{on }\Omega\\
        ju_{*}^{N}\cdot\tau=\widetilde{\kappa}_{i}
        -\pi\sum\limits_{j=1}^{M_{2,i}}d_{ij}\delta_{c_{ij}}&
        \text{on }(\partial\Omega)_{i}\text{ for }i=1,2,\ldots,b\\
        ju_{*}^{N}\cdot\tau=\widetilde{\kappa}_{0}
        -\pi\sum\limits_{j=1}^{M_{2,0}}d_{0j}\delta_{c_{0j}}&
        \text{on }(\partial\Omega)_{0}.
    \end{cases}
\end{equation*}
This is the canonical harmonic map with prescribed singularities and normal
part zero.
In order to simplify the renormalized energy we observe that we may apply
Gram-Schmidt and normalization to $\{\eta_{1},\eta_{2},\ldots,\eta_{b}\}$ to
obtain an orthonormal basis
$\{\bar{\eta}_{1},\bar{\eta}_{2},\ldots,\bar{\eta}_{b}\}$
of $H_{T}^{1}(\Omega)$.
Notice that by \eqref{CurrentEquation},
$\bar{j}\in(H_{T}^{1}(\Omega))^{\perp}$
and so if we set
\begin{equation*}
    \Phi_{j}\coloneqq\int_{\Omega}\!{}ju_{*}^{N}\cdot\bar{\eta}_{j}
\end{equation*}
we have the $L^{2}$ orthogonal decomposition 
\begin{equation}\label{Modifiedj}
    ju_{*}^{N}=\bar{j}+\sum_{i=1}^{b}\Phi_{j}\bar{\eta}_{j}.
\end{equation}

Finally we note that to obtain the canonical harmonic map with prescribed
singularities and tangential part zero we set
\begin{equation*}
    u_{*}^{T}\coloneqq(u_{*}^{N})^{\perp}.
\end{equation*}

\begin{remark}\hspace{2pt}\\[-10pt]
    \begin{enumerate}
        \item
            We observe that $ju_{*}^{N}-\bar{j}\in{}H_{T}^{1}(\Omega)$ is a
            similar lattice condition to what was identified in
            \cite{IgJ} as well as \cite{BBH}.
        \item
            Observe that in the case that $\Omega$ is simply connected
            one could stop at \eqref{ProvisionalCurrentEquation}
            since $H_{T}^{1}(\Omega)\cong0$.
            As a result, there are no flux integrals, $\Phi_{j}$,
            in this case.
            This provides an alternative to the construction found in
            \cite{IgKur1}.
            Note in this case that \eqref{ConjugatePDE} characterizes
            the canonical harmonic map while in the non-simply connected
            case \eqref{ConjugatePDE} is not sufficient.
    \end{enumerate}
\end{remark}

\subsubsection{Step \texorpdfstring{$2$}{}:}
Here we compute the renormalized energy.
As the Ginzburg-Landau energy is invariant under rotation by $\frac{\pi}{2}$
radians we only do this computation for $u_{*}^{T}$.\\

Consider $\{a_{i}\}_{i=1}^{M_{1}}\subseteq\Omega$ and for each $i=0,1,\ldots,b$
we consider $\{c_{ij}\}_{j=1}^{M_{2,i}}\subseteq(\partial\Omega)_{i}$.
Corresponding to these points we let $d_{i}\in\mathbb{Z}\setminus\{0\}$ for 
$i=1,2,\ldots,M_{1}$ and for each $i=0,1,\ldots,b$ we consider
$d_{ij}\in\mathbb{Z}\setminus\{0\}$ corresponding to $c_{ij}\in(\partial\Omega)_{i}$.
We also require that
\begin{align*}
    &\sum_{i=1}^{M_{1}}d_{i}
    +\frac{1}{2}\sum_{i=0}^{b}\sum_{j=1}^{M_{2,i}}d_{ij}
    =\chi_{Euler}(\Omega)\\
    &\frac{1}{2}\sum_{j=1}^{M_{2,i}}d_{ij}\in\mathbb{Z}
    \text{ for }i=0,1,\ldots,b.
\end{align*}
If $u_{*}^{T}$ is the associated canonical harmonic map with prescribed
singularities satisfying the above then we let
\begin{align*}
    \mathbf{a}&\coloneqq(a_{1},a_{2},\ldots,a_{M_{1}})\\
    \mathbf{c}_{i}&\coloneqq(c_{i1},c_{i2},\ldots,c_{iM_{2,i}})\hspace{2pt}
    \text{for }i=0,1,\ldots,b\\
    \mathbf{c}&\coloneqq
    (\mathbf{c}_{0},\mathbf{c}_{1},\ldots,\mathbf{c}_{b})\\
    \mathbf{d}_{1}&\coloneqq(d_{1},d_{2},\ldots,d_{M_{1}})\\
    \mathbf{d}_{2,i}&\coloneqq(d_{i1},d_{i2},\ldots,d_{iM_{2,i}})\\
    \mathbf{d}_{2}&\coloneqq(\mathbf{d}_{2,0},\mathbf{d}_{2,1},\ldots,
    \mathbf{d}_{2,b})\\
    \Phi&\coloneqq(\Phi_{1},\Phi_{2},\ldots,\Phi_{b})
\end{align*}
and we define
\begin{equation}\label{RenormalizedEnergy}
    \mathbb{W}(\mathbf{a},\mathbf{c},\mathbf{d}_{1},\mathbf{d}_{2},\Phi)
    \coloneqq\lim_{\sigma\to0^{+}}
    \Biggl\{\frac{1}{2}
    \int_{\Omega_{\sigma}}\!{}
    |\nabla{}u_{*}^{T}|^{2}
    -\pi\Bigl[|\mathbf{d}_{1}|^{2}+\frac{1}{2}|\mathbf{d}_{2}|^{2}\Bigr]
    \mathopen{}\left|\log(\sigma)\right|\mathclose{}\Biggr\}
\end{equation}
where
\begin{equation*}
    \Omega_{\sigma}\coloneqq\Omega\setminus
    \biggl[\bigsqcup_{i=1}^{M_{1}}B_{\sigma}(a_{i})\sqcup
    \bigsqcup_{i=0}^{b}\bigsqcup_{j=1}^{M_{2,i}}B_{\sigma}(c_{ij})\biggr]
\end{equation*}
and where
\begin{equation*}
    0<\sigma<\frac{1}{2}\min\Bigl\{\min\limits_{i\neq{}j}\{|a_{i}-a_{j}|\},
    \min\limits_{i_{1}\neq{}i_{2},\,j_{1}\neq{}j_{2}}
    \{|c_{i_{1}j_{1}}-c_{i_{2}j_{2}}|\},
    \min\limits_{
    i,j,k}\{|a_{i}-c_{jk}|\}\Bigr\}.
\end{equation*}
We now compute \eqref{RenormalizedEnergy}.
By the $L^{2}$ orthogonal decomposition
\eqref{Modifiedj}
and
\eqref{CurrentEquation} we have
\begin{align*}
    \frac{1}{2}\int_{\Omega_{\sigma}}\!{}|\nabla{}u_{*}^{T}|^{2}
    &=\frac{1}{2}\int_{\Omega_{\sigma}}\!{}|\bar{j}|^{2}
    +\frac{1}{2}|\Phi|^{2}+o(1)\hspace{36pt}\text{as }\sigma\to0^{+}\\
    &=\frac{1}{2}\int_{\Omega_{\sigma}}\!{}|\nabla\Psi|^{2}
    +\frac{1}{2}|\Phi|^{2}+o(1)\hspace{25pt}\text{as }\sigma\to0^{+}.
\end{align*}
Integrating by parts and using \eqref{ConjugatePDE} we obtain
\begin{align*}
    \frac{1}{2}\int_{\Omega_{\sigma}}\!{}|\nabla\Psi|^{2}
    &=\frac{1}{2}\int_{(\partial\Omega)_{0}\setminus
    \bigsqcup\limits_{j=1}^{M_{2,0}}B_{\sigma}(c_{0j})}\!{}
    \Psi\frac{\partial\Psi}{\partial\mathbf{n}}
    +\frac{1}{2}\sum_{i=1}^{b}
    \int_{(\partial\Omega)_{i}\setminus\bigsqcup\limits_{j=1}^{M_{2,i}}
    B_{\sigma}(c_{ij})}\!{}
    \Psi\frac{\partial\Psi}{\partial\mathbf{n}}\\
    &+\frac{1}{2}\sum_{j=1}^{M_{2,0}}
    \int_{\Omega\cap\partial{}B_{\sigma}(c_{0j})}\!{}
    \Psi\frac{\partial\Psi}{\partial\mathbf{n}}
    +\frac{1}{2}\sum_{i=1}^{b}\sum_{j=1}^{M_{2,i}}
    \int_{\Omega\cap\partial{}B_{\sigma}(c_{ij})}\!{}
    \Psi\frac{\partial\Psi}{\partial\mathbf{n}}\\
    &+\frac{1}{2}\sum_{i=1}^{M_{1}}\int_{\partial{}B_{\sigma}(a_{i})}\!{}
    \Psi\frac{\partial\Psi}{\partial\mathbf{n}}\\
    &=(A)+(B)+(C)+(D)+(E).
\end{align*}
Observe that by \eqref{ConjugatePDE}, the calculation that lead to
\eqref{FiniteLimit}, and the trace theorem we have
\begin{align*}
    (A)&=\frac{1}{2}\int_{(\partial\Omega)_{0}}\!{}
    \widetilde{\kappa}_{0}\Psi+o(1)\hspace{68pt}\text{as }(\sigma\to0^{+}),\\
    (B)&=\frac{1}{2}\sum_{i=1}^{b}\int_{(\partial\Omega)_{i}}\!{}
    \widetilde{\kappa}_{i}\Psi+o(1)\hspace{46pt}\text{as }(\sigma\to0^{+}).
\end{align*}
Next we note that since $\Psi=\Psi_{1}+\Psi_{2}+H_{\Omega}$, where $\Psi_{1}$
is defined in \eqref{InteriorLogs}, $\Psi_{2}$ is defined in
\eqref{BoundaryLogs}, and $H_{\Omega}$ solves \eqref{FurtherReducedProblem},
we have, as $\sigma\to0^{+}$, that
\begin{align*}
    (C)&=-\frac{\pi}{2}\sum_{j=1}^{M_{2,0}}d_{0j}^{2}\log(\sigma)
    -\frac{\pi}{2}\sum_{j=1}^{M_{2,0}}\sum_{i=1}^{M_{1}}d_{i}d_{0j}
    \log(|c_{0j}-a_{i}|)\\
    &-\frac{\pi}{2}\sum_{j=1}^{M_{2,0}}\sum_{\substack{i=0,1,\ldots,b\\
    k=1,2,\ldots,M_{2,i}\\ (i,k)\neq(0,j)}}d_{ik}d_{0j}\log(|c_{0j}-c_{ik}|)
    -\frac{\pi}{2}\sum_{j=1}^{M_{2,0}}d_{0j}H_{\Omega}(c_{0j})+o(1)
    \\
    (D)&=-\frac{\pi}{2}\sum_{i=1}^{b}\sum_{j=1}^{M_{2,i}}d_{ij}^{2}
    \log(\sigma)
    -\frac{\pi}{2}\sum_{i=1}^{b}\sum_{j=1}^{M_{2,i}}\sum_{k=1}^{M_{1}}
    d_{ij}d_{k}\log(|c_{ij}-a_{k}|)
    \\
    &-\frac{\pi}{2}\sum_{i=1}^{b}\sum_{j=1}^{M_{2,i}}
    \sum_{\substack{k=0,1,\ldots,b\\
    l=1,2,\ldots,M_{2,k}\\ (k,l)\neq(i,j)}}d_{ij}d_{kl}\log(|c_{ij}-c_{kl}|)
    -\frac{\pi}{2}\sum_{i=1}^{b}\sum_{j=1}^{M_{2,i}}
    d_{ij}H_{\Omega}(c_{ij})+o(1)\\
    (E)&=-\pi\sum_{i=1}^{M_{1}}d_{i}^{2}\log(\sigma)
    -\pi\sum_{i=1}^{M_{1}}\sum_{\substack{j=1,2,\ldots,M_{1}\\ j\neq{}i}}
    d_{i}d_{j}\log(|a_{i}-a_{j}|)\\
    &-\pi\sum_{i=1}^{M_{1}}\sum_{j=0}^{b}\sum_{k=1}^{M_{2,j}}
    d_{i}d_{jk}\log(|c_{jk}-a_{i}|)
    -\pi\sum_{i=1}^{M_{1}}d_{i}H_{\Omega}(a_{i})+o(1)
\end{align*}
Combining the above computations for $(A)$, $(B)$, $(C)$, $(D)$, and $(E)$
we obtain the following expression for the renormalized energy:
\begin{align*}
    \mathbb{W}(\mathbf{a},\mathbf{b},\mathbf{d}_{1},\mathbf{d}_{2},\Phi)
    &=-\frac{\pi}{2}\sum_{i\neq{}j}d_{i}d_{j}\log(|a_{i}-a_{j}|)
    -\frac{\pi}{4}\sum_{(i_{1},j_{1})\neq(i_{2},j_{2})}d_{i_{1}j_{1}}
    d_{i_{2}j_{2}}\log(|c_{i_{1}j_{1}}-c_{i_{2}j_{2}}|)
    \nonumber\\
    &-\frac{3\pi}{2}\sum_{i=1}^{M_{1}}\sum_{j=0}^{b}\sum_{k=1}^{M_{2,j}}
    d_{i}d_{jk}\log(|a_{i}-c_{jk}|)
    -\pi\sum_{i=1}^{M_{1}}d_{i}H_{\Omega}(a_{i})
    \nonumber\\
    &-\frac{\pi}{2}\sum_{j=0}^{b}\sum_{k=1}^{M_{2,j}}d_{jk}H_{\Omega}(c_{jk})
    +\frac{1}{2}\int_{\partial\Omega}\!{}\widetilde{\kappa}\Psi
    +\frac{1}{2}|\Phi|^{2}.
\end{align*}
\begin{remark}
    In the case that $\Omega$ is simply connected and there are no interior
    vortices we observe that this matches the renormalized energy calculated
    in Proposition $20$ of \cite{IgKur1}.
    In addition, the flux contribution is similar to the one found in
    \cite{IgJ}.
\end{remark}

\subsubsection{Step \texorpdfstring{$3$}{}:}
We first observe that, by approximation, it suffices to demonstrate
the upper bound in the case that $d_{i}=\pm1$ for all $i=1,2,\ldots,M_{1}$
and $d_{jk}$ for all $k=1,2,\ldots,M_{2,j}$ and $j=0,1,\ldots,b$.
Note that the conditions on the degrees of vortices found in 
Theorem\ \ref{ZerothOrderNormal} part \eqref{Compactness}
can be preserved while splitting a vortex of absolute
degree larger than $1$.\\

Next, we set $\rho_{\varepsilon}(s)\coloneqq\min\{\frac{s}{\varepsilon},1\}$
and define a sequence $\{u_{\varepsilon}\}_{\varepsilon\in(0,1]}$ by
\begin{equation*}
    u_{\varepsilon}(x)=
    \biggl(\prod_{i=1}^{M_{1}}\rho_{\varepsilon}(|x-a_{i}|)\biggr)
    \biggl(\prod_{j=0}^{b}\prod_{k=1}^{M_{2,j}}\rho_{\varepsilon}(|x-c_{jk}|)
    \biggr)u_{*}^{T}
\end{equation*}
where $u_{*}^{T}$ is the map defined in \eqref{def:TangentMap}.
Observe that since $u_{\varepsilon}=u_{*}^{T}$ on
$\Omega_{\varepsilon}\coloneqq\Omega\setminus
\biggl(\bigcup\limits_{i=1}^{M_{1}}B_{\varepsilon}(a_{i})\biggr)\cup
\biggl(\bigcup\limits_{j=0}^{b}\bigcup\limits_{k=1}^{M_{2,j}}B_{\varepsilon,+}(c_{jk})\biggr)$
then by \eqref{RenormalizedEnergy}
\begin{align}\label{eq:LeadingOrder}
    E_{\varepsilon}(u_{\varepsilon})
    &=\int_{\Omega\setminus\Omega_{\varepsilon}}\!{}
    e_{\varepsilon}(u_{\varepsilon})
    +\pi\biggl[|\mathbf{d}_{1}|^{2}+\frac{1}{2}|\mathbf{d}_{2}|^{2}\biggr]
    \mathopen{}\left|\log(\varepsilon)\right|\mathclose{}
    +\mathbb{W}(\mathbf{a},\mathbf{c},\mathbf{d}_{1},\mathbf{d}_{2},\Phi)
    +o(1)\nonumber\\
    &=\int_{\Omega\setminus\Omega_{\varepsilon}}\!{}
    e_{\varepsilon}(u_{\varepsilon})
    +\pi\biggl[|\mathbf{d}_{1}|+\frac{1}{2}|\mathbf{d}_{2}|\biggr]
    \mathopen{}\left|\log(\varepsilon)\right|\mathclose{}
    +\mathbb{W}(\mathbf{a},\mathbf{c},\mathbf{d}_{1},\mathbf{d}_{2},\Phi)
    +o(1)
\end{align}
where we have used that $d_{i}=\pm1$ for all $i=1,2,\ldots,M_{1}$ and
$d_{j,k}=\pm1$ for all $k=1,2,\ldots,M_{2,j}$ and $j=0,1,\ldots,b$ to remove
the squares.
Since $\Omega\setminus\Omega_{\varepsilon}
=\Bigl(\bigcup\limits_{i=1}^{M_{1}}B_{\varepsilon}(a_{i})\Bigr)\cup
{\Bigl(\bigcup\limits_{j=0}^{b}\bigcup\limits_{k=1}^{M_{2,j}}B_{\varepsilon,+}(c_{jk})\Bigr)}$
and
\begin{equation*}
    |\rho_{\varepsilon}|=\frac{|x|}{\varepsilon},\hspace{15pt}
    |\nabla\rho_{\varepsilon}|=\frac{1}{\varepsilon}\hspace{15pt}
    \text{on }B_{\varepsilon}(0),
\end{equation*}
then
\begin{equation}\label{eq:SmallRegionCalculation}
    \int_{\Omega\setminus\Omega_{\varepsilon}}\!{}
    e_{\varepsilon}(u_{\varepsilon})=O(1).
\end{equation}
From \eqref{eq:LeadingOrder} and \eqref{eq:SmallRegionCalculation} we conclude
that
\begin{equation*}
    \limsup_{\varepsilon\to0^{+}}\frac{E_{\varepsilon}(u_{\varepsilon})}
    {\mathopen{}\left|\log(\varepsilon)\right|\mathclose{}}
    =\left\|J_{*}\right\|
\end{equation*}
where $J_{*}=\pi\sum\limits_{i=1}^{M_{1}}d_{i}\delta_{a_{i}}+
\frac{\pi}{2}\sum\limits_{j=0}^{b}
\sum\limits_{k=1}^{M_{2,j}}d_{j,k}\delta_{c_{j,k}}$.
Next, observe that since $|u_{\varepsilon}|=1$ on $\Omega_{\varepsilon}$ and
$u_{\varepsilon}$ is smooth here then $Ju_{\varepsilon}=0$ on
$\Omega_{\varepsilon}$.
Notice that for each $i=1,2,\ldots,M_{1}$ we have, for
$\varphi\in{}C^{0,\alpha}(\Omega)$, that
\begin{align*}
    \int_{B_{\varepsilon}(a_{i})}\!{}\varphi(x)\cdot{}Ju_{\varepsilon}(x)
    \mathrm{d}x
    &=\int_{B_{\varepsilon}(a_{i})}\!{}[\varphi(x)-\varphi(a_{i})]
    \cdot{}Ju_{\varepsilon}(x)\mathrm{d}x
    +\varphi(a_{i})\int_{B_{\varepsilon}(a_{i})}\!{}Ju_{\varepsilon}(x)
    \mathrm{d}x\\
    &=\frac{\varphi(a_{i})}{2}\int_{B_{\varepsilon}(a_{i})}\!{}
    \nabla\times{}ju_{\varepsilon}(x)
    \mathrm{d}x
    +O(\varepsilon^{\alpha}
    \mathopen{}\left|\log(\varepsilon)\right|\mathclose{})\\
    &=\frac{\varphi(a_{i})}{2}\int_{\partial{}B_{\varepsilon}(a_{i})}\!{}
    ju_{\varepsilon}\cdot\tau
    +O(\varepsilon^{\alpha}
    \mathopen{}\left|\log(\varepsilon)\right|\mathclose{}).
\end{align*}
Since $|u_{\varepsilon}|=1$ on $\partial{}B_{\varepsilon}(a_{i})$
then from the construction of $u_{*}^{T}$ we have that
\begin{equation*}
    \frac{\varphi(a_{i})}{2}\int_{\partial{}B_{\varepsilon}(a_{i})}\!{}
    ju_{\varepsilon}\cdot\tau
    =\pi{}d_{i}\varphi(a_{i}).
\end{equation*}
A similar calculation holds $c_{jk}$ for $j=0,1,\ldots,b$ and
$k=1,2,\ldots,M_{2,j}$ except there are additional error terms due to the
boundary.

\section*{Appendix}\label{Appendix}
\addcontentsline{toc}{section}{Appendix}
In this Appendix we provide a proof for a relation between the Euler
characteristic of $\Omega$ and the number of boundary components provided
$\partial\Omega$ is $C^{2,1}$.
In addition, we show that a general $C^{2,1}$ domain, with possibly
non-trivial homology, is obtained by excising simply connected sets from
a larger simply connected set (i.e. general domains of this regularity are
obtained by making holes in the domain).
These results will be needed to establish a relationship between the
possible degree configurations and canonical harmonic map to the topological
properties of the domain.\\

In the next lemma we follow the ideas from \cite{BJOS} to show that the first
Betti number of $\overline{\Omega}$ corresponds, for a bounded, connected open
set $\Omega\subseteq\mathbb{R}^{2}$
with $C^{2}$-boundary, to one less than the number of connected components of
its boundary.
In addition, we demonstrate that the second Betti number of $\overline{\Omega}$ is zero.
\begin{lemmaA}\label{BettiNumberLemma}
    Suppose $\Omega\subseteq\mathbb{R}^{2}$ is an open, bounded, connected set
    with $C^{2,1}$-boundary.
    Then if we let $b+1$, where $b\ge0$, denote the number of connected
    components of $\partial\Omega$ then we have
    \begin{equation*}
        \beta_{1}(\overline{\Omega})=b,\hspace{20pt}
        \beta_{2}(\overline{\Omega})=0
    \end{equation*}
    where $\beta_{i}(\overline{\Omega})$ denotes the $i^{\text{th}}$
    Betti number of $\overline{\Omega}$.
\end{lemmaA}

\begin{proof}
    Since $\overline{\Omega}$ is a compact orientable $2$-manifold with
    boundary then by Lefschetz Duality, see Theorem $3.43$ of \cite{Hat} with
    $A=\partial{}\Omega$ and $B=\varnothing$, we have
    \begin{equation*}
        H_{1}(\overline{\Omega})\cong{}
        H^{1}(\overline{\Omega},\partial{}\Omega).
    \end{equation*}
    It is shown in Lemma $13$ of \cite{BJOS} that
    \begin{equation*}
        H^{1}(\overline{\Omega},\partial\Omega)\cong{}\mathbb{Z}^{b}.
    \end{equation*}
    Finally, by definition, we have that
    \begin{equation*}
        \beta_{1}(\overline{\Omega})=\text{rank}(H_{1}(\overline{\Omega}))
        =\text{rank}(\mathbb{Z}^{b})=b.
    \end{equation*}
    By another application of Lefschetz Duality with $A=\partial\Omega$ and
    $B=\varnothing$ we have
    \begin{equation*}
        H_{2}(\overline{\Omega})\cong{}H^{0}(\overline{\Omega},\partial\Omega)
        \cong{}\text{Hom}(H_{0}(\overline{\Omega},\partial\Omega),\mathbb{Z}).
    \end{equation*}
    Notice that $H_{0}(\overline{\Omega},\partial\Omega)=0$ since
    $\overline{\Omega}$ is a connected two dimensional manifold with
    boundary, and hence path-connected, then any
    $0$-chain in $\overline{\Omega}$ differs from a $0$-chain in $\partial\Omega$
    by a boundary.
    We conclude that
    \begin{equation*}
        \text{Hom}(H_{0}(\overline{\Omega},\partial\Omega),\mathbb{Z})
        \cong\text{Hom}(0,\mathbb{Z})\cong0.
    \end{equation*}
\end{proof}

\begin{lemmaA}\label{HomolgyInterior}
    Suppose $\Omega\subseteq\mathbb{R}^{2}$ is an open, bounded, connected
    subset with $C^{2,1}$-boundary.
    Then
    \begin{equation}
        H_{i}(\Omega)=H_{i}(\overline{\Omega}),\hspace{15pt}\text{for }i=0,1,2.
    \end{equation}
\end{lemmaA}

\begin{proof}
This result follows from the fact that the maps
$f\colon\Omega\to\overline{\Omega}$ and $g\colon\overline{\Omega}\to\Omega$
defined by
\begin{equation*}
    f=\iota_{\Omega},\hspace{10pt}
    g(x)=
    \begin{cases}
        x,&\text{dist}(x,\partial\Omega)>r_{1}\\
        x+r_{1}\nu_{i}((\psi_{i,j}^{-1})^{1}(x)),&\text{if }
        x\in\overline{\mathcal{U}}_{i,j}
    \end{cases}
\end{equation*}
form a homotopy equivalence for $\Omega$ and $\overline{\Omega}$.
\end{proof}

\begin{corollaryA}\label{EulerCharacteristic}
    Suppose $\Omega\subseteq\mathbb{R}^{2}$ is an open, bounded, connected
    subset with $C^{2,1}$-boundary.
    Then
    \begin{equation}
        \chi_{Euler}(\Omega)=\chi_{Euler}(\overline{\Omega})=1-b.
    \end{equation}
\end{corollaryA}

\begin{proof}
    This follows from Lemmas \ref{BettiNumberLemma} and \ref{HomolgyInterior}.
\end{proof}

\begin{lemmaA}\label{FilledInDomain}
    Suppose $\Omega\subseteq\mathbb{R}^{2}$ is an open, bounded, connected set with
    $C^{2,1}$-boundary.
    Suppose also that $\partial\Omega$ has $b+1$, where $b\ge0$, connected components.
    Then we have
    \begin{equation}\label{ExteriorHomology}
        H_{0}(\mathbb{R}^{2}\setminus\Omega)\cong{}\mathbb{Z}^{b+1},
        \hspace{15pt}
        H_{1}(\mathbb{R}^{2}\setminus\Omega)\cong{}\mathbb{Z}
    \end{equation}
    and hence $\mathbb{R}^{2}\setminus\Omega$ consists of $b+1$
    connected components and all, except one, are simply connected.
    In addition, if
    $\mathbb{R}^{2}\setminus\overline{\Omega}
    =B_{0}\sqcup\bigsqcup\limits_{i=1}^{b}B_{i}$
    where $B_{0}$ is the unbounded component of
    $\mathbb{R}^{2}\setminus\overline{\Omega}$ and $B_{i}$ for $1\le{}i\le{}b$
    denote the bounded components then
    \begin{equation}\label{FilledInHomology}
        H_{1}\biggl(\Omega\cup\bigcup_{i=1}^{b}\overline{B}_{i}\biggr)=0
    \end{equation}
    and hence $\Omega$ is obtained from a simply connected set by
    removing $b$ simply connected open sets with $C^{2}$-boundary from it.
\end{lemmaA}

\begin{proof}
By the Classification of One-Manifolds, see page $64$ of \cite{GuPo}, each
connected component of $\partial\Omega$ is diffeomorphic to $\mathbb{S}^{1}$.
Connectedness of $\Omega$ combined with the Jordan-Schonenflies theorem, see
Theorem $4$ on page $72$ of \cite{Moi}, gives \eqref{ExteriorHomology}.
\eqref{FilledInHomology} follows from Lemma \ref{BettiNumberLemma}.
\end{proof}

\bibliographystyle{acm}
\bibliography{Bibliography}

\begin{thebibliography}{10}

\bibitem{ABM}
{\sc Alama, S., Bronsard, L., and Millot, V.}
\newblock {$\Gamma$}-convergence of 2{D} {G}inzburg-{L}andau functionals with
  vortex concentration along curves.
\newblock {\em J. Anal. Math. 114\/} (2011), 341--391.

\bibitem{ABvB}
{\sc Alama, S., Bronsard, L., and van Brussel, L.}
\newblock On minimizers of the 2d {G}inzburg-{L}andau energy with tangential
  anchoring.
\newblock \url{https://arxiv.org/pdf/2301.05361.pdf}.

\bibitem{AlBaOr}
{\sc Alberti, G., Baldo, S., and Orlandi, G.}
\newblock Variational convergence for functionals of {G}inzburg-{L}andau type.
\newblock {\em Indiana Univ. Math. J. 54}, 5 (2005), 1411--1472.

\bibitem{AlPo}
{\sc Alicandro, R., and Ponsiglione, M.}
\newblock Ginzburg-{L}andau functionals and renormalized energy: a revised
  {$\Gamma$}-convergence approach.
\newblock {\em J. Funct. Anal. 266}, 8 (2014), 4890--4907.

\bibitem{BJOS}
{\sc Baldo, S., Jerrard, R.~L., Orlandi, G., and Soner, H.~M.}
\newblock Convergence of {G}inzburg-{L}andau functionals in three-dimensional
  superconductivity.
\newblock {\em Arch. Ration. Mech. Anal. 205}, 3 (2012), 699--752.

\bibitem{BaTi}
{\sc Bardos, C.~W., and Titi, E.~S.}
\newblock {$C^{0,\alpha}$} boundary regularity for the pressure in weak
  solutions of the 2d {E}uler equations.
\newblock {\em Philos. Trans. Roy. Soc. A 380}, 2218 (2022), Paper No.
  20210073, 15.

\bibitem{Ba}
{\sc Bates, S.~M.}
\newblock Toward a precise smoothness hypothesis in {S}ard's theorem.
\newblock {\em Proc. Amer. Math. Soc. 117}, 1 (1993), 279--283.

\bibitem{BBH}
{\sc Bethuel, F., Brezis, H., and H\'{e}lein, F.}
\newblock {\em Ginzburg-{L}andau vortices}, vol.~13 of {\em Progress in
  Nonlinear Differential Equations and their Applications}.
\newblock Birkh\"{a}user Boston, Inc., Boston, MA, 1994.

\bibitem{CJ}
{\sc Contreras, A., and Jerrard, R.~L.}
\newblock Nearly parallel vortex filaments in the 3{D} {G}inzburg-{L}andau
  equations.
\newblock {\em Geom. Funct. Anal. 27}, 5 (2017), 1161--1230.

\bibitem{DoC}
{\sc do~Carmo, M.~P.}
\newblock {\em Differential geometry of curves \& surfaces}.
\newblock Dover Publications, Inc., Mineola, NY, 2016.
\newblock Revised \& updated second edition of [ MR0394451].

\bibitem{Go}
{\sc Gordon, W.~B.}
\newblock On the diffeomorphisms of {E}uclidean space.
\newblock {\em Amer. Math. Monthly 79\/} (1972), 755--759.

\bibitem{GuPo}
{\sc Guillemin, V., and Pollack, A.}
\newblock Differential topology, 2010.

\bibitem{Hat}
{\sc Hatcher, A.}
\newblock {\em Algebraic topology}.
\newblock Cambridge University Press, Cambridge, 2002.

\bibitem{IgJ}
{\sc Ignat, R., and Jerrard, R.~L.}
\newblock Renormalized energy between vortices in some {G}inzburg-{L}andau
  models on 2-dimensional {R}iemannian manifolds.
\newblock {\em Arch. Ration. Mech. Anal. 239}, 3 (2021), 1577--1666.

\bibitem{IgKur1}
{\sc Ignat, R., and Kurzke, M.}
\newblock An effective model for boundary vortices in thin-film micromagnetics.
\newblock \url{https://arxiv.org/pdf/2202.02778.pdf}.

\bibitem{RaIg}
{\sc Ignat, R., and Otto, F.}
\newblock A compactness result for {L}andau state in thin-film micromagnetics.
\newblock {\em Ann. Inst. H. Poincar\'{e} C Anal. Non Lin\'{e}aire 28}, 2
  (2011), 247--282.

\bibitem{JMS}
{\sc Jerrard, R., Montero, A., and Sternberg, P.}
\newblock Local minimizers of the {G}inzburg-{L}andau energy with magnetic
  field in three dimensions.
\newblock {\em Comm. Math. Phys. 249}, 3 (2004), 549--577.

\bibitem{JS}
{\sc Jerrard, R.~L., and Soner, H.~M.}
\newblock The {J}acobian and the {G}inzburg-{L}andau energy.
\newblock {\em Calc. Var. Partial Differential Equations 14}, 2 (2002),
  151--191.

\bibitem{KrPa}
{\sc Krantz, S.~G., and Parks, H.~R.}
\newblock {\em The implicit function theorem}.
\newblock Modern Birkh\"{a}user Classics. Birkh\"{a}user/Springer, New York,
  2013.
\newblock History, theory, and applications, Reprint of the 2003 edition.

\bibitem{Moi}
{\sc Moise, E.~E.}
\newblock {\em Geometric topology in dimensions {$2$} and {$3$}}.
\newblock Graduate Texts in Mathematics, Vol. 47. Springer-Verlag, New
  York-Heidelberg, 1977.

\bibitem{Pr}
{\sc Pressley, A.}
\newblock {\em Elementary differential geometry}, second~ed.
\newblock Springer Undergraduate Mathematics Series. Springer-Verlag London,
  Ltd., London, 2010.

\bibitem{VL}
{\sc Volovik, G.~E., and Lavrentovich, O.~D.}
\newblock Topological dynamics of defects: Boojums in nematic drops.
\newblock {\em Journal of Experimental and Theoretical Physics 58\/} (1983),
  1159--1167.

\end{thebibliography}

\end{document}